\documentclass[reqno,10pt]{amsart}
\usepackage{svg}
\usepackage{amsmath}
\usepackage{amssymb}
\usepackage{amsfonts}
\usepackage{mathtools}
\usepackage{comment}
\usepackage{graphicx}
\usepackage[abs]{overpic}
\usepackage{caption}
\usepackage{subcaption}
\usepackage{wrapfig}
\usepackage{float}
\usepackage{graphicx}
\usepackage{physics}
\usepackage{esint} % ('\fint' command)
\usepackage{braket} % ('\set' command)
\usepackage[colorinlistoftodos]{todonotes}
\usepackage{enumitem}

\newcommand{\mres}{\mathbin{\vrule height 1.6ex depth 0pt width 0.13ex\vrule height 0.13ex depth 0pt width 0.8ex}}
\newcommand{\Wtrunc}{V}
\newcommand{\WtruncLI}{V_0}
\newcommand{\Wtruncc}{\sqrt{\mathfrak{W}}}
\newcommand{\Wtrunccc}{\mathfrak{W}}

\usepackage{tikz}

\definecolor{blue_links}{RGB}{13,0,180} 

\usepackage{hyperref}
\hypersetup{
    colorlinks=true, %set true if you want colored links
    linktoc=all,     %set to all if you want both sections and subsections linked
    linkcolor=blue_links,  %choose some color if you want links to stand out
    citecolor=blue_links,
    urlcolor=blue_links,
}
\usepackage{mathtools}
\usepackage{xcolor}
\usepackage{natbib}

\newtheorem{theorem}{Theorem}[section]
\newtheorem{lemma}[theorem]{Lemma}
\newtheorem{proposition}[theorem]{Proposition}

\newtheorem{definition}[theorem]{Definition}

\newtheorem{remark}[theorem]{Remark}

\newtheorem*{theorem*}{Theorem}

\newcommand{\N}{\mathbb{N}}

\newcommand{\R}{\mathbb{R}}

\newcommand\ep{\varepsilon}

\def\eps{\varepsilon}

\def\weaklystar{\stackrel{*}{\rightharpoonup}}
\def\dist{\operatorname{dist}}

\def\XXint#1#2#3{{\setbox0=\hbox{$#1{#2#3}{\int}$}
\vcenter{\hbox{$#2#3$}}\kern-.5\wd0}}

\usepackage{color}

\numberwithin{equation}{section}

\newcommand{\Landau}{\mathcal{O}}
\renewcommand{\H}{\mathcal{H}}

\renewcommand{\epsilon}{\varepsilon}
\begin{document} 

\title[Solid-solid phase transitions  with   space-dependent wells]{Solid-solid phase transitions  with  space-dependent wells}
\author[E. Davoli]{Elisa Davoli} 
\address[Elisa Davoli]{Institute of Analysis and Scientific Computing, TU Wien, Wiedner Hauptstra\ss e 8-10, 1040 Vienna, Austria}
\email{elisa.davoli@asc.tuwien.ac.at}
\author[J. Deutsch]{Jakob Deutsch} 
\address[Jakob Deutsch]{Institute of Analysis and Scientific Computing, TU Wien, Wiedner Hauptstra\ss e 8-10, 1040 Vienna, Austria}
\email{jakob.deutsch@asc.tuwien.ac.at}
\author[M. Friedrich]{Manuel Friedrich} 
\address[Manuel Friedrich]{Institute of Analysis, Johannes Kepler University, Altenberger Stra\ss e 69, 4040 Linz, Austria.}
\email{manuel.friedrich@jku.at}
\author[K. Stinson] {Kerrek Stinson} 
\address[Kerrek Stinson]{Department of Mathematics, University of Utah, 155 South 1400 East, JWB 233
Salt Lake City, UT 84112, USA.}
\email{kerrek.stinson@math.utah.edu}

\subjclass[2010]{49J45, 70G75, 74B20, 74N15 , 74G65}
\keywords{Solid-solid phase transitions, space-dependent wells, sharp-interface limit, Gamma-convergence}

 %\addbibresource{references.bib}

\begin{abstract}  
We investigate a  second-order  Modica-Mortola functional  of the form  
\[
    E_\eps[u] := \int_\Omega \frac{1}{\eps} W(x, \nabla u) + \eps|\nabla^2 u|^2 \, dx,
\]
which models solid--solid phase transitions in heterogeneous media. We neglect the assumption of frame indifference  in the elastic energy density $W$   while allowing for   space-dependent, pointwise-compatible wells. Under suitable assumptions on  $W$ and regularity conditions on the associated rank-one connection, we prove that $E_\varepsilon$ $\Gamma$-converges to a local interfacial energy  defined for suitable laminate-type configurations.   \\
\end{abstract}

\maketitle
%\tableofcontents

\section{Introduction}\label{intro}
The theory of solid-solid phase transitions is a  broad field of research, explaining a wide range of  phenomena
 in areas such as
shape-memory alloys,  laminated structures in applied chemistry, or  lithium-ion evolution in batteries.    The   past few years have witnessed the emergence of new fields of applications such as the modeling of compositionally graded materials, or the manufacturing of programmable metamaterials (see, e.g., \cite{Zhang,Meng}), in which elastic phases vary with position.    While   the variational analysis of solid-solid phase transitions with \emph{fixed wells} has been by now extensively developed,  settings with \emph{space-dependent wells} remain largely unexplored.  

A particularly natural instance of a phenomenon  involving space-dependent wells   arises in the study of prestrained materials coupled with solid-solid phase transitions (see, e.g., \cite{neukamm} and the references therein for an overview of recent mathematical developments on prestrained elasticity). Prestrain may originate from a variety of physical mechanisms and often gives rise to novel functionalities \cite{Ge}, allowing, for example, controlled shape transformations. From a mathematical perspective, if the admissible elastic phases are represented by two fixed wells $E_1$ and $E_2$, then, in the absence of prestrain,  ground-state  deformations would satisfy the differential inclusion
\begin{equation*}
    \nabla y(x)\in\{E_1,E_2\}.
\end{equation*}
If the material is prestrained by a  matrix   field $U$, however, its elastic response is measured  relative  to the local prestrain, and the deformation $y$ must  satisfy  
\begin{equation*}
    \nabla y(x)\in\{E_1U(x),E_2U(x)\}.
\end{equation*}
In other words, even when the underlying phases are fixed, prestrain naturally transforms a classical solid-solid phase  transition problem into one involving \emph{space-dependent wells}.

The purpose of this paper is to take a first step in the rigorous variational analysis of solid-solid phase transitions with space-dependent wells.  More  precisely, for a bounded domain $\Omega\subset \R^d$ and  deformations $u \in H^2(\Omega;\R^d)$, we consider a free energy of the form
\begin{equation}
\label{eq:prototype}
 \frac{1}{\eps}\int_{\Omega}W(x,\nabla u)\, dx + \eps \int_{\Omega}|\nabla^2 u|^2 \, dx,
\end{equation}
where $W(x,\cdot)$ vanishes only if $\nabla u(x) \in \{ A(x),B(x)\}$, $A$ and $B$ are sufficiently regular functions  with values in $\R^{d \times d}$ representing the \emph{wells}, and $\eps$ denotes the length scale of the transition layer between the phases. From a modeling point of view, $A$ and $B$ can be seen as locally stress-free configurations of a crystalline body that is free to switch between different target strains. In \eqref{eq:prototype}, the first term corresponds to   a nonlinear elastic contribution, neglecting frame indifference  for  simplicity. Roughly speaking, it  forces the gradient of energetically favorable deformations to lie  pointwise close to one of the prescribed wells $A$ and $B$.   By itself,  this elastic term would lead to unphysical effects, for highly oscillatory sequences might exhibit very low energy. The role of the second contribution in \eqref{eq:prototype} is to prevent this phenomenon, by introducing a competing effect  through  a higher-order surface penalization that  excludes too many phase changes. Whereas the setting of variable wells in models for phase transitions has attained considerable attention in the liquid-liquid case starting from \cite{bouchitte, sternberg}, this paper constitutes a first step towards a theory in the setting of solid-solid phase transformations.  Our  main results    are a description  of the asymptotic behavior of  \eqref{eq:prototype} in the sense of $\Gamma$-convergence, as $\eps \to 0$, along with a characterization of the  limit   energy's  domain consisting of suitable laminate-type configurations.  

 Before we  describe the contents   of the paper, let us briefly review the literature on variational theories for solid-solid phase transitions with fixed wells. In  that  setting, stress-free solutions for energy functionals with multiwell nonconvex densities were first characterized in the seminal work of Ball and James \cite{ball-james1987} and, accounting for frame indifference, by Dolzmann and M\"uller \cite{dolzmann-mueller95} (see also \cite{kirchheim98}). 
   Sharp-interface limits of solid-solid phase field models were first studied by Conti, Fonseca, and Leoni in \cite{contifonsecaleoni2002}, in the absence of frame indifference. The aforementioned result was recently extended in \cite{deutsch25}, where the limiting interfacial energy was characterized in the planar setting via periodic recovery sequences and  without  symmetry assumptions on the potential. In two-dimensions, Conti and Schweizer tackled the frame-indifferent  case  in the series of papers \cite{contischweizerLin06,contischweizer06,contischweizerImp06}. Later,  \cite{contischweizer06} was extended to higher dimensions by including an anisotropic penalization \cite{davolifriedrich20}. Also the passage from nonlinear-to-linearized models has been analyzed in this  setting \cite{davolifriedrich25}. An extension to the modeling of surfactants was carried out in \cite{cicaleseheilmann25}, and a nonlocal variant was analyzed in \cite{dalmaso18}.
A  related   solid-solid phase transition model coupling chemistry and elastic misfit in lithium-ion batteries was studied in \cite{stinson21}. For some related results on microscopic modeling of martensitic transformations we refer to \cite{kitavsev.luckhaus.rueland15, kitavtsev.luckhaus.rueland17}.

In the setting of space-dependent wells, lower-order  liquid-liquid  models   
were first studied  in the scalar case in \cite{bouchitte, sternberg, ZunigaSternberg2016},  and recently in the vectorial setting in \cite{CristoferiGravina2021} (see also \cite{davolitasso25} for a nonlocal counterpart, as well as \cite{CrisFonGa23, CrisFonGa25, CrisFon19, CrisFonGa22} for a homogenization analysis). 
Here, the energies are of Modica--Mortola-type for a vectorial parameter $c\colon \Omega\to \R^d$ (possibly describing a multi-species concentration)  and take the form    
\begin{equation}\label{eqn:heteroMM}
\frac{1}{\eps} \int_\Omega W(x,c)\,dx + \eps \int_\Omega|\nabla c|^2\,dx, 
\end{equation}
where $W(x,\cdot)$ vanishes only if $c(x) \in \{a_1(x), \ldots, a_n(x)\}$, for suitable vector-valued functions $a_1,\dots,a_n$.

The results in \cite{CristoferiGravina2021} form the starting point of our analysis. Let us  however  stress that the setting of solid-solid phase transformations differs substantially from the study of classical first-order Modica--Mortola functionals, as in the present setting  all  constructions must   be \emph{curl-free}. In the vectorial context of \eqref{eqn:heteroMM}, this amounts to the requirement that $c$ satisfies ${\rm curl}\, c = 0$, i.e., that $c = \nabla u $ for some function $u$. In the limit $\eps \to 0$, this forces finite-energy sequences to converge to laminate-type sets of finite perimeter.
Furthermore, in the  liquid-liquid  setting of \cite{CristoferiGravina2021}, to construct recovery sequences, one may use the fact that the energetically optimal phase transitions are  \emph{one-dimensional geodesics}. In contrast, for  solid-solid phase transitions,  cell energies can be  genuinely  multidimensional, cf.\ \cite[Section 8]{contifonsecaleoni2002}. To render   the   analysis of the heterogeneous curl-free setting tractable, we will impose structural assumptions on the density $W$, locally enforcing one-dimensionality of the optimal profiles.  
 Analogous  to the work \cite{contifonsecaleoni2002}, we will  neglect frame indifference for simplicity.

We now describe our results in more detail.   
Under standard quadratic growth and well-separatedness assumptions (see Subsection \ref{subs:ass}), we consider the second-order heterogeneous Modica-Mortola functionals in \eqref{eq:prototype}. We assume the two wells $A$ and $B$ to be the gradients of   corresponding $C^2$-functions and to be  pointwise  rank-one connected. From a mathematical point of view, this   ensures   both admissibility of pure phases at the diffuse-interface level, and the formation of phase transitions in the sharp-interface limit. Our  results are  threefold:

First,  in   Theorem \ref{mainthm:compactness} we prove that sequences with equibounded energies \eqref{eq:prototype} are precompact in $H^1$, and that any  limiting  function has a gradient with bounded variation  which  locally  takes values in the    two phases.  To be precise, we show that limiting deformations belong to the space \begin{equation*}
    X(\Omega) : = \big\{u\in H^1(\Omega;\R^d) \colon \,  \nabla u \in BV(\Omega;\{A(\cdot),B(\cdot)\})\big\},
\end{equation*}
where 
\begin{equation*}
    BV(\Omega;\{A(\cdot),B(\cdot)\}) : = \big\{M \in BV(\Omega;\R^{d\times d}) \colon \, M(x) \in \{ A(x),B(x)\} \text{ for a.e.\ }x\in \Omega\big\}.
\end{equation*}
The proof of this result relies on the combination of a classical Young-measure approach with an estimate on the gradient of suitable truncations of a path-dependent geodesic distance.  The  latter quantity measures, roughly speaking, how much energy is spent moving between any two matrices along a prescribed path in $\R^{d\times d}$.

Our second contribution is in the spirit of the seminal rigidity analysis in \cite{ball-james1987}. After showing in Lemma~\ref{lem:rep} that the assumptions on the two wells yield a local representation of $A-B$ as   the outer product of two vectors, one of which is the gradient of a   suitable  $C^2$-function    $\gamma$, we  connect the geometry of the jump set of the limiting gradient  to the level sets of $\gamma$.  Indeed, by  a careful application of the Constant Rank Theorem, we show in Theorem \ref{thm:characterisation-of-piecewise-BV-space} that the set $\{x \in \Omega : \nabla u(x) = A(x)\}$ has finite perimeter, that the normal to the jump set of the limiting gradient is determined by the rank-one connection,  and that the jump set can be locally characterized as the union of level sets of $\gamma$.

Our third and last main result, see Theorem \ref{thm:main}, is a $\Gamma$-convergence analysis for the energies in \eqref{eq:prototype} in the sharp-interface limit, as the parameter $\eps$ tends to zero. First, a generic lower bound for the asymptotic behavior of the functionals in \eqref{eq:prototype} is identified by a classical blow-up  approach \cite{FonsecaMueller1992}.    A priori, the  resulting    cell energy (see Definition \ref{def:nonlocal-cell}) is multidimensional.  Then, by  employing slicing techniques combined with suitable Lipschitz estimates  
 for the space-dependent 
geodesic distance (see Section \ref{sec:geodesics-disc}), we can always recover a (possibly suboptimal) lower bound in terms of a geodesic cell energy. To recover the geodesic lower bound in the proof of the $\Gamma$-$\limsup$ inequality, we  impose an assumption on the potential $W$ guaranteeing  one-dimensional optimal profiles  (see \eqref{eqn:W1Dstructure}),   and we  postulate the existence of a globally defined function $\gamma$ which describes the rank-one connection and has level sets that  can be flattened   via a $C^2$-diffeomorphism. 
This diffeomorphism can be extended to the ambient Euclidean space and allows us to reduce the construction of the recovery sequence to  the case of   hyperplanes. The construction of the recovery sequence itself is  an inherently local problem, in the sense that we glue together a selection of one-dimensional optimal profiles identified first on cylinders contained in the diffeomorphed domain. The construction for the original domain is then achieved by relying on a series of auxiliary lemmas (cf.\ Subsection \ref{sec71}) allowing to `pull back' the `flattened' recovery sequence to the reference configuration $\Omega$.

 There are two primary difficulties in our analysis. The first is localizing the cell-energy coming from the $\liminf$ argument. For this, we generalize the arguments of \cite{CristoferiGravina2021, CrisFonGa23,ZunigaSternberg2016} to show that there exist geodesics with uniformly bounded path length. In particular, our analysis applies to path-dependent geodesic distances that are not invariant under reparametrization.
The second difficulty is constructing the locally optimal recovery sequence. To construct a recovery sequence for a given interface in $\Omega$, existing approaches \cite{contischweizer06,davolifriedrich20} embed the flat interface into a larger auxiliary rectangular interface and extract an optimal sequence from  the cell-energy formula for this simpler geometry. Using the spatial-homogeneity of the potential $W$ in these settings, standard arguments show that this sequence on the auxiliary rectangle is optimal on all subdomains, and in particular it defines an appropriate recovery sequence when restricted to the interface contained in $\Omega$. Conti, Fonseca, and Leoni \cite{contifonsecaleoni2002} adopt a similar approach, but show how these ideas can be extended to more general geometries. In contrast, our approach uses the cell-energy in a much more local fashion to construct energetically optimal sequences. 
With this strategy, we can directly treat complex geometries for $\Omega$ (cf. \cite{contifonsecaleoni2002}).
Moreover, we believe that this localization argument for recovery sequences may be extended beyond the setting of one-dimensional optimal profiles, but, due to technical obstacles coming from the interaction of symmetry conditions and diffeomorphisms, we do not pursue this here.

The structure of the paper is the following. Our main results and our precise assumptions are detailed in Section \ref{sec:main}. Limit laminate-type configurations are characterized in Section \ref{sec:limit-laminates}. In Section \ref{sec:compactness}, we carry out our compactness analysis. The $\Gamma$-convergence result, cf.\ Theorem \ref{thm:main}, is proven in Sections \ref{sec:liminf-interm} and \ref{sec:limsup}. Finally, Section \ref{sec:geodesics-disc} collects a few technical results about geodesic curves.

We close the introduction with some notation that will be used throughout the paper. We let $\Omega\subset \R^d$ be an open,  bounded set with Lipschitz boundary.
We use the notation  $(x',x_d)$ with $x' \in \R^{d-1}$ and $x_d \in \R$ for elements $x \in \R^d$. Given a function $u\colon \Omega \to \R^d$, we write $\nabla'$ for the derivative with respect to the first $(d-1)$ variables $x'$. The second derivative corresponds to a third-order tensor, and we use the notation $[\nabla^2 u]_{ijk} = \partial_j \partial_k u_i$ to refer to its components.

For a set $\mathcal{O}\subset \Omega$, we denote by $\partial^*\mathcal{O}$   the reduced boundary. We let $\nu_\mathcal{O}$ be the measure theoretic inner normal on $\partial^*\mathcal{O}$. Moreover, $\chi_A$ is the characteristic function of a set $A$, meaning that
$$\chi_A=\begin{cases}1&\text{if }x\in A\\
0&\text{otherwise}.\end{cases}$$
We will adopt classical notation for Sobolev spaces, bounded Radon measures, and functions of bounded variation (see \cite{ambrosio2000fbv}). For $u\in BV(\Omega; \R^m )$, we define one-sided limits on a rectifiable surface $M$ with orientation at a point $x\in M$ specified by the normal $\nu_M(x)$ as $u^+(x)$ and $u^-(x)$; here, $u^+$ is the limit from the direction $\nu_M(x)$ and $u^-$ the limit from the direction $-\nu_M(x).$ These limits are well defined at $\mathcal{H}^{d-1}$-a.e.\ point $x\in M$.  
We  write $\eps\to 0$ in place of a generic sequence $\eps_n\to 0+$. Throughout the work, $C>0$  denotes a generic constant whose value may change from line to line within the same equation. 

For every $x_0\in \R^d$ and $\nu\in \mathbb{S}^{d-1}$, let $Q_\nu(x_0,r)$ be the cube centered at $x_0$ with sides of length $r$, and such that two faces are orthogonal to $\nu$. We also define the half cubes 
$Q^\pm_{\nu(x_0)}(x_0, r ) = Q_{\nu(x_0)}(x_0, r)\cap \{x\colon \pm (x-x_0)\cdot \nu > 0\}$, where $\pm$ stands as placeholder for $+$ and $-$.

\section{Setting and main results}
\label{sec:main}
In this section, we present the model and our main results.
\subsection{Assumptions}
\label{subs:ass}
Let $\Omega\subset \R^d$ be an open, bounded set with Lipschitz boundary. For every $U \subset \Omega$ and every $\eps>0$, we define the energy
\begin{equation}\label{eqn:energy}
E_\eps [u, U]: = \int_{U} \left(\frac{1}{\eps}W(x,\nabla u) + \eps |\nabla^2 u |^2 \right) dx \quad \text{for } u\in H^2(\Omega;\R^d),
\end{equation}
where $W \colon \overline{\Omega} \times \R^{d \times d} \to \R$ is a nonnegative, continuous potential. For brevity, we write $E_\eps [u]$ instead of $E_\eps [u, \Omega]$. We assume that for any $x \in \overline \Omega$ we have
\[
    \big\{ M \in \R^{d \times d} : W(x,M) = 0 \big\} = \{ A(x), B(x) \}
\]
for two functions $A\colon \overline \Omega \to \R^{d \times d}$ and $B \colon \overline \Omega \to \R^{d \times d}$ which are called the \emph{wells} of the potential. 
We require the following assumptions on the potential $W$ and its wells $A, B$: 
\begin{enumerate}[label=(H\arabic*), ref={\rm (H\arabic*)}]
    \item \label{H1}\textit{(Structure of wells)} There exist $\alpha,\beta \in C^2(\overline{\Omega};\R^d)$ such that 
    $$
        A  = \nabla \alpha \qquad \text{ and } \qquad B = \nabla \beta.
    $$
    The wells are rank-one connected, i.e.,  
    \[
        {\rm rank}\, (A(x) - B(x)) = 1 \qquad  \text{for all $x \in \overline{\Omega}$,} 
    \]
    and well-separated in $\overline \Omega$, meaning
    \begin{equation}
        \label{eq:delta-well}
        \min_{x\in \overline{\Omega}}|A(x)- B(x)| > 0. 
    \end{equation}

    \item \label{H2} \textit{(Quadratic growth)} We assume that there exists $C_* > 0$ such that
    \begin{equation*}
    %\label{eqn:all W conditions}
        \frac{1}{C_*}\min\big\{|M-A(x)|^2,|M-B(x)|^2\big\}\leq W(x,M) \leq C_*\min\big\{|M-A(x)|^2,|M-B(x)|^2\big\}
    \end{equation*}
    for all $x \in \overline \Omega$ and $M\in \R^{d\times d}$. In particular,  for every $x\in \overline{\Omega}$ it holds that $W(x,M) = 0$ if and only if $M = A(x)$ or $M = B(x)$. 
    \item \label{H3} \textit{(Regularity of $\sqrt W$)} For every $R>0$ there exists $C_{\rm Lip}(R)>0$ such that 
    $$|\sqrt{W(x,M)}-\sqrt{W(y,M)}|\leq C_{\rm Lip}(R)|x-y|$$ for every $M\in B_R(0)$ and every $x,y\in \Omega$.
\end{enumerate}
Condition \ref{H1} is a natural assumption: The wells being gradients allows pure phases to be admissible in the model. The rank-one connectedness of the gradients allows for phase transitions in the sharp-interface limit. The well-separability condition guarantees that the wells do not coincide, and that the gradients of functions in the limit space are of $BV$-regularity (cf.\ \cite{CristoferiGravina2021} for an analogous condition). 
Since $A,B \in L^\infty(\Omega; \R^{d \times d})$, we infer that condition \ref{H2} is equivalent to $W$ being comparable to $\min\{|M-A(x)|^2,|M-B(x)|^2\}$ in a small neighborhood of the wells and having growth comparable to $|M|^2$ as $|M|\to \infty.$
Assumption \ref{H3} is satisfied by the prototypical double-well function $\min\{|M-A(x)|^2,|M-B(x)|^2\}$, as well as by maps fulfilling the differential constraint
\begin{equation}\label{eqn:dampened osc}
|\nabla_x W(x,M)|\leq C\sqrt{W(x,M)} \quad \text{for some }C>0,\,\text{for every }x\in\Omega\text{ and }M\in \R^{d\times d}, 
\end{equation}
as can be directly checked via the chain rule. Similar but stronger differential assumptions have been used to study parabolic-phase field models with moving wells \cite{GanediMarveggioStinson}.

While assumptions in the spirit of \ref{H1}--\ref{H3}   are standard in the theory of phase transitions and are sufficient for compactness and the lower bound of the $\Gamma$-limit, we impose a stronger condition on $W$ for the construction of  recovery sequences. To motivate this additional hypothesis, we show in Lemma~\ref{lem:rep}  below that the assumption on rank-one connectedness, along with the fact that the wells $A$ and $B$ are gradient fields, i.e., \ref{H1}, yields that $A-B$ can be locally represented by a   outer product of vectors involving a gradient field. For the construction of recovery sequences, we will strengthen this to a \emph{global} assumption. Moreover, we also require a by now standard condition on $W$ (cf.\ \cite{contifonsecaleoni2002}) ensuring that the optimal profile problem reduces to a one-dimensional geodesic problem. Further technical geometric assumptions are made to simplify the construction of   recovery sequences. Explicitly, we require: 
\begin{enumerate}[resume, label=(H\arabic*), ref={\rm (H\arabic*)}]
    \item \label{H4} There exist $\kappa \in C^1(\R; \R^d \setminus \lbrace 0 \rbrace )$ and  $\gamma \in C^3(\R^d)$ with $\nabla \gamma \neq 0$ satisfying the following properties: \\[-0.5em]
    \begin{enumerate}
        \item[(i)] We have
            \begin{align}\label{eqn:global_well_structure}
                A = -B = (\kappa \circ \gamma) \otimes \nabla \gamma.
            \end{align}
        \item[(ii)] For $\nu := \nabla \gamma/|\nabla \gamma|$ we have
            \begin{equation}\label{eqn:W1Dstructure}
                W(x,M)\geq W\big(x,M\nu(x) \otimes \nu(x)\big) \qquad \text{for all $x\in \Omega$}.
            \end{equation}
        \item[(iii)]
            For any $t \in \gamma(\R^d)$ there exists a $C^2$-diffeomorphism $G\colon \{\gamma = t\} \to \R^{d-1}$ and the relative boundary of $\Omega \cap \{\gamma = t\}$ in $\{\gamma = t\}$ has zero $\H^{d-1}$-measure.
    \end{enumerate}
\end{enumerate}

We proceed with thorough comments on the specific properties in  Remarks  \ref{remark:reduction_of_wells}--\ref{remark:diffeomorphism_construction}.  The reader is invited to skip   these remarks on first reading as they become only relevant in   Section \ref{sec:limsup} on the construction of recovery sequences. 
 
\begin{remark}[Comments on \ref{H4}{\rm (i)}]\label{remark:reduction_of_wells}
  (a)  Without restriction, we can always reduce to the case of $A = -B$. Indeed, suppose that a potential $\Wtrunc$ with wells $\tilde A,\tilde B$  and corresponding functions  $\tilde{\alpha}$, $\tilde{\beta}$ are given as in {\rm \ref{H1}}. Then, by defining 
    \[
        W(x, M) := \Wtrunc\left(x, M + \frac{1}{2}(\tilde A(x) + \tilde B(x)) \right),
    \]
    we obtain a potential with wells given by $A = - B  :=   \frac{1}{2} (\tilde A - \tilde B)$. More precisely, considering the energy 
    \[
        \tilde E_\eps(u) = \int_{\Omega} \left(\frac{1}{\eps} \Wtrunc(x,\nabla u) + \eps |\nabla^2 u |^2 \right) dx
    \]
    for $u \in H^2(\Omega; \R^d)$, we can reformulate the energy in terms of $W$ by setting     \begin{align}\label{eqn:definition_shift}
        v = u + \frac{1}{2} (\tilde{\alpha} + \tilde{\beta}), \qquad   
        E_\eps(v) := \int_{\Omega} \left(\frac{1}{\eps}  W(x,\nabla v) + \eps |\nabla^2 v |^2 \right) dx
    \end{align} 
    to derive
    \[
        \tilde E_\eps(u) = E_\eps(v) + F_\eps(v)
    \]
    with 
    \[
        F_\eps(v) =  - \eps \int_\Omega  \frac{1}{4}|\nabla  \tilde{A} + \nabla \tilde{B} |^2  +(\nabla^2 v) :  (\nabla \tilde{A} + \nabla \tilde{B})  \, dx. 
    \]
    By an application of Young's inequality, we have that, for any sequence $\lbrace u_\eps\rbrace_\eps \subset  H^2(\Omega; \R^d)$, $\tilde E_\eps(u_\eps) $ is uniformly bounded in $\eps$ if and only if $E_\eps(v_\eps) $ is uniformly bounded in $\eps$, where $v_\eps$ is defined as in \eqref{eqn:definition_shift}. In particular, we have 
    \[
        \sup_{\eps > 0} E_\eps(v_\eps) < \infty \qquad \implies \qquad F_\eps(v_\eps) \xrightarrow{\eps \to 0} 0.
    \]
    As a consequence, $E_\eps$ and $\tilde E_\eps$ induce the same $\Gamma$-limit, up to shifting the functions in the sense of \eqref{eqn:definition_shift}.

    (b) For a local version of \eqref{eqn:global_well_structure} we refer to Lemma \ref{lem:rep} below. The global version used here, paired with the additional regularity on $\gamma$  ($C^3$ instead of $C^2$) has a purely constructive purpose allowing us to foliate the ambient Euclidean space by level sets with normals compatible with the rank-one connection. More precisely, let $f_x$ be the solution of the nonlinear ODE system 
     \[
        \begin{cases}
            f_x' = \frac{\nabla \gamma(f_x)}{|\nabla \gamma(f_x)|^2}  \\ 
            f_x(0) = x \in \{ \gamma = t \}
        \end{cases}
     \]
     for any (fixed) $t \in \gamma(\R^d)$. Since $\nabla \gamma \neq 0$, the ODE has a global solution. Furthermore, the induced flow 
   \begin{align}\label{barF}
        \begin{cases}
            \bar F\colon \{ \gamma = t \} \times \R \to \R^d \\
            \bar F(x, s ) := f_x(s)
        \end{cases}
    \end{align}
    has $C^{2}$-regularity since $\gamma \in C^{3}(\R^d)$.  Since 
    \[
        (\gamma(f_x))' = 1,
    \]
    $\bar F$ maps $\{ \gamma = t \} \times \{s\}$ diffeomorphically onto $\{ \gamma = t+s \}$.
    In particular, the level sets foliate the ambient Euclidean space.  
\end{remark}

\begin{remark}[Comments on \ref{H4}{\rm (ii)}]\label{remark:diffeomorphism_construction_2} 
The assumption \eqref{eqn:W1Dstructure}  implies the reduction of the local cell formulas to a one-dimensional geodesic problem. Similar assumptions have been used in the spatially homogeneous case, see e.g.\ \cite{contifonsecaleoni2002, davolifriedrich25, GalvaosousaMillot}. Without such an assumption,  the cell-energy density that appears in the $\Gamma$-limit  is expected to be nonlocal (cf.\ also \eqref{eqn:cellEnergyNonlocal}) and  `multidimensional'.  Indeed, already in the homogeneous case it is not known if one can reduce to a one-dimensional cell formula   (cf.\ \cite[Section 8]{contifonsecaleoni2002}).
\end{remark}

\begin{remark}[Comments on \ref{H4}{\rm (iii)}]\label{remark:diffeomorphism_construction}  
    (a)  The diffeomorphism $G$ along with \eqref{eqn:global_well_structure}  implies the existence of a $C^2$-diffeomorphism $F \colon \R^d \to \R^d$ which straightens the level sets of $\gamma$ globally. In fact, defining  $F \colon \R^d \to \R^d$ as the inverse of
    \[
        x \mapsto \bar F\big(G^{-1}(x'), x_d - t \big)
    \]
    with $\bar{F}$ given in \eqref{barF} (for a fixed $t \in \R^d$) yields a $C^2$-diffeomorphism such that $F(\R^{d-1} \times \{s\}) = \{\gamma = s\}$ for every $s \in \R$. As a consequence, the pushed-forward  rank-one connection
    \[
        ((\kappa \circ \gamma \circ F^{-1}) \otimes (\nabla \gamma \circ F^{-1})) \nabla F^{-1} 
    \]
    is of the form 
    \[
         \tilde \kappa(x_d) \otimes e_d
    \]
    for a suitable function $\tilde\kappa \in C^1(\R; \R^d \setminus \lbrace 0 \rbrace)$. This will simplify our construction of the recovery sequence for the $\Gamma$-limit later on in Section \ref{sec:limsup}. We also refer to Section \ref{sec:limit-laminates} for more details on the structure of the rank-one connection.  

    (b) The existence of $G$  is guaranteed for $d = 2,3$ if \ref{H4}{\rm (i)} holds since the line and the plane are the only non-compact, contractible manifolds modulo diffeomorphism in one and two dimensions, respectively. Indeed, by Remark \ref{remark:reduction_of_wells}{(ii)}, we have that $\{ \gamma = t \} \times \R$ is diffeomorphic to $\R^d$ for any $t \in \gamma(\R^d)$ and hence contractible. Since $\R$ is contractible, we immediately derive that any level set $\{\gamma = t\}$ must be contractible as well. Clearly, any level set of $\gamma$ is also non-compact since $\gamma$ is regular, which implies that any level set is diffeomorphic to the line (for $d = 2$), or the plane (for $d = 3$). Indeed, for $ d=2$ this follows from the fact that any connected 1-manifold without boundary is diffeomorphic to an open interval, while for $d=3$ this observation follows from the classification theorems for surfaces (cf.\ \cite[Theorem 3]{Richards1963}). \\[0.5em] 
    In higher dimensions, for instance, a condition of the type $\partial_d \gamma \neq 0$  would ensure that the level sets can be globally written as a graph. Lastly, we remark that a spatially independent version of this assumption appears in \cite{contifonsecaleoni2002}: in this `simplest' case, one has a foliation by hyperplanes. 

    (c) Requiring the relative boundary of $\{\gamma = t\} \cap \Omega$ in $\{\gamma = t\}$ to have zero $(d-1)$-Hausdorff measure prevents accumulation of fat Cantor-type sets in the boundary of $\Omega$ along the level sets. For instance, this condition can also be enforced by requiring the level sets of $\gamma$ to be transversal to $\partial \Omega$. Such an assumption has been present implicitly already in \cite{contifonsecaleoni2002}.  As an alternative, in the literature  of homogeneous solid-solid phase transitions, i.e., the case that $\gamma$ is affine, the domain has been required to be convex or strictly star-shaped for the construction of the recovery sequence, or that 
    \[
        t \mapsto \H^{d-1}(\Omega \cap \{\gamma = t\})
    \]
    is continuous, and the level sets $\Omega \cap \{\gamma = t\}$ are connected, see e.g.\ \cite{contifonsecaleoni2002, contischweizer06, GalvaosousaMillot}.
 
\end{remark}

\subsection{Main results} 

We now state our main results on $\Gamma$-convergence of the energies $E_\eps$ given in \eqref{eqn:energy} to a limit energy $E_0$. The limit energy $E_0$ is only  finite on a subset of  $H^2(\Omega; \R^d)$ consisting of suitably rigid functions.  We denote this collection of functions by 
\begin{equation}\label{eq:X}
    X(\Omega) : = \big\{u\in H^1(\Omega;\R^d) : \nabla u \in BV(\Omega;\{A(\cdot),B(\cdot)\})\big\},
\end{equation}
where we use the shorthand notation 
\begin{equation*}
%\label{eq:BV-AB}
    BV(\Omega;\{A(\cdot),B(\cdot)\}) : = \big\{M \in BV(\Omega;\R^{d\times d})\colon \,  M(x) \in \{ A(x),B(x)\} \text{ for a.e.\ }x\in \Omega\big\}.
\end{equation*}
  The functions in $X(\Omega)$ are   \emph{rigid} in the sense that they show laminate-type behavior with discontinuities in the gradient occurring along regular surfaces. This is in the spirit of the seminal work of Ball and James \cite{ball-james1987}, where discontinuities of the gradient occur along hyperplanes. In particular, we show that, under a natural requirement on the rank-one connection between $A$ and $B$, the set $J_{\nabla u}$ is the union of level sets of a suitable $C^2$-function related to the two wells. Our first result reads as follows.

\begin{theorem}[Local structure of jump set]\label{thm:characterisation-of-piecewise-BV-space}
    Let $A$ and $B$ be as in \ref{H1}. Assume that for an open set $U \subset \Omega$ we have
    \begin{equation}\label{eqn:rank1Structure}
    A(x) - B(x) = \kappa(\gamma(x)) \otimes \nabla \gamma(x)
    \end{equation} 
    for every $x\in U$, where $\kappa \in C^1(\R;\R^d\setminus\{0\})$ and $\gamma \in C^2(U)$ with $\nabla \gamma \neq 0$ in $U$. Let $u \in X(\Omega)$ and, for every $z \in J_{\nabla u}$, denote by $K_z$ the connected component of $\{x\in U \colon \,  \gamma(x) = \gamma(z) \}$ containing $z$. Then, it holds that
    \[
       \bigcup_{z\in J_{\nabla u} \cap U} K_z = J_{\nabla u}. 
    \]
\end{theorem}
\begin{figure}
    \centering
    \includesvg[width=0.5\linewidth]{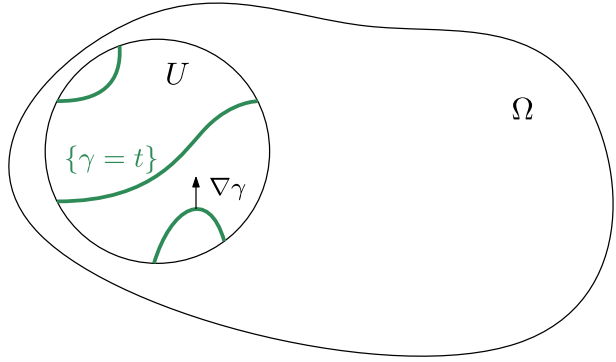}
    \caption{The prescribed rank-one connection in Theorem \ref{thm:characterisation-of-piecewise-BV-space} gives a local structure of the jump set.} \label{figure1} 
\end{figure}
We remark  that \eqref{eqn:rank1Structure} is always fulfilled locally if \ref{H1} holds, see Lemma \ref{lem:rep} below. When the stronger global assumption \ref{H4}(i) holds, Theorem \ref{thm:characterisation-of-piecewise-BV-space} gives a complete characterization of the interfaces in terms of the connected components of the level sets of $\gamma$, see also Figure~\ref{figure1}. 
The proof of Theorem \ref{thm:characterisation-of-piecewise-BV-space} will be given in Section \ref{sec:limit-laminates}. To write the limit energy functional, we need to introduce the notion of \emph{geodesic cell-energy density}.

\begin{definition}\label{def:geo-cell-en}
    Let $A$ and $B$ be as in \ref{H1}. For every $x_0\in \Omega$, we define
    \begin{equation}\label{eqn:geoCellEnergy}
        \begin{aligned}
            \sigma_{\rm geo}(x_0) : = \inf \left\{ \int_{0}^1 2\sqrt{W(x_0,\Phi)}| \Phi'| \, dt \colon \, \Phi\in W^{1,1}((-1,1);\R^{d\times d})\text{ with }\Phi(-1) = B(x_0),\, \Phi(1) = A(x_0)  \right\}.
        \end{aligned}
    \end{equation}
\end{definition}
\noindent With this, we introduce the limit energy functional as
\begin{equation}\label{eqn:limitEnergy}
    E_0(u) : = \int_{J_{\nabla u}} \sigma_{\rm geo}  \, d\mathcal{H}^{d-1} \quad \quad \text{ for }u\in X(\Omega),
\end{equation}
where $X(\Omega)$ is defined in \eqref{eq:X}. Our main result is the following.
\begin{theorem}[Compactness and $\Gamma$-convergence]
\label{thm:main}
 Let $\Omega$ be an open, bounded Lipschitz set, and let  $W$ satisfy assumptions {\rm \ref{H1}}--{\rm \ref{H3}}.  Let $E_\eps$ and $E_0$ be as in \eqref{eqn:energy} and \eqref{eqn:limitEnergy}, respectively. Then, the following holds.
\begin{itemize}
\item {(Compactness and $\liminf$-inequality).} 
    Let  $\{u_{\ep}\}_\ep \subset H^2(\Omega, \R^d)$ be such that 
    \begin{align}\nonumber %\label{bounded-energy}
        \sup_{\ep>0} E_{\ep}  (u_\ep)   <+\infty. 
    \end{align}
    Then, there exists $u \in X(\Omega)$ such that, up to a (not relabeled) subsequence, it holds that
    \[
        u_{\ep} - \frac{1}{|\Omega|}\int_\Omega u_{\ep}  \, dx  \to u \quad \text{ strongly in } H^1(\Omega;\R^d).
    \]
    Furthermore,
    \begin{equation*}
      E_{0}(u) \leq   \liminf_{\eps \to 0} E_{\eps}(u_\eps).
    \end{equation*}
    \item {($\limsup$-inequality).} 
    Assume further that \ref{H4} holds, and that $\Omega$ is simply connected.   Then, for every $\bar{u}\in X(\Omega)$ there exists a sequence $\{\bar{u}_{\ep}\}_\ep \subset H^2(\Omega, \R^d)$ such that  \[
        \bar{u}_{\ep} \to \bar{u} \text{ strongly in } H^1(\Omega;\R^d),
    \]
    and
    \begin{equation*}
        \limsup_{\eps \to 0} E_{\eps}(\bar{u}_\eps)\leq E(\bar u) .
    \end{equation*}
\end{itemize}
\end{theorem}
The proof will be given in Sections \ref{sec:compactness}--\ref{recoverysequnecesection} and is organized as follows: The compactness of sequences with equibounded energy is carried out in Section~\ref{sec:compactness}. The proof of the $\Gamma$-liminf inequality is tackled in Section~\ref{sec:liminf-interm}: we first provide an intermediate lower bound for the asymptotic behavior of the energy functionals in Proposition \ref{thm:liminf-inter}, by means of a multidimensional cell-energy density, cf.\ Definition \ref{def:nonlocal-cell}.  In Proposition \ref{thm:cellEnergyGeoBound}, we prove that this cell-energy density is bounded from below by the geodesic cell-energy density \eqref{eqn:geoCellEnergy}. We emphasize that assumption \ref{H4} is not used in the proof of the compactness or in the identification of the lower bound.
The construction of  recovery sequences is given in Section \ref{sec:limsup} under assumption \ref{H4}. In particular, we will use the diffeomorphism $F$  provided by Remark \ref{remark:diffeomorphism_construction}  to flatten out the jump set and to first construct recovery sequences in the diffeomorphed domain $F(\Omega)$. This   construction  is first provided for simple geometries of $F(\Omega)$ in Theorem \ref{thm:first-limsup}, and the extension to general geometries is then subject of Theorem \ref{thm:general-limsup}.

\section{Characterization of limit configurations}
\label{sec:limit-laminates}

This section is devoted to the proof of Theorem \ref{thm:characterisation-of-piecewise-BV-space}, which in turn is based on several intermediate results. We begin by showing that, under mild regularity assumptions on the maps $A$ and $B$  (and no gradient structure), the set where a map in $BV(\Omega;\{A(\cdot),B(\cdot)\})$ coincides with one of the two wells has finite perimeter.

\begin{lemma}[Finite perimeter of   phases]\label{lem:Oset}
    Let $A,B\in W^{1,\infty}(\Omega;\R^{d\times d})$ satisfy  $
        \min_{x\in \overline{\Omega}}|A(x)- B(x)| > 0$ 
    and let $M \in BV(\Omega, \{A(\cdot), B(\cdot)\})$. Then the set $\Landau := \{x\in\Omega \colon \, M(x) = A(x) \}$ has finite perimeter with
    \[
        \H^{d-1}\big((\partial^*\Landau \cap \Omega) \triangle J_M\big) = 0.
    \]
\end{lemma}
\begin{proof} Defining the function $V : = M-A\in BV(\Omega, \{0, B(\cdot)-A(\cdot)\})$, we find that $\mathcal{O} = \{|V|<\eta\}$ for all $\eta\in (0,\delta)$, where $\delta :=\min_{x\in \overline{\Omega}}|A(x)- B(x)|$. Since $|V|\in BV(\Omega)$, by the coarea formula, $\{|V |<\eta\}$ is a set of finite perimeter for almost every (and thereby every) $\eta\in (0,\delta)$. The $\H^{d-1}$-equivalence of $\partial^*\mathcal{O}$ and $J_M$ follows directly from blow-up properties of the jump set, cf.\ \cite[Definition 3.67]{ambrosio2000fbv}.
\end{proof}

We show that, along the jump set $J_{\nabla u}$ of any map $u\in X(\Omega)$ (see \eqref{eq:X}), the wells $A$ and $B$ must have a rank-one connection.  

\begin{lemma}[Necessity of a rank-one connection]\label{lemma-rank-1-necessity}
    Let $A,B\in C^1(\Omega;\R^{d\times d})$ be such that $ A  =\nabla \alpha $ and $B  =\nabla \beta$ for some  $\alpha,\beta \in C^2(\overline{\Omega};\R^d)$ and  
        $\min_{x\in \overline{\Omega}}|A(x)- B(x)| > 0$. Let $u\in X(\Omega)$. Then for $\H^{d-1}$-a.e.\ $x \in J_{\nabla u}$, it holds that
    $$A(x) - B(x)  = \xi (x)\otimes \nu_{\nabla u}(x),$$
    where $\xi(x)\in \R^{d}\setminus\{0\}$ and $\nu_{\nabla u}$ is a normal to $J_{\nabla u}.$ 
%    Furthermore, if
%    \[
%        A(x) - B(x) = k(x) \otimes \nu(x)
%    \]
%    holds for measurable functions $k: \Omega \to \R^n \setminus \{0\}$ and $\nu: \Omega \to \S^{n-1}$ then we have $\nu(x) = \pm \nu_{J_{\nabla u}}(x)$ for $\H^{n-1}$-a.e. $x \in J_{\nabla u}$ where the sign is up to choice of normal along the jump set of $\nabla u$.
\end{lemma}
\begin{proof}
    By the decomposition of distributional derivatives for $BV$-functions, see \cite[Section 3.9]{ambrosio2000fbv}, we have
    \[
        \partial_i \partial_j u_k \, \mres_{J_{\nabla u}} = ((\partial_j u_k)^+ - (\partial_j u_k)^-)(\nu_{J_{\nabla u}})_i \H^{d-1}\mres_{J_{\nabla u}}
    \]
    for every $i,j,k\in \{1,\dots,d\}$.
    In particular, by exchanging partial derivatives in the sense of distributions we find
    \[
        ((\partial_j u_k)^+ - (\partial_j u_k)^-)(\nu_{J_{\nabla u}})_i \, \H^{d-1}\mres_{J_{\nabla u}} = ((\partial_i u_k)^+ - (\partial_i u_k)^-)(\nu_{J_{\nabla u}})_j \, \H^{d-1}\mres_{J_{\nabla u}}
    \] 
    for every $i,j,k\in \{1,\dots,d\}$, which, in turn, yields 
    \begin{equation}
    \label{eq:distr-id}
        ((\partial_j u_k)^+ - (\partial_j u_k)^-)(\nu_{\nabla u})_i = ((\partial_i u_k)^+ - (\partial_i u_k)^-)(\nu_{\nabla u})_j
    \end{equation}
    $\H^{d-1}$-almost everywhere on $J_{\nabla u}$. 
    Now, let $x_0\in J_{\nabla u}$ be such that \eqref{eq:distr-id} holds. Up to possibly reversing the orientation of $\nu_{\nabla u}(x_0)$, we may assume that $\nabla u^+(x_0)=A(x_0)$ and $\nabla u^-(x_0)=B(x_0)$. Choosing $i_0:=i(x_0)\in \{1,\dots,d\}$ such that $(\nu_{\nabla u})_{i_0}(x_0)\neq 0$, condition \eqref{eq:distr-id} with $i = i_0$ becomes  
    $$ A(x_0)e_j - B(x_0)e_j = \xi (x_0)(\nu_{\nabla u})_j(x_0),$$
    for every $j\in\{1,\dots,d\}$, with $\xi(x_0) := \frac{A(x_0)e_{i_0} - B(x_0)e_{i_0}}{(\nu_{\nabla u})_{i_0}(x_0)}$. The lemma follows by observing that the above construction holds for $\H^{d-1}$-almost every $x_0\in J_{\nabla u}$.
\end{proof}

Next, we show that the rank-one connection between $A$ and $B$ can be described locally by a gradient field if $A$ and $B$ also are, i.e., as in \ref{H1}. This leads to a local version of \ref{H4}(i). 

\begin{lemma}[Local representation of the rank-one connection]
\label{lem:rep}
    Let $A$ and $B$ as in \ref{H1}. Then, for every $y \in \Omega$ there exists an open set $U \subset \Omega$ with $y\in U$, as well as two maps $\kappa \in C^1(\R;\R^d \setminus \lbrace 0 \rbrace)$  and $\gamma \in C^2(U)$     such that   
    \begin{equation}\nonumber
    %\label{eq:here-local}
        A(x) - B(x)  = \kappa(\gamma(x)) \otimes \nabla \gamma(x)
         \end{equation}
    for every $x\in U$.  
\end{lemma}
\begin{proof}
    Fix $y\in U$. Let $\alpha, \beta$ be as in \ref{H1}. Then, there exists an index $i\in\{ 1, \ldots, d\}$ such that 
    \[
        \nabla(\alpha_i - \beta_i)(y) \neq 0.
    \]
    Set $$U := \{ x \in \Omega \colon \,  \nabla(\alpha_i - \beta_i)(x) \neq 0 \},$$
    as well as,  for every $x \in U$,
    \begin{gather*}
         \quad \gamma(x) := \alpha_i(x) - \beta_i(x), \quad \nu(x) := \frac{\nabla (\alpha_i(x) - \beta_i(x))}{|\nabla (\alpha_i(x) - \beta_i(x))|}, \text{ and } \\
        \quad \xi_j(x) := \nabla (\alpha_j(x) - \beta_j(x)) \cdot \nu (x) \quad \text{ for all }j = 1, \ldots,  d.
    \end{gather*}   
    Now, by definition, $\nu(x)$ is a positive multiple of the $i$-th row of $A(x)-B(x)$. Since ${\rm rank}\, (A(x) - B(x)) = 1$, every row of the matrix $A(x) - B(x)$ is a scalar multiple of $\nu(x)$ for every $x\in U$. In other words,
    $$\nabla (\alpha_j(x)-\beta_j(x))=\big(\nabla (\alpha_j(x)-\beta_j(x))\cdot \nu (x) \big) \,\nu(x)=\xi_j(x)\nu(x)$$
    for every  $j = 1,\ldots, d$ and $x\in U$. In particular,  for $\zeta = |\nabla \gamma|^{-1} \xi$ we have shown
    \begin{align}\label{ABC}
        A - B = \zeta \otimes \nabla \gamma.
    \end{align}
    Inserting \eqref{ABC} into $\partial_i \partial_k (\alpha_j - \beta_j) = \partial_k \partial_i (\alpha_j - \beta_j)$, direct computation yields 
    \[
        \nabla \zeta_j \otimes \nabla \gamma = (\nabla \zeta_j \otimes \nabla \gamma)^T
    \]
    for each component $j = 1, \ldots, d$.
    This, in turn, implies $ 
        \nabla \zeta_j = \lambda \nabla \gamma$ 
    for a suitable function $\lambda \in C^1(U)$. By using the Constant Rank Theorem for manifolds \cite[Theorem 4.12]{LeeSmoothManifolds} applied to $\gamma$, we restrict our neighborhood $U$ (without renaming) and derive the existence of inner and outer diffeomorphisms $f\colon\R^d \to \R^d$ and $g\colon \R \to \R$ such that
    \[
        \tilde \gamma(x) := (g \circ \gamma \circ f)|_{B_\eps(0)}(x) = x_1
    \]
    for some small $\eps > 0$. Denoting $\tilde \zeta_j := g \circ \zeta_j \circ f$ and $\tilde \lambda := (\lambda \circ f) \frac{g'\circ \zeta_j \circ f}{g'\circ \gamma \circ f}$,  by using $ 
        \nabla \zeta_j = \lambda \nabla \gamma$ 
    we derive on $B_\eps(0)$
    \[
        \nabla \tilde \zeta_j = \tilde \lambda e_1.
    \]
    We infer for $x \in B_\eps(0)$ that
    \[
        \tilde \zeta_j(x) = F_j(x_1) = F_j(\tilde \gamma(x))
    \]
    for some   $F_j \in C^1(\R)$. In particular, 
    \[
        \zeta_j = g^{-1} \circ \tilde \zeta_j \circ  f^{-1} =  g^{-1} \circ F_j \circ g \circ \gamma.
    \]
    Now, setting 
    \[
        \kappa_j := g^{-1} \circ F_j \circ g, 
    \]
    we conclude by \eqref{ABC}.     
 \end{proof}
   
As preparation for the proof of Theorem \ref{thm:characterisation-of-piecewise-BV-space}, we show that, if a point $x_0$ belongs to the jump set $J_{\nabla u}$, then the level set $\{x \colon \,  \gamma(x)  = \gamma(x_0)\}$ must be contained in the $J_{\nabla u}$, at least locally. 

%we further specify the \RRR laminate-like  structure of $J_{\nabla u}$. 
\begin{proposition}\label{prop:local-characterisation}
Let $A$ and $B$ be as in \ref{H1}. Assume that for an open set $U \subset \R^d$ we have
$$
    A(x) - B(x) = \kappa(\gamma(x)) \otimes \nabla \gamma(x)
$$ 
for every $x\in U$, where 
$\kappa \in C^1(\R;\R^d\setminus\{0\})$ and $\gamma \in C^2(U)$ with $\nabla \gamma \neq 0$ in  $U$.  Let $u \in X(\Omega)$. Then, for every compact  set  $K \subset U$ there exists $\eta = \eta(K)>0$ such that for every $z\in J_{\nabla u} \cap K$ the following holds: there exist $t_1,\ldots, t_N\in \R$ such that $t_i < t_{i+1}$ for every $i=1,\dots,N-1$, and
\begin{equation}\label{eqn:subsetInclusion}
J_{\nabla u} \cap B(z,\eta) = \bigcup_{i=1}^N \{x  \in U\colon \,   \gamma(x) = t_i\} \cap B(z,\eta).
\end{equation}
Additionally, up to an $\mathcal{L}^d$-null  set, either $\{\nabla u = A\} \cap B(z,\eta)$ or $\{\nabla u = B\} \cap B(z,\eta)$  coincides with $$ \bigcup_{i = 0}^{\lceil {(N-1)/2}\rceil } \{x\colon\,  \gamma(x)\in (t_{2i},t_{2i+1})  \}\cap B(z,\eta),$$
where for definiteness $t_0 = - \infty$ and $t_{N+1} = +\infty$.
\end{proposition}

\begin{proof}
By possibly reducing to a smaller set, we can assume that  $U$ has Lipschitz boundary. Since we assume that $\nabla \gamma \neq 0$, we can apply the Constant Rank Theorem \cite[Theorem 4.12]{LeeSmoothManifolds} to $\gamma$ near $z \in K$ to find $r_z > 0$ and diffeomorphisms $\psi_{z}\colon\R\to \R$ and $\Phi_z \colon {B_z} \to \Phi_z({B_z}) \subset U$ for $B_z := B(0, r_z)$ such that 
$${ \psi_z \circ \gamma \circ \Phi_z(x) = x_1 \quad \quad \text{ and }\quad \quad \Phi_z(0) = z, \quad \psi_z(\gamma(z)) = 0 }.$$
This means that the function $\Phi_z^{-1}$ pulls the level sets of $\gamma$ back to hyperplanes locally. This implies that for all $x \in B_z$ we have 
\begin{equation}\label{eqn:orthogonalRelation}
(\nabla \gamma\circ \Phi_z)(x) \nabla \Phi_z(x) = (\psi_z^{-1})'(x_1)e_1.
\end{equation} 
Now, let $\tilde r_z > 0$  be   such that $B(z, 2\tilde r_z) \subset \Phi_z(B_z)$. Since $K$ is compact, we find finitely many $   \lbrace z_k \rbrace_{k=1}^m  \subset K$  such that 
\[
   K \subset \bigcup_{k = 1}^m B(z_k,  \tilde r_{z_k}) \subset \bigcup_{k = 1}^m  \Phi_{z_k}(B_{z_k}).
\]
Set $\eta = \min_{k \in \{1, \ldots,  m\}} \tilde r_{z_k}$. We fix   $z \in K$ and notice that $B(z, \eta) \subset B(z_{k_0}, 2\tilde r_{z_{k_0}}) \subset \Phi_{z_{k_0}}(B_{z_{k_0}})$ for some $k_0 \in \{1,  \ldots, m\}$. In particular, for $\Phi := \Phi_{z_{k_0}}$ and $\psi := \psi_{z_{k_0}}$ on $ B := B_{z_{k_0}}$, we have \eqref{eqn:orthogonalRelation}. Further, recalling Lemma \ref{lem:Oset} with $\mathcal{O}: = \{x\in \Omega\colon \,  \nabla u(x) = A(x)\}$, we now apply \cite[Proposition 17.1]{maggi_2012} 
to conclude that $\mathcal{U}:= {\Phi}^{-1} (\mathcal{O})\cap B $ is a set of finite perimeter and satisfies the relation
\begin{equation}\label{eqn:integralMaggi}
\int_{\partial^* \mathcal{U}} \phi \, \nu_{\mathcal{U}} \, d\H^{d-1} = \int_{\partial^* \mathcal{O}} \phi\circ \Phi^{-1} (J\Phi^{-1}) (\nabla \Phi \circ \Phi^{-1})^T \nu_{\mathcal{O}} \, d\H^{d-1}
\end{equation}
for all $\phi \in C_c( B )$ where $J\Phi^{-1}$ is the Jacobian of $\Phi^{-1}$.
Since by Lemma \ref{lem:rep} it holds that $\nu_{\mathcal{O}} =\pm \frac{\nabla \gamma^T}{|\nabla \gamma|}$, property 
\eqref{eqn:orthogonalRelation} (up to transposes) entails that
$$ (\nabla \Phi \circ \Phi^{-1})^T \nu_{\mathcal{O}} = \beta e_1 \quad  \text{ for some } \beta = \beta(x)\in \R.$$
It follows that the only nonzero component of the vector in \eqref{eqn:integralMaggi} is the first component, and a localization argument gives that $\nu_{\mathcal{U}} = \pm e_1$ $\H^{d-1}$-a.e.\ on $\partial^*\mathcal{U}.$ A mollification argument (cf.\ \cite[Proposition 15.15]{maggi_2012}) can then be used to show that $\chi_{\mathcal{U}}$ depends only on $x_1$. Since $\mathcal{U}$ has finite perimeter (in $\Omega$), there exist increasing $s_1,\ldots, s_N\in \R$ (with $s_0 = -\infty$ and $s_{N+1} = +\infty$)  such that either $\mathcal{U}$ or $ B \setminus \mathcal{U}$ is given by 
$${\bigcup_{i = 0}^{ \lceil {(N-1)/2}\rceil } \big\{x = (x_1,\ldots,x_d)\colon \,  x_1\in (s_{2i},s_{2i+1}) \big\}\cap B.}$$
Carrying this information to $\mathcal{O} \cap B(z, \eta)$ using the diffeomorphism $\Phi$, we obtain the thesis for the ball $B(z,\eta)$ with $t_i = \psi^{-1} (s_i)$ or $t_i = \psi^{-1} (  s_{N-i + 1})$ (depending if $\psi$ is increasing or decreasing). 
\end{proof}

We summarize the above findings. By Lemma \ref{lemma-rank-1-necessity}, if there are nontrivial elements in  $X(\Omega)$, then $A$ and $B$ must be rank-one connected along the jump set. If we want the jump set to include some surface, this forces the rank-one condition to hold pointwise a.e.\ along the surface, thereby justifying the need for \ref{H1}. Moreover, Lemma \ref{lem:rep} characterizes the relation between the rank-one connection and a local gradient field: there is a function $\gamma$ such that the rows of $A-B$ are orthogonal to its level sets. This points towards an assumption like \ref{H4}. 
%The norm of the rows are a function of $\gamma$. 
When \ref{H1} holds,   Lemma \ref{lem:rep} together with Proposition \ref{prop:local-characterisation} for $u\in X(\Omega)$ shows that the jump set $J_{\nabla u}$ locally consists of finitely many connected surfaces, and on each of  these surfaces the rank-1-connection is of the form $F \otimes \nabla \gamma(x)$ for a constant $F \in \R^d$.

We are now in a position to prove the main result of this section, which effectively shows that the conclusion of the previous paragraph holds in an open set $U$, provided that one prescribes the rank-one connection.  

\begin{proof}[Proof of Theorem \ref{thm:characterisation-of-piecewise-BV-space}] 
Fix $z \in J_{\nabla u} \cap U$. We show that $\emptyset \neq J_{\nabla u}\cap K_z$ is relatively open and closed in $K_z$, which implies it must be equal to $K_z$. Indeed, by Proposition \ref{prop:local-characterisation}, the set $J_{\nabla u}\cap K_z$ is relatively open in $K_z$. To see that it is closed, consider a sequence $\{y_i\}_{i\in \N} \subset J_{\nabla u}\cap K_z$ converging to a point $y \in K_z$. Let $\overline{B(y, \eps)} \subset U$ for some $\eps > 0$. By Proposition \ref{prop:local-characterisation} we obtain the existence of $\eta_\eps = \eta(\overline{B(y, \eps)})$  such that for each $y_i \in \overline{B(y, \eps)}$ (which holds for $i$ sufficiently large) we get that inside  $B(y_i, \eta_\eps)$ the jump $J_{\nabla u}$ has the structure as given in  \eqref{eqn:subsetInclusion}.  Now, we choose $i_0  \in \N$ sufficiently large such that 
\[
    y_i \in B(y,\eta_\eps/4) \subset B(y_{i_0}, \eta_\eps/2)  \qquad \text{ for all } i \geq i_0
\]
holds. Since $\overline{B(y_{i_0}, \eta_\eps/2)} \cap J_{\nabla u}$ is a closed subset of $\R^d$ (by \eqref{prop:local-characterisation})  and since $y_i$ converges to $y$  we infer $y \in J_{\nabla u}$. This concludes the proof.  \end{proof}

\section{Compactness}
\label{sec:compactness}
In this section, we show that sequences of deformations with equibounded energy converge (up to subsequences) strongly in $H^1(\Omega; \R^d )$ to maps $u\in X(\Omega)$, see \eqref{eq:X}. 
Before stating our compactness result, we introduce a geodesic distance which generalizes the geodesic cell-energy density  \eqref{eqn:geoCellEnergy} by measuring how much energy must be expended to move between any two points $N$ and $M$ along a path in $\R^{d\times d}$:
For every $ x_0 \in \Omega$, and every $M, N  \in \R^{d\times d}$, we set
$$
\begin{aligned}
&d_W(x_0, M, N) : =  \inf \left\{\int_{-1}^1 2\sqrt{W(x_0,\Phi)}|\Phi'|\, dt \colon \,  \Phi \in W^{1,1}((-1,1);\R^{d\times d}) \text{ with }\Phi(-1)=M, \ \Phi(1) = N\right\}.
\end{aligned} $$
In the sequel, it will be convenient to work with a truncated geodesic distance. To this end, given $C_0>0$, we introduce the truncated potential
\begin{align}\label{Wtrunc}
    \Wtrunc(x,M):=c\min \{|M - A(x)|^2, |M - B(x)|^2, 1 \} \quad \text{for every $x\in \overline{\Omega} $ and $M\in \R^{d\times d}$}
\end{align} 
where $c > 0$ is chosen such that  
\begin{equation}\label{eqn:WtoWtilde}
    W \geq \Wtrunc,
\end{equation}
which is possible due  to \ref{H2}.
 We state an estimate on the gradient of the truncated geodesic distance.  
\begin{lemma}[Lipschitz properties]\label{geodesic-distance-gradient-estimates-new}
   Let $W$ satisfy {\rm \ref{H1}--\ref{H3}} and define $ \Wtrunc$ as in \eqref{Wtrunc}.  
    Then, there exists $C>0$ such that the function $f_\Wtrunc(x,M) := d_\Wtrunc(x,A(x),M)$ satisfies      \begin{equation}\label{eqn:gradMest}
        |\nabla_M f_{\Wtrunc}(x,M)| \leq C\sqrt{\Wtrunc(x, M)}
    \end{equation}
    and 
    \[
        |\nabla_x f_{\Wtrunc}(x,M)| \leq C(1 + |M|)
    \]
    for almost  every $x\in \Omega$ and $M\in \R^{d\times d}$.
\end{lemma}

The proof of the result will be given in Section \ref{sec:geodesics-disc}, see in particular Lemma \ref{geodesic-distance-gradient-estimates}. Now, we come to the compactness result. 

\begin{theorem}[Compactness]
\label{mainthm:compactness}
    Let $W$ satisfy {\rm \ref{H1}}--{\rm \ref{H3}}. Suppose that for a sequence $\{u_n\}_{n \in \N} \subset H^2(\Omega; \R^d)$ and $\eps_n \to 0^+$ we have
    \begin{align}\label{bounded-energy}
        \sup_{n \in \N} E_{\epsilon_n}(u_n) < + \infty.
    \end{align}
    Then, there exists a subsequence (not relabeled) and $u \in X(\Omega)$ such that
    \[
        u_{n} - \frac{1}{|\Omega|}\int_\Omega u_{n}  \, dx \to u \text{ strongly in } H^1(\Omega;\R^d).
    \]
   
\end{theorem}
\begin{proof}

\emph{Step 1 (Weak convergence of $u_n$).} Up to a translation, it is not restrictive to assume that $\int_\Omega u_{\eps_n}  \,  dx =0$ for all $n \in \N$. In view of assumption \ref{H2} and $A,B \in L^\infty(\Omega;\R^{d\times d})$,  by \eqref{bounded-energy}  we have 
    \begin{equation}\label{eqn:aPrioriH1}
        \int_\Omega |\nabla u_n|^2 \, dx \leq C \int_\Omega W(x, \nabla u_n)\,dx + C|\Omega|  \leq C.
    \end{equation}
 Moreover, for all $t$ sufficiently large depending only on $\Vert A \Vert_\infty, \Vert B \Vert_\infty$ we find  
    \[
        \int_{\{ |\nabla u_n|  > t  \}} |\nabla u_n|^2 \, dx \leq C\eps_n.
    \]
From this it follows that 
    \begin{align}\label{eq:2-integrability}
        \lim_{t \to \infty}  \sup_{n \in \N} \int_{\{ |\nabla u_n| > t  \}} |\nabla u_n|^2 \, dx = 0,
    \end{align}
    i.e.,  the equi-integrability of $|\nabla u_n|^2$ and, consequently, equi-integrability of $\nabla u_n$. In particular, $\{u_n\}_{n \in \N}$ is equi-bounded in $H^1(\Omega; \R^d)$. Up to extracting a  subsequence (not relabeled), we may assume that there exists  $u\in H^1(\Omega; \R^d)$ and a gradient Young measure $\lbrace \eta_x\rbrace_{x\in\Omega}$, such that  $u_{n}  \rightharpoonup u$ weakly in $H^1(\Omega; \R^d$) and $\nabla u_{n}  $ generates $ \lbrace \eta_x \rbrace_{x\in \Omega}$. Now, by the fundamental theorem of Young measures \cite[Theorem 8.6]{FonsecaLeoni07},
   we infer
    \[
        \int_\Omega \int_{\R^{d \times d}} W(x, M) \, d\eta_x(M) \, dx \leq \lim_{n \to \infty} \int_\Omega W(x, \nabla u_n(x))\, dx = 0,
    \]
    which implies 
    \[
        \int_{\R^{d \times d}} W(x, M) \, d\eta_x(M) = 0
    \]
    for almost every $x \in \Omega$.  Hence, by \ref{H2}, we see that the measure $\eta_x$ must be supported on $\{ A(x), B(x) \}.$  In particular, since $\eta_x$ is a probability measure, we find
    \begin{equation} \label{eqn:etaYoung}
        \eta_x = \theta_x \delta_{A(x)} + (1 - \theta_x) \delta_{B(x)} \quad 
    \end{equation}
    for some $\theta_x\in [0,1]$, where $\delta_{M}$ is the Dirac measure  at $M$.

   \emph{Step 2 (Convergence of geodesic distance).}
    Recalling   the function $f_{\Wtrunc}$ introduced in Lemma \ref{geodesic-distance-gradient-estimates-new}, define 
    \[
        \phi_n(x) := f_{\Wtrunc}(x, \nabla u_n(x))
    \]
    for $x\in \Omega$.
    Since $f_{\Wtrunc}$ is locally Lipschitz continuous by Lemma \ref{geodesic-distance-gradient-estimates-new}, the chain rule yields
    \[
        |\nabla \phi_n| \leq |\nabla_x f_{\Wtrunc}( \cdot, \nabla u_n)| + |\nabla_M f_{\Wtrunc}(\cdot, \nabla u_n)||\nabla^2 u_n| \leq C\left(1 + |\nabla u_n| + \sqrt{\Wtrunc(\cdot, \nabla u_n)}|\nabla^2 u_n|\right).
    \]
    Therefore, by \eqref{eqn:aPrioriH1}, a classical Young's inequality argument, and \eqref{eqn:WtoWtilde}, we deduce
    \[
        \int_\Omega |\nabla \phi_n (x)| \, dx \leq C(1 + E_{\eps_n}(u_n)).
    \]
    By \eqref{eqn:gradMest} and by the boundedness of $\Wtrunc$, we infer that 
    $$|\phi_n(x)| = |f_{\Wtrunc}(x,\nabla u_n(x)) - f_{\Wtrunc}(x,A(x))|\leq  C  |\nabla u_n(x) - A(x)|.$$ 
    By the triangle inequality, Young's inequality, and the boundedness of $A$, this further implies that
    \[
        \int_\Omega |\phi_n| \, dx \leq C\left(1 +   \int_\Omega |\nabla u_n|^2 \, dx\right).
    \]
    In particular, the sequence $\lbrace \phi_n\rbrace_{n \in \N}$ is equibounded in $W^{1,1}(\Omega)$. Thus, there exists a (not relabeled) subsequence such that 
    \begin{equation}\label{eqn:phiL1converge}
        \phi_n \to \phi  \text{ in $L^1(\Omega)$ }
    \end{equation}
     for $\phi \in BV(\Omega)$, and $\phi_n$ generates a Young measure $\lbrace \rho_x\rbrace_{x\in \Omega}$ as $n \to \infty$. Since \eqref{eqn:phiL1converge} implies convergence in measure, we find (cf.\ \cite[Corollary 8.9]{FonsecaLeoni07})  
    \[
        \rho_x = \delta_{\phi(x)} \quad \text{ for a.e.\ $x \in \Omega$}.
    \]

    \emph{Step 3 (Identification of the limit).}
    Recalling the Young measures $\lbrace \eta_x\rbrace_{x\in \Omega}$ and $\lbrace \rho_x\rbrace_{x\in \Omega}$ from Steps~1 and~2 (see also \eqref{eqn:etaYoung}), and testing against $F = h \otimes \varphi$   where $h \in L^1(\Omega)$ and $\varphi \in C_0( \R )$ we infer 
    \begin{align}\label{3X}
        \int_\Omega \int_{\R} F(x,y) \, d\delta_{\phi(x)}(y) \, dx &= \lim_{n \to \infty} \int_\Omega F(x,\phi_n(x)) \, dx \notag  \\
        &= \lim_{n \to \infty} \int_\Omega F(x, f_V(x, \nabla u_n(x)) \, dx \notag  \\
        &= \int_\Omega \int_{\R^{d\times d}} F(x,f_{\Wtrunc}(x, M)) \, d\eta_x(M) \, dx. 
    \end{align}
    Here, we used that 
    $$
        |{ F(x,  \phi_n  (x))|  \leq C|h(x)|}
    $$ 
   implies equi-integrability of $F(x,  \phi_n (x))$, so we can also apply the fundamental theorem for Young measures for $\eta$, whereas the first equality simply follows by the definition of a Young measure (cf.\ \cite[Definition 8.3]{FonsecaLeoni07}). Since
    \begin{align*}
        \int_\Omega \int_{\R^{d\times d}} F(x,f_{\Wtrunc}(x, M)) \, d\eta_x(M) \, dx 
        &
        = \int_\Omega \big(\theta_x F(x,f_{\Wtrunc}(x,A(x)))+ (1 - \theta_x) F(x,f_{\Wtrunc}(x, B(x))) \big)\, dx \\ & =  \int_\Omega \int_{\R} F(x,y) \, d\big(\theta_x\delta_{f_{\Wtrunc}(x,A(x))} (y) + (1 - \theta_x)\delta_{f_{\Wtrunc}(x, B(x))} (y) \big)\, dx,
    \end{align*}
    we thus infer, by \eqref{3X}, 
    \begin{equation}\label{eqn:DiracMeasureRel}
        \delta_{\phi(x)} = \theta_x\delta_{f_{\Wtrunc}(x,A(x))} + (1 - \theta_x)\delta_{f_{\Wtrunc}(x, B(x))} \quad \text{ for a.e.\ $x \in \Omega$.}
    \end{equation} This implies that $\theta_x \in \{0, 1\}$ for almost every $x \in \Omega$. Next, we define a set  $E\subset \Omega$ by the relation 
    $
        \chi_E(x)  = \theta_x
    $
    for a.e.\ $x \in \Omega$. Then, we rewrite \eqref{eqn:DiracMeasureRel} as \[
        \phi(x) = \chi_E(x)f_{\Wtrunc}(x, A(x))+(1 - \chi_E(x))f_{\Wtrunc}(x, B(x))=(1 - \chi_E(x))f_{\Wtrunc}(x, B(x))
    \] for almost every $x\in \Omega$, which in particular implies that $E$ is Lebesgue measurable.
    Note that in the second equality we have used    $f_{\Wtrunc}(x,A(x))\equiv 0$ by definition of $f_{\Wtrunc}$.
    
    Since $\phi  \in BV(\Omega)$ and $\inf_{x\in \Omega}f(x, B(x)) > 0$ by \eqref{eq:delta-well},  by repeating the arguments in the proof of Lemma~\ref{lem:Oset} with $A$ and $B$ replaced by $f_{\Wtrunc}(\cdot,A(\cdot))$ and $f_{\Wtrunc}(\cdot,B(\cdot))$, respectively, we find that $E$ is a set of finite perimeter.
Note that for every $g\in L^2(\Omega; \R^{d \times d})$  we find that $g : \nabla u_n$ is equi-integrable by \eqref{eq:2-integrability}. Then, recalling  \eqref{eqn:etaYoung} with $\theta_x = \chi_E(x)$ a.e., by the fundamental theorem of Young measures we obtain
    $$\int_{\Omega} g : \nabla u \, dx = \lim_{n \to \infty} \int_{\Omega} g :\nabla u_n  \, dx = \int_{\Omega}\int_{\R^{d\times d}} g:M \, d \eta_x(M) \, dx = \int_{\Omega} g:(\chi_E A + (1-\chi_E)B)  \, dx.$$
    This proves that $\nabla u = \chi_E A + (1-\chi_E)B$. Thus,  $u\in X(\Omega)$.
    In particular, we infer that $\eta_x = \delta_{\nabla u(x)}$ for almost every $x\in \Omega$. Recalling that $\lbrace \eta_x\rbrace_{x\in \Omega}$ is the gradient Young measure generated by $\nabla u_n$,  again using \eqref{eq:2-integrability} we obtain
    \[
       \int_{\Omega}|\nabla u_n|^2 \, dx \to \int_{\Omega}|\nabla u|^2 \, dx
    \]
    Hence, by uniform convexity of $L^2$, we conclude that $u_n\to u$ strongly in $H^1(\Omega;\R^d)$.    
\end{proof}

\section{The liminf inequality}\label{sec:liminf-interm}
This section is devoted to the proof of the $\Gamma$-liminf inequality. \ We begin by proving an intermediate lower bound: we use  the blow-up method (cf.\ \cite{FonsecaMueller1992}) to show that the surface energy density is controlled from below by a nonlocal and multidimensional cell energy. Afterwards, we prove that this cell energy can be further bounded from below by the local and one-dimensional cell energy $\sigma_{\rm geo}$ defined in \eqref{eqn:geoCellEnergy}. Recall the notation $Q_\nu(x_0,r)$ introduced at the end of Section~\ref{intro}. 

\begin{definition}
\label{def:nonlocal-cell}
Let $W$ satisfy \ref{H1}--\ref{H3}. Let $x_0\in \Omega$ and $\nu\in \mathbb{S}^{d-1}$.
  For $\rho > 0$ and $u\in H^2(Q_\nu(x_0,\rho);\R^d)$,   we   define the rescaled function $$u^{x_0,\rho} (x): = \frac{1}{\rho} u(x_0+\rho x)  \quad \text{for $x \in Q_\nu(0,1)$}.$$
We define the \emph{nonlocal cell-energy density} as
\begin{align}\label{eqn:cellEnergyNonlocal}
    \sigma_{\rm nl} (x_0,\nu) : = &\inf\Big\{ \liminf_{\eps \to 0} \frac{1}{\rho_\eps^{d-1}}\int_{Q_{\nu}(x_0,\rho_\eps)} \left(\frac{1}{\eps}W(x,\nabla u_\eps(x)) + \eps |\nabla^2 u_\eps(x)|^2 \right) dx:  \  \rho_\eps\to 0, \  \eps/\rho_\eps \to 0  \notag\\
&\  \ \ \quad\quad \quad\nabla u^{x_0,\rho_\eps}_\eps \to A(x_0)\chi_{Q_{\nu}^+(0,1)}+B(x_0)\chi_{Q_{\nu}^-(0,1)}\text{ strongly in } {L^1\left(Q_{\nu}(0,1);\R^{d\times d}\right)}\Big\}.
\end{align} 
\end{definition}
 We point out that we call this a `nonlocal' cell energy  because in the integrand we have $W(x,\nabla u_\eps(x))$ and not $W(x_0,\nabla u_\eps(x))$. Freezing the first spatial input effectively `localizes' the cell energy. We will show that, under \ref{H1}--\ref{H3}, we can localize the geodesic problem that appears as a lower bound of the $\Gamma$-liminf, but it is an interesting and challenging question to find more general conditions under which localization is possible.  Note that this density is only relevant in the case that $\nu$ equals the normalized   row vector of $(A - B)(x_0)$ up to a sign, as otherwise  $\sigma_{\rm nl} (x_0,\nu) = + \infty$ due to the compactness and characterization results in Sections \ref{sec:limit-laminates}--\ref{sec:compactness}.  

We state now the first lower bound for the $\Gamma$-liminf.  Recall the energy $E_\eps$ \eqref{eqn:energy}, the space in $X(\Omega)$ \eqref{eq:X}, and \eqref{eqn:cellEnergyNonlocal}.

\begin{proposition}
\label{thm:liminf-inter}
    Let $W$ satisfy \ref{H1}--\ref{H3}. Let $\lbrace u_\eps \rbrace_\eps \subset  H^2(\Omega; \R^d)$ and $u \in X(\Omega)$ be such that $u_\eps \to u$ strongly in $H^1(\Omega; \R^d)$. Then,
    \begin{equation}
    \label{eq:liminf-in}
        \liminf_{\eps \to 0} E_{\eps}(u_\eps) \geq \int_{J_{\nabla u}}\sigma_{\rm nl}(x,\nu(x)) \, d\H^{d-1}(x),
    \end{equation}
    where for $\H^{d-1}$-a.e.\ $x\in J_{\nabla u}$ the unit vector $\nu(x)$ denotes the normal to $J_{\nabla u}$ in $x$  with $(\nabla u)^+(x) = A(x)$.
\end{proposition}
 
\begin{proof}
Although the proof is standard, we include it here for the convenience of the reader.
If the left-hand side of \eqref{eq:liminf-in} is infinite, then there is nothing to prove. Therefore, 
we can assume that, up to the extraction of a (not relabeled) subsequence,
     $\lim_{\eps\to 0} E_{\eps}(u_\eps) = \liminf_{\eps\to 0} E_{\eps}( u_\eps) < \infty $. 
     Setting
    \[
        \mu_\eps(B) := E_{\eps}[ u_\eps, B]
    \]
    for every Borel set $B\subset \Omega$, by compactness we infer the existence of a bounded Radon measure $\mu\in \mathcal{M}_b(\Omega)$ such that $\mu_\eps \weaklystar \mu$. In particular,
    \[
        \mu(\Omega) \leq \lim_{\eps\to 0} \mu_\eps(\Omega).
    \]
   Using the blow-up method \cite{FonsecaMueller1992}, we show that 
    \begin{equation}\label{eqn:blowUp}
    \frac{d\mu}{d\H^{d-1}\mres_{J_{\nabla u}}}(x_0)\geq \sigma_{\rm nl}(x_0,\nu(x_0)) \quad \text{ for $\H^{d-1}$-a.e.\ $x_0\in J_{\nabla u}$},
    \end{equation} from which the thesis will follow.     By the Radon--Nikod\'ym  differentiation theorem  we have
    \begin{equation}\label{eqn:radonNikodym}
        \frac{d\mu}{d\H^{d-1}\mres_{J_{\nabla u}}}(x_0) = \lim_{\rho \to 0^+} \frac{\mu(Q_{\nu(x_0)}(x_0, \rho))}{\H^{d-1}(J_{\nabla u} \cap Q_{\nu(x_0)}(x_0, \rho))} = \lim_{\rho \to 0^+} \frac{\mu(Q_{\nu(x_0)}(x_0, \rho))}{\rho^{d-1}}
    \end{equation}
    at $\H^{d-1}$-almost every $x_0 \in J_{\nabla u}$, where $\nu(x_0)$ is the normal  with $(\nabla u)^+(x_0) = A(x_0)$.   Similarly,   at $\mathcal{H}^{d-1}$-almost every $x_0 \in J_{\nabla u}$, it holds that
    \begin{equation}\label{eqn:BVfineProp}
    \lim_{\rho\to 0}\frac{1}{\rho^d}\int_{Q^+_{\nu(x_0)}(x_0, \rho)} |\nabla u - A(x_0)|\, d x  = 0 \quad \text{ and }\quad  \lim_{\rho\to 0}\frac{1}{\rho^d}\int_{Q^-_{\nu(x_0)}(x_0, \rho)} |\nabla u - {B(x_0)}|\, d x  = 0,
    \end{equation}
    where $Q^\pm_{\nu(x_0)}(x_0, \rho) $ is defined  at the end of Section~\ref{intro}. Choosing a point $x_0$ satisfying both \eqref{eqn:radonNikodym} and \eqref{eqn:BVfineProp}, and selecting 
    a sequence of side lengths $\rho \to 0$ (still denoted by $\rho\to 0$) such that $\mu(\partial Q_{\nu(x_0)}(x_0, \rho)) = 0$ for each $\rho$, 
    %by \cite[Theorem 1.40]{EvansGariepy}  
    we obtain   
    \[
        \frac{d\mu}{d\H^{d-1}\mres_{J_{\nabla u}}}(x_0) =  \lim_{\rho\to 0}\lim_{\eps\to 0} \frac{\mu_\eps(Q_{\nu(x_0)}(x_0, \rho))}{\rho^{d-1}},
    \]
    as well as
    \begin{equation}\nonumber
    \lim_{\rho\to 0}\lim_{\eps\to 0}\frac{1}{\rho^d}\int_{Q^+_{\nu(x_0)}(x_0, \rho)} |\nabla u_\eps(x) - A(x_0)|\, d x  = 0 \quad \text{ and }\quad  \lim_{\rho\to 0}\lim_{\eps\to 0}\frac{1}{\rho^d}\int_{Q^-_{\nu(x_0)}(x_0, \rho)} |\nabla u_\eps(x) - {B(x_0)}|\, d x  = 0,
    \end{equation}
     where for the latter property we have used the assumption that $u_{\eps}\to u$ strongly in $H^1(\Omega;\R^d)$.
    
    Diagonalizing on $\eps$,  we  
    find a subsequence $\rho_\eps\to 0$ such that $\eps/\rho_\eps \to 0$  and 
    \begin{align}\label{5X} 
        \frac{d\mu}{d\H^{d-1}\mres_{J_{\nabla u}}}(x_0) =  \lim_{\eps\to 0} \frac{\mu_\eps( Q_{\nu(x_0)}(x_0, \rho_\eps))}{\rho_\eps^{d-1}},
    \end{align}
    as well as
    \begin{equation}\label{eqn:finePropDiag}
        \lim_{\eps\to 0}\frac{1}{\rho_\eps^d}\int_{Q^+_{\nu(x_0)}(x_0, \rho_\eps)} |\nabla u_\eps(x) - A(x_0)|\, d x  = 0 \quad \text{ and }\quad  \lim_{\eps\to 0}\frac{1}{\rho_\eps^d}\int_{Q^-_{\nu(x_0)}(x_0, \rho_\eps)} |\nabla u_\eps(x) - {B(x_0)}|\, d x  = 0.
    \end{equation}
    Recalling that, by definition,
    $$ 
        \frac{\mu_\eps( Q_{\nu(x_0)}(x_0, \rho_\eps))}{\rho_\eps^{d-1}} = \frac{1}{\rho_\eps^{d-1}}\int_{Q_{\nu(x_0)}(x_0,\rho_\eps)} \left(\frac{1}{\eps}W(x,\nabla u_\eps) + \eps |\nabla^2 u_\eps|^2 \right) dx,
    $$
    \eqref{eqn:cellEnergyNonlocal}, \eqref{5X}--\eqref{eqn:finePropDiag},  and a change of variables yield \eqref{eqn:blowUp}.
\end{proof} 
  
  Extending the notation from immediately before Lemma \ref{geodesic-distance-gradient-estimates-new}, we define the \emph{geodesic distance} between $M, N  \in \R^{d\times d}$ along a path $\psi\in W^{1,1}(  (-1,1);  \Omega)$ as
\begin{align}\label{eqn:geoGeneral}
    &d_W(\psi,M, N) : = \nonumber \\ 
    &\quad \inf \left\{\int_{-1}^{1} {2}\sqrt{W(\psi,\Phi)}|\Phi'|\, dt \colon \,  \Phi \in W^{1,1}((-1,1);\R^{d\times d}) \text{ with } \Phi(-1)=M \text{ and }\Phi(1) = N \right\}. 
\end{align}
Note that for the constant curve $\psi(t)\equiv x_0$ we have  
$$
    \sigma_{\rm geo}(x_0)  = d_W(x_0,A(x_0),B(x_0))  = d_W(\psi,A(x_0),B(x_0)). 
$$
We proceed with two auxiliary statements for the truncated potential 
\begin{align}\label{Wtrunc-neu}
    \WtruncLI :=\min \{W, C_0\}
\end{align}
with $C_0 > 0$.
\begin{proposition}\label{comparision-geodesic-distances}
    Let $W$ satisfy \ref{H1}--\ref{H3}. Let $C_0>0$ and define $\WtruncLI$ as in \eqref{Wtrunc-neu}. Let $\psi_1, \psi_2 \in W^{1,1}((-1,1); \R^d)$ and $M, N \in \R^{d \times d}$. Then, 
    \begin{align}\nonumber
        d_{\WtruncLI}(\psi_1, M, N) \leq  d_{\WtruncLI}( \psi_2, M, N) + C\|\psi_1 -  \psi_2\|_{L^\infty((-1,1))}(1 + \max\{|M|, |N|\})
    \end{align}
   for a constant $C = C(A, B, \|\psi'_2\|_{L^1(-1,1)} , C_0) >0$.
\end{proposition}

\begin{lemma}\label{lemma:trunction_does_not_matter}
    Let $W$ satisfy \ref{H1}--\ref{H3}. Let $C_0>0$  and define $\WtruncLI$ as in \eqref{Wtrunc-neu}. Then, for $C_0$ large enough, it holds that 
    \[
        d_{W}(x, A(x), B(x)) = d_{\WtruncLI}(x, A(x), B(x))
    \]
    for all $x \in \Omega$.
\end{lemma}
We postpone the proof of these two  results to Section \ref{sec:geodesics-disc}. The purpose of the following result is to show that the nonlocal cell-energy density in \eqref{eqn:cellEnergyNonlocal} can be bounded from below by a localized cell energy. 
% Recall the notion of the \emph{geodesic cell-energy density} $\sigma_{\rm geo}$ in Definition \ref{def:geo-cell-en}.
\begin{proposition}\label{thm:cellEnergyGeoBound}
Let  $\sigma_{\rm nl}$ be defined as in \eqref{eqn:cellEnergyNonlocal} and $\sigma_{\rm geo}$ as in \eqref{eqn:geoCellEnergy}. Let $W$ satisfy \ref{H1}--\ref{H3}.  Then, for every $x_0\in \Omega$,  it holds that 
\[
    \sigma_{\rm nl}   (x_0, \pm \nu(x_0))\geq \sigma_{\rm geo}(x_0)
\]
where $\nu(x_0)$ is the normalized  row vector of $(A - B)(x_0)$. 
\end{proposition}
\begin{proof}
Fix $x_0\in \Omega$. Let $\eps\to 0$ and $\rho_\eps\to 0 $ with $\eps/\rho_\eps \to 0$, and consider a sequence  $\lbrace u_\eps\rbrace_\eps$  with $u_\eps \in H^2(Q_{\nu(x_0)}(x_0,\rho_\eps); \R^d)$  such that 
\begin{align}\label{eq:L1_convergence}
    \nabla u^{x_0,\rho_\eps}_\eps \to A(x_0)\chi_{Q_{\nu(x_0)}^+(0,1)}+B(x_0)\chi_{Q_{\nu(x_0)}^-(0,1)}
\end{align}
strongly in ${L^1(Q_{\nu(x_0)}(0,1);\R^{d\times d})}$, where we have adopted the notation of Definition \ref{def:nonlocal-cell}.  We prove that
\begin{equation}\label{eqn:liminfGeo}
\liminf_{\eps \to 0} \frac{1}{\rho_\eps^{d-1}}\int_{Q_{\nu(x_0)}(x_0,\rho_\eps)} \left(\frac{1}{\eps}W(x,\nabla u_\eps) + \eps |\nabla^2 u_\eps|^2 \right) dx \geq \sigma_{\rm geo}(x_0) ,
\end{equation} from which the statement follows.  Without loss of generality, we  assume that 
$$
    x_0 = 0 \qquad \text{ and } \qquad \nu(x_0) = e_d.
$$
Moreover, we adopt the notation $Q_{e_d}(0,1)=Q(0,1)=Q'(0,1)\times \left(-\tfrac12,\tfrac12\right)$ and write $u_\eps^{\rho_\eps} = u_\eps^{x_0,\rho_\eps}$. 
  In particular, by applying a  standard  slicing argument to \eqref{eq:L1_convergence}, we find  some $\tau \in (1/4,1/2)$ such that   
\begin{equation}\label{eqn:traceConvergence}
\int_{Q'(0,1)\times \{\tau\}} |\nabla u_\eps^{\rho_\eps} - A(x_0)|\, d\mathcal{H}^{d-1} + \int_{Q'(0,1)\times \{-\tau\}} |\nabla u_\eps^{\rho_\eps} - B(x_0)|\, d\mathcal{H}^{d-1} \to 0.
\end{equation}
For every $y'\in Q'(0,1)$ and $s \in (-1/2,1/2)$ we define 
\begin{equation}\label{eqn:traceConvergence2}
    \Phi_\eps^{y', \rho_\eps}(s) := \nabla u_\eps (\rho_\eps y',\rho_\eps s),
\end{equation}
and remark that by  slicing and trace properties of Sobolev functions and functions of bounded variations (cf.\ \cite[Section 3.11]{ambrosio2000fbv})   we get $\Phi_\eps^{y', \rho_\eps} \in W^{1,1}((-1/2,1/2);  \R^{d \times d})$ with 
\begin{equation}\label{eqn:traceBlowRel}
    \Phi_\eps^{y', \rho_\eps}(\pm \tau) = \nabla u_\eps^{\rho_\eps}(y', \pm \tau)
\end{equation}
for $\mathcal{H}^{d-1}$-a.e.\   $y' \in Q'(0,1)$. 
We also  note that
$$ \frac{d}{ds}\Phi_\eps^{y', \rho_\eps}(s) = \rho_\eps \nabla (\partial_n u_\eps) (\rho_\eps y', \rho_\eps s).  
$$
Now, let $C_0>0$  and define $\WtruncLI$ as in \eqref{Wtrunc-neu}.  If the truncation constant $C_0 > 0$ is large enough, the statement of Lemma \ref{lemma:trunction_does_not_matter} holds. Thus, by Young's inequality, Fubini's theorem, and a change of variables  we estimate
\begin{align}
    \nonumber&\frac{1}{\rho_\eps^{d-1}}\int_{Q(0,\rho_\eps)} \left(\frac{1}{\eps}W(y,\nabla u_\eps) + \eps |\nabla^2 u_\eps|^2 \right) \, dy \\ 
     & \nonumber \hspace{4em} \geq \frac{1}{\rho_\eps^{d}}\int_{Q(0,\rho_\eps)} 2\sqrt{W(y,\nabla u_\eps)}|\rho_\eps\nabla^2 u_\eps| \, dy \\
    \nonumber& \hspace{4em} \geq \int_{Q(0,1)} 2\sqrt{W(\rho_\eps y,\Phi_\eps^{y', \rho_\eps}(y_d) )} \left| \tfrac{d}{ds}\Phi_\eps^{y', \rho_\eps}  (y_d)\right|  \, d(y', y_d) \\ 
     &\label{eq:est1_LI} \hspace{4em} \geq  \int_{Q'(0,1)}\int_{-\tau}^{\tau} 2\sqrt{\WtruncLI(\rho_\eps y,\Phi_\eps^{y', \rho_\eps}(y_d))}\left| \tfrac{d}{ds}\Phi_\eps^{y', \rho_\eps}  (y_d)\right| \, dy_d \, dy'.
\end{align}
Now, defining 
$$
    {  \psi^{y', \rho_\eps} (s):= \rho_\eps (y', \tau s ) \quad  \text{ for $s\in (-1,1)$, }  }
$$
using the definition of $d_{\WtruncLI}$ in \eqref{eqn:geoGeneral} for every point $y' \in Q'(0,1)$ (with $\Phi_\eps^{y', \rho_\eps}$ in place of $\Phi$), and applying the change of variable $y_d = \tau \tilde y_d$ we infer
\begin{align}
    \int_{Q'(0,1)}\int_{-\tau}^{\tau} 2\sqrt{\WtruncLI(\rho_\eps y,\Phi_\eps^{y', \rho_\eps}(y_d))}| \tfrac{d}{ds}\Phi_\eps^{y', \rho_\eps}  (y_d)| \, dy_d \, dy'  \label{eq:est2_LI} \geq \int_{Q'(0,1)} d_{\WtruncLI} (  \psi^{y', \rho_\eps} , \Phi_\eps^{y', \rho_\eps} (\tau) , \Phi_\eps^{y', \rho_\eps} (-\tau) )\, d y'.
\end{align}
To obtain the  statement, we  need  to pass to the limit on the right-hand side of \eqref{eq:est2_LI} as $\eps\to 0$. 
Precisely,  using the definition of $\psi^{y', \rho_\eps}$, we first apply Proposition \ref{comparision-geodesic-distances} to infer 
\begin{align}\label{eqn:liminf_inequ_1}
    & d_{\WtruncLI} ( \psi^{y', \rho_\eps}, \Phi_\eps^{y', \rho_\eps}(\tau) , \Phi_\eps^{y', \rho_\eps}(- \tau)) \nonumber \\
    & \hspace{2cm} \geq  d_{\WtruncLI} (x_0 , \Phi_\eps^{y', \rho_\eps}(\tau) , \Phi_\eps^{y', \rho_\eps}(-\tau))  - C\rho_\eps\big(1 + \max\{|\Phi_\eps^{y', \rho_\eps}(- \tau))|, |\Phi_\eps^{y', \rho_\eps}(\tau)|\}\big).  
\end{align}
Since 
\begin{align}\label{startar}
    \int_{Q'(0,1)}|\Phi_\eps^{y', \rho_\eps}(\pm \tau))| \, dy' \leq C\big(1 + |A(x_0)| + |B(x_0)| \big)
\end{align}
by \eqref{eqn:traceConvergence}--\eqref{eqn:traceConvergence2},  the integral over the last term in \eqref{eqn:liminf_inequ_1} vanishes as $\eps \to 0$. Now, as
\[
    (M, N) \mapsto d_{\WtruncLI}(x_0, M, N)
\]
is Lipschitz, see  Lemma \ref{geodesic-distance-gradient-estimates-new},   and as by \eqref{eqn:traceConvergence}  and \eqref{eqn:traceBlowRel} we have
\[
    (\Phi_\eps^{y', \rho_\eps}(\tau) , \Phi_\eps^{y', \rho_\eps}(- \tau)) \to (A(x_0), B(x_0))
\]
in $L^1(Q'(0,1); (\R^{d\times d})^2)$,  we obtain by dominated convergence
\[
   \lim_{\eps \to 0}  \int_{Q'(0,1)} d_{\WtruncLI} (x_0, \Phi_\eps^{y', \rho_\eps}(\tau) , \Phi_\eps^{y', \rho_\eps}(- \tau)) \, dy' = \int_{Q'(0,1)} d_{\WtruncLI} (x_0, A(x_0) , B(x_0)) \, dy'.
\]
Since Lemma \ref{lemma:trunction_does_not_matter} holds by our choice of $C_0$, we therefore derive
\[
    \lim_{\eps \to 0}  \int_{Q'(0,1)} d_{\WtruncLI} ( x_0 , \Phi_\eps^{y', \rho_\eps}(\tau) , \Phi_\eps^{y', \rho_\eps}(- \tau)) \, dy' = \sigma_{\rm geo}(x_0).
\]
In conjunction with \eqref{eq:est1_LI}, \eqref{eq:est2_LI}, \eqref{eqn:liminf_inequ_1}, and \eqref{startar}, we infer \eqref{eqn:liminfGeo}.
\end{proof}
 We close this section by noting that the combination of  Propositions \ref{thm:liminf-inter} and \ref{thm:cellEnergyGeoBound} yields the $\Gamma$-liminf inequality in Theorem \ref{thm:main}. 

     \section{Construction of   recovery sequences}
    \label{recoverysequnecesection}
     \label{sec:limsup}
 
    In this section, we prove that,   under the assumption \ref{H4}, the lower bound identified in Section \ref{sec:liminf-interm} is optimal. Throughout this section, we assume that $W$ satisfies \ref{H1}--\ref{H4} if not specified otherwise. We remind the reader of the discussion in Remarks \ref{remark:reduction_of_wells}--\ref{remark:diffeomorphism_construction} corresponding to \ref{H4}.   In particular, we recall that we simplify to the case  
    \[
        A = -B = (\kappa \circ \gamma) \otimes \nabla \gamma
    \]
    with $\kappa  \in C^1(\R; \R^d \setminus \lbrace 0 \rbrace )$ and $\gamma \in  C^3  ( \R^d)$ (cf.\  Remark \ref{remark:reduction_of_wells}). We set 
    \begin{align}\label{nununu}
        \nu  :=  \frac{\nabla \gamma}{|\nabla \gamma|}.
    \end{align}
    We call the $C^2$-diffeomorphism $F \in C^2(\R^d; \R^d)$ provided by Remark \ref{remark:diffeomorphism_construction} the \emph{diffeomorphism  associated   to $\gamma$}. For a discussion on this diffeomorphism, we refer the reader to Remark \ref{remark:reduction_of_wells}(b) and Remark~\ref{remark:diffeomorphism_construction}(a)--(b).
    
    The main idea for constructing a recovery sequence is to use the diffeomorphism induced by the rank-one connection to flatten the interfaces, construct recovery sequences locally on the flattened domain, and then glue them together. For this procedure, in Subsection \ref{sec71}, we first introduce notation for the pushforward of the  involved quantities. Then, we introduce some auxiliary lemmas which essentially state that the construction of recovery sequences on the diffeomorphed domain with respect to these diffeomorphed quantities is equivalent to the construction of a recovery sequence on the non-diffeomorphed domain. Then, we discuss geodesics and the interaction of \ref{H4}(ii) with the diffeomorphism. Afterwards, in Subsection \ref{sec72}, we come to the construction of recovery sequences for flat interfaces, first on cylinders contained in the diffeomorphed domain and then for general domains. 

    In this section, we will use the variable $x$ for elements of the diffeomorphized domain $F(\Omega)$ and the variable $y$ for elements of $\Omega$, usually subject to the relation $F^{-1}(x) = y$. Moreover,  for point evaluation we use notation of the type
    \[
        (\nabla F)_y := \nabla F(y),  \ \ \  A_x^F:= A^F(x), \ \ \ \text{etc.} 
    \]
    when it improves readability.

    \subsection{Flattening of interfaces by diffeomorphism}\label{sec71}
    
    We start with some definitions.
    \begin{definition}\label{def:diffeomorphed_quantities}
        Let $F \colon \R^d \to \R^d$ be the $C^2$-diffeomorphism associated to $\gamma$. For any $x \in F(\Omega)$ and $y \in \Omega$ with $y = F^{-1}(x)$ we define the pushforward of $W$ via $F$ by
        \begin{align*}
            W_F(x ,M) := W\big(y, M(\nabla F)_y \big) \big|\det (\nabla F^{-1})_x\big|
        \end{align*}
        for $M \in \R^{d \times d}$. Furthermore, we set 
        \begin{align*}
            |\mathfrak T|_{F,x} := |(\nabla F)^T_y \mathfrak T (\nabla F)_y| \left|\det (\nabla F^{-1})_x \right|^{1/2}
        \end{align*}
        for any $\mathfrak T \in \R^{d \times d \times d}$,  where  $\mathfrak{T}_i = (\mathfrak T_{ijk})_{j,k = 1}^d$  denotes  the $d\times d$ matrix with  fixed  first index    and 
        \[
            |(\nabla F)^T_y \mathfrak T (\nabla F)_y|^2 :=  \sum_{i = 1}^d |(\nabla F)^T_y \mathfrak{T}_i (\nabla F)_y|^2.
        \] 
        We write $|\mathfrak T|_{F}$ if it is clear at which point this norm is evaluated. Moreover,  we denote the geodesic distance adapted to the diffeomorphism by
        \begin{align*}
            d_{W}^F(x, M, N) := |(\nabla F)_y^T e_d|\left|\det \nabla (F^{-1})_x \right|^{1/2}\inf \bigg\{ & \int_{-1}^1 \sqrt{W_F(y, \Phi)}|(\Phi')(\nabla F)_y| \, ds : \\
            &\Phi \in W^{1,1}((-1,1);\, \R^{d \times d}), \, \Phi(-1) = M, \, \Phi(1) = N \bigg\}
        \end{align*}
        for matrices $M, N \in \R^{d \times d}$. Lastly, we define the pushforward of the wells via $F$ by 
        \[
            \alpha^F := \alpha \circ F^{-1} \qquad \text{ and }  \qquad \beta^F := \beta \circ F^{-1}
        \]
        and
        \[
            A^F := (A \circ F^{-1})\nabla F^{-1} \qquad \text{ and } \qquad B^F := (B \circ F^{-1})\nabla F^{-1}.
        \]
    
    \end{definition}
    First, we prove that the properties \ref{H1} and \ref{H2} are preserved under the diffeomorphism $F$ associated to $\gamma$. 
    \begin{lemma}\label{lemma:diffeomorphed_quantities_are_comparable} 
        Let $W$ satisfy \ref{H1} and \ref{H2}. Let $F\colon \R^d \to \R^d$ be the $C^2$-diffeomorphism associated to $\gamma$. Then $W_F$ also satisfies \ref{H1} and \ref{H2} with wells $\alpha^F$ and $\beta^F$, and there exists a constant $C_F > 0$ such that
        \[
            \frac 1 {C_F} |\mathfrak T| \leq |\mathfrak T|_F \leq C_F |\mathfrak T| \quad  \text{for all  $\mathfrak{T} \in \R^{d \times d \times d}$}. 
        \]
    \end{lemma}
    \begin{proof}
        By \ref{H2} we   have
        \begin{align*}
            W_F(x, M) &= W(y, M(\nabla F)_y) |\det (\nabla F^{-1})_x| \\
            &\geq C_*^{-1}  \min \{|M (\nabla F)_y - A_y|^{ 2},|M (\nabla F)_y - A_y|^{ 2} \} \left|\det  (\nabla F^{-1})_x  \right| \\
            & \geq  C_*^{-1}   \min \{|M  - A^F_x|^{ 2},|M - A^F_x|^{ 2} \} \inf_{z \in F(\Omega)}\frac{|\det (\nabla F^{-1})_z|}{|(\nabla F^{-1})_z|^{ 2}}.
        \end{align*}
      Now, we observe
        \[
            \inf_{z \in F(\Omega)}\frac{\left|\det (\nabla F^{-1})_z\right|}{|(\nabla F^{-1})_z|^{ 2}} > 0
        \]
        since $\Omega$ is a bounded domain and $F$ is a ($C^2$-)diffeomorphism on $\R^d$. Similarly, we have
        \[
            W_F(x, M) \leq C_F\min \{|M  - A^F_x|^{ 2},|M - A^F_x|^{2} \}
        \]
        for some $C_F >0$. This shows that $W_F$ satisfies \ref{H2}.
        By analogous reasoning, we have 
        \[
            \frac 1 {C_F} |\mathfrak T| \leq |\mathfrak T|_F \leq C_F |\mathfrak T|
        \]
        for all $\mathfrak T \in \R^{d \times d \times d}$ for some constant $C_F > 0$. 
    \end{proof}    
    The symmetry assumption 
    \[
        A = -B = (\kappa \circ \gamma) \otimes \nabla \gamma
    \]
    transfers to the push-forward of the wells. More precisely, we have that 
    \begin{align}\label{symmetriewell}
        A^F(x) = -B^F(x) = \tilde \kappa(x_d) \otimes e_d 
    \end{align}
    for a suitable function $\tilde \kappa \in C^1(\R; \R^d \setminus \lbrace 0 \rbrace )$ (cf.\ Remark \ref{remark:diffeomorphism_construction} and the computations in the proof of Lemma~\ref{lem:rep}). Throughout the proof of the $\limsup$ inequality only $\tilde{\kappa}$, but not $\kappa$, is needed; for notational simplicity and by an abuse of notation, we write $\kappa $ in place of $ \tilde \kappa $ for the remainder of this section. 

    Now, we observe that the geodesic distance along the wells adapted to the diffeomorphed setting is just the pushforward of $\sigma_{\rm geo}$ (cf.\ Definition \ref{def:geo-cell-en}) times a surface density. Recall the definition of $\nu$ in \eqref{nununu}.
    \begin{lemma}\label{lemma:relation_of_pushforward_of_geodesic_distance}
        Let $F$ be the $C^2$-diffeomorphism associated to $\gamma$.  Let $x \in \R^d$ and $y \in \Omega$ with $F^{-1}(x) = y$. Then, we have
        \[
            d_{W}^F(x, A^F_x, B^F_x) = \sigma_{\rm geo}(y)|(\nabla F)_y^T e_d| \big|\det (\nabla F^{-1})_x\big|.
        \]
          \end{lemma}

    \begin{proof}
        Let $\Phi \in W^{1,1}( (-1,1); \R^{d \times d})$ be such that $\Phi(-1) = A^F_x$ and $\Phi(1) = B^F_x$. Then, for 
        \[
            \Psi := \Phi (\nabla F)_y
        \]
        we obtain
        \[
            \int_{-1}^1 \sqrt{W_F(x, \Phi)}|\Phi'(\nabla F)_y| \, ds =  \int_{-1}^1 \sqrt{W(y, \, \Psi)}|\Psi'|\, ds \left|\det \nabla (F^{-1})_x \right|^{1/2}.
        \]
        Note that $\Psi (-1) = A_y$ and $ \Psi (1) = B_y$. Since $\Phi \mapsto \Phi (\nabla F)_y$ defines a bijection between 
        \[
            \big\{ \Phi \in  W^{1,1}((-1,1); \R^{d \times d}) \colon \, \Phi(-1) =  A^F_x, \, \Phi(1) = B^F_x \big\}
        \]
        and 
        \[
            \big\{ \Psi \in  W^{1,1}((-1,1); \R^{d \times d}) \colon \,  \Psi(-1) =  A_y, \, \Psi(1) = B_y \big\},
        \]
        we infer that the infima over the respective sets coincide, i.e., 
        \[
            d_{W}^F(x, A^F_x, B^F_x) = d_{W}(y, A(y), B(y))|(\nabla F)_y^T e_d| \big|\det (\nabla F^{-1})_x\big|.
        \]
        Thus, the proof is concluded.  
    \end{proof}
    The next lemma shows that the energies $E_\eps$ behave well with respect to diffeomorphisms. Essentially, the recovery sequence with respect to $E_\eps$ can be transferred to recovery sequences for the diffeomorphed functional.
    \begin{lemma}\label{lem:equivalence_of_gamma_limits_under_diffeomorphism}
        Let $F$ be the $C^2$-diffeomorphism associated to $\gamma$ and consider a sequence $\eps \to 0$. Let $\lbrace u_\eps  \rbrace_{\eps>0} \subset  H^2(\Omega; \R^d)$ and $u \in X(\Omega)$ be  such that $u_\eps \to u$ in $L^1(\Omega; \R^d)$. Moreover, define the pushforward of $u_\eps$ and $u$ via $F$ by $v_\eps := u_\eps  \circ F^{-1}$ and $v := u \circ F^{-1}$. Then, we have 
        \[
           E_{\eps}(u_\eps) \to \int_{J_{\nabla {u}}} \sigma_{\rm geo}(y) \, d\H^{d-1}(y)
        \]
        if and only if 
        \[
            \int_{F(\Omega)} \frac{1}{\eps} W_F(x, \nabla v_\eps) + \eps |\nabla^2 v_\eps|_F^2 \, dx \to \int_{J_{\nabla v}}d_{W}^F( x,  A^F_x, B^F_x) \, d\H^{d-1}(x). 
        \]
    \end{lemma}
 
    \begin{proof}
        It is not restrictive to assume that the sequence $\lbrace u_\eps \rbrace_\eps $ is bounded in $H^1(\Omega;\R^d)$, cf.\ \ref{H2}. By a change of variables, we get 
        \begin{equation}\label{eqn:energyRewrite}
            E_{\eps}(u_\eps)  = \int_{F(\Omega)} \frac{1}{\eps} W_F(x, \nabla v_\eps) + \eps |\nabla^2 v_\eps|_F^2 \, dx + O(\eps), 
        \end{equation}
        where the $O(\eps)$ term comes from an application of the product rule to the Hessian. Observe  that 
        \[
            J_{\nabla v} = J_{\nabla u \circ F^{-1}} = F(J_{\nabla u})
        \]
        by a straightforward computation. In particular, applying the transformation rule for surfaces/rectifiable sets (cf.\ for instance \cite[Theorem 11.6 \& Proposition 17.1]{maggi_2012}) we infer
        \begin{align*}
            \int_{J_{\nabla v}} d_{W}^F  (x, A^F_x, B^F_x) \, d\H^{d-1}(x) &= \int_{F(J_{\nabla u})} d_{W}^F (
            x, A^F_x, B^F_x) \, d\H^{d-1}(x) \\
            &= \int_{J_{\nabla u}} d_{W}^F(F(y), A^F_{F(y)}, B^F_{F(y)})|\det(\nabla F)_y||(\nabla F^{-1})^{T}_{F(y)}{\nu_{J_u}}| \, d\H^{d-1}(y).
        \end{align*}
        By \ref{H4}, \eqref{symmetriewell}, and Theorem \ref{thm:characterisation-of-piecewise-BV-space}, we have that the normal of $J_{\nabla v}$ is equal to $\pm e_d$.
        It then follows from \cite[Proposition 17.1]{maggi_2012} (in the reverse direction, starting from $J_{\nabla u} = F^{-1}(J_{\nabla v})$) that 
        \begin{equation}\label{eqn:nuedRelation}
            \nu_{J_u}(y) =  \pm \frac{(\nabla F)_{y}^{T}e_d}{|(\nabla F)_{y}^{T}e_d|}, 
        \end{equation} 
        which applying  $(\nabla F^{-1})_{F(y)}^{T}$ to both sides gives $$ |(\nabla F^{-1})^{T}_{F(y)}{\nu_{J_u}} (y)  | = \frac{1}{|(\nabla F)_{y}^{T}e_d|}.$$
        Then, tying together the above equalities  with Lemma \ref{lemma:relation_of_pushforward_of_geodesic_distance}, we derive  
        \begin{align*}
            & \int_{J_{\nabla u}} d_{W}^F (F(y), A^F_{F(y)}, B^F_{F(y)})|\det (\nabla F)_y||(\nabla F^{-1})_{F(y)}^T\nu_{J_u}(y) | \, d\H^{d-1}(y)  = \int_{J_{\nabla u}}d_W(y, A_y, B_y) \, d\H^{d-1}(y).
        \end{align*}
        Given the equality determined by the above integral relations and \eqref{eqn:energyRewrite}, we see that one energetic limit determines the other as desired, concluding the proof. 
    \end{proof}

    The previous lemma justifies the following definitions.
    \begin{definition}\label{def:diffeomorphed_surface_density}
        We define the diffeomorphed energy for $v \in H^2(\Omega; \R^d)$ and $\eps > 0$ on a Borel set $\mathcal{U} \subset F(\Omega)$ by 
        \[
            E_{\eps}^F[v, \mathcal{U} ] := \int_{ \mathcal{U} } \frac{1}{\eps} W_F(x, \nabla v) + \eps  |\nabla^2 v|_F^2 \, dx.
        \]
        Furthermore, for $x \in F(\Omega)$ with $y = F^{-1}(x)$, we denote the interfacial density by 
        \[
            \sigma_{\rm geo}^F(x) :=    d_{W}^F  (x, A^F_x, B^F_x).
        \]
    \end{definition}
    Our general idea is to use the diffeomorphism associated to $\gamma$ to reduce the problem to flat interfaces as in \cite[Section 5]{contifonsecaleoni2002}. By means of \ref{H4}(ii), this consists in reducing to curves in  $\R^{d \times d}$ which one glues together suitably along the interface. However, in our setting, there are two important adaptations necessary. First, \ref{H4}(ii) is not preserved under the diffeomorphism and, secondly, the construction of \cite{contifonsecaleoni2002} needs to be localized on cubes together with a partition of unity argument due to the spatial dependence of the potential $W$ and its wells.  To address these problems, we state a variety of lemmas meant to understand the (near) optimal paths. In particular, Lemma \ref{lemma7.7} states that \ref{H4}(ii)  reduces the geodesic distance between the wells to curves in $\R^d \otimes \nu$, as it has already been observed in \cite[Proposition 5.3]{contifonsecaleoni2002}.
    Further, we discuss a uniform bound on the path length in Lemma \ref{lemma:rescaling_of_nearly_geodesics}, in a similiar spirit to  \cite{CristoferiGravina2021}. Finally,  in Lemma \ref{lemma:behavior_of_geodesics_under_diffeomorphisms}--\ref{lem:near_optimal_geodesic_diffeomorphed_domain} we show that these results also apply in the diffeomorphed domain.
    \begin{lemma}\label{lemma7.7}
        Let $V \in C(\R^{d \times d}; [0, \infty))$ and $\nu \in \mathbb{S}^{d-1} $ such that
        \[
            V(M) \geq V( M\nu \otimes \nu)
        \]
        for all $M \in \R^{d \times d}$. Then, for any $M, N \in (\R^d \otimes \nu)$ we have         \begin{align}\label{eq:geodesic_distance reduction}
            & \inf\left\{ 2\int_{-1}^1 \sqrt{V(\Phi)}|\Phi'|\, ds \colon \,  \Phi \in W^{1,1}((-1,1); \R^{d \times d}), \, \varphi(-1) = M\nu \otimes \nu, \, \varphi(1) = N\nu \otimes \nu \right\} \nonumber  \\
            & \ \ \geq\inf\left\{ 2\int_{-1}^1 \sqrt{V(\varphi \otimes \nu)}|\varphi' \otimes \nu|\, ds \colon \,  \varphi \in W^{1,1}((-1,1); \R^{d}), \, \varphi(-1) = M\nu, \, \varphi(1) = N\nu \right\}.
        \end{align}
    \end{lemma}
    \begin{proof}
        For any $\Phi \in W^{1,1}((-1,1); \R^{d \times d})$ we have for $\varphi := \Phi \nu$
        \[
            V(\Phi) \geq V( \varphi  \otimes \nu) \qquad \text{ and } \qquad |\Phi'| \geq  \varphi'  | = | \varphi'   \otimes \nu|.
        \]
        Taking the infima over the corresponding sets yields \eqref{eq:geodesic_distance reduction}. 
    \end{proof}
     
    We proceed with a lemma on specific nearly-optimal paths. 
    \begin{lemma}\label{lemma:rescaling_of_nearly_geodesics}
        Let $W$ satisfy \ref{H1}--\ref{H4}. Then, there exists a constant $C > 0$ with the following property: For any $\lambda > 0$ there exists a constant $K_\lambda > 0$ such that for any $\eps \in (0, K_\lambda^{-1})$ and any $ y_0 \in \Omega$ there exists a curve $\Phi \in  W^{1,1}((-1,1); \R^{d \times d})$ with $\Phi(s) = A(y_0)$ and $\Phi(-s) = B(y_0)$ for all $s \in (\eps K_\lambda, 1)$ such that 
        \[
            \int_{-1}^1 \left( \frac{1}{\eps} W(y_0, \Phi) + \eps |\Phi'|^2  \right) \, ds \leq d_W(y_0, A(y_0), B(y_0)) + C\lambda  
        \]
        and 
        \begin{align}\label{boundonthelength}
            \int_{-1}^1 |\Phi'| \, ds \leq C.
        \end{align}
    \end{lemma}
    We defer the proof of this lemma to Section \ref{sec:geodesics-disc}. The uniform bound on the path length with respect to $y_0$ in Lemma \ref{lemma:rescaling_of_nearly_geodesics} then guarantees that certain estimates in the construction of the recovery sequence later on, which are dependent on the spatial discretization we use, are uniformly small. 
    
    To obtain nearly-optimal paths in the diffeomorphed domain, we must push them forward from the original domain via the diffeomorphism associated to $\gamma$. We will use Lemma \ref{lemma:rescaling_of_nearly_geodesics} in conjunction with the next statement, which tells us that we can transfer geodesics under \ref{H4} into ones with respect to the adapted geodesic distance.

    \begin{lemma}\label{lemma:behavior_of_geodesics_under_diffeomorphisms}
        Let $F$ be the $C^2$-diffeomorphism associated to $\gamma$ and $C > 0$. Then, there exists $\tilde C = \tilde C(F, C) > 0$ with the following property: if for $\varphi \in W^{1,1}((-1,1); \R^{d})$, $\lambda > 0$, $y \in \Omega$,
        and 
        \begin{align}\label{speznudef}
            \nu := \frac{(\nabla F)_{y}^{T}e_d}{|(\nabla F)_{y}^{T}e_d|}
      \end{align}
        we have
        \[
            \int_{-1}^1 \Big( \frac{1}{\eps} W(y, \varphi \otimes \nu) + \eps |\varphi' \otimes \nu|^2\Big) \, ds  \leq d_W(y, A(y), B(y)) + C\lambda,
        \]
        then for $x = F(y)$, $\tilde \eps = |(\nabla F)_{y}^{T} e_d|^{-1} \eps$ and $\psi = \varphi |(\nabla F)_{y}^{T} e_d|^{-1}$ it follows that
        \[
            \int_{-1}^1 \Big(\frac{1}{\tilde \eps}W_F( x, \psi \otimes e_d) + \tilde \eps |\psi' \otimes e_d \otimes e_d|^2_{F,x} \Big)\, ds \leq d_W^F( x,   A^F_x, B^F_x )  + \tilde C\lambda.
        \]
    \end{lemma}
    We emphasize that $\nu$ in \eqref{speznudef} naturally appears as $\nu = \nu_{J_{\nabla u}} (y)$ due to the relation \eqref{eqn:nuedRelation}. 
    
    \begin{proof}
    In view of Definition \ref{def:diffeomorphed_quantities}  and \eqref{speznudef}, by a direct computation, we first note that 
        \[
            W_F( x, \psi \otimes e_d) = W(y, \varphi \otimes \nu) \big|\det (\nabla F^{-1})_x\big| 
        \]
        and 
        \[
            |\psi' \otimes e_d \otimes e_d|^2_{F,x} = |(\nabla F^{T})_y e_d|^2|\varphi' \otimes \nu|^2 \big|\det (\nabla F^{-1})_x\big|.
        \]
        Thus, we observe 
        \begin{align*}
            &\int_{-1}^1 \Big(\frac{1}{\tilde \eps} W_F( x,  \psi \otimes e_d) + \tilde \eps |\psi' \otimes e_d \otimes e_d|^2_{F, x} \Big)\, ds \\
            &\hspace{4em}= \left( \int_{-1}^1 \frac{1}{\eps} W( y, \varphi \otimes \nu) + \eps |\varphi' \otimes \nu|^2 \, ds \right) \big|\det (\nabla F^{-1})_x\big||(\nabla F^{T})_y e_d| \\
            &\hspace{4em}\leq d_W(x, A(x), B(x)) \big|\det (\nabla F^{-1})_x\big| |(\nabla F^{T})_y e_d| + \tilde C\lambda,
        \end{align*}
        where $\tilde C = \|\nabla F^{-1}\|^d_{C(F(\Omega))}\|\nabla F\|_{C(\Omega)}C$. This concludes the proof by Lemma \ref{lemma:relation_of_pushforward_of_geodesic_distance}.  
    \end{proof}

    We summarize the previous findings of Lemmas \ref{lemma7.7}--\ref{lemma:behavior_of_geodesics_under_diffeomorphisms} in the following lemma. We use $\zeta$ in place of $\psi$ to be consistent with notation used in later sections. We also recall $\sigma_{\rm geo}^F$ in  Definition \ref{def:diffeomorphed_surface_density}. 

    \begin{lemma}\label{lem:near_optimal_geodesic_diffeomorphed_domain}
        Let $W$ satisfy \ref{H1}--\ref{H4}. Let $F$ be the $C^2$-diffeomorphism associated with $\gamma$.  Then, there exists a constant $C > 0$ with the following property: for any $\lambda > 0$ there exists a constant $K_\lambda > 0$ such that for any $\eps \in (0, K_\lambda^{-1})$ and any $x \in F(\Omega)$ there exists a curve $ \zeta \in W^{1,1}((-1,1);\R^d)$  with $ \zeta (s) = \kappa(x_d)$ and $\zeta(-s) = - \kappa (x_d)$ for all $s\in (\eps K_\lambda, 1)$ such that
        \begin{align}\label{eq:near_optimality_condition_of_pushforward_of_geodesic}
            \int_{-1}^1 \Big(\frac{1}{\eps} W_F(x, \zeta \otimes e_d) + \eps |\zeta ' \otimes e_d \otimes e_d |^2_{F, x} \Big) \, ds \leq 
            \sigma_{\rm geo}^F(x) + C\lambda
        \end{align}
        and         \begin{align}\label{eq:uniform_bound_path_length_pushforward_of_geodesic}
            \int_{-1}^1 |\zeta'| \, ds \leq C.
        \end{align}
    \end{lemma}

\subsection{Construction of recovery sequences}\label{sec72}
  We now come to the construction of the recovery sequence. We start by treating  cylinders contained in the diffeomorphed domain $F(\Omega)$. Recall the energy $E_\eps^F$ and the surface energy $\sigma_{\rm geo}^F$ given  in Definition \ref{def:diffeomorphed_surface_density}.
  
    \begin{theorem} \label{thm:first-limsup}
        Let $W$ satisfy \ref{H1}--\ref{H4}. Let $F\colon \R^d \to \R^d$ be the $C^2$-diffeomorphism associated with $\gamma$, and consider  
        $$
        \omega_\delta :=    \omega \times (h - \delta, h + \delta)   \subset F(\Omega)
        $$ 
        for an open set $\omega \subset \R^{d-1}$ with $\H^{d-1}(\partial \omega) = 0$, $\delta > 0$, and $h \in \R$. Then, for any sequence $\eps\to 0$ there exists a sequence  $ \lbrace v_\eps\rbrace_\eps \subset H^2(\omega_\delta; \R^d)$ such that 
        \[
            v_\eps \to v_0 \quad \text{ strongly in $H^1(\omega_\delta; \R^d)$},
        \]
        where 
        \[
            v_0(x) :=  (\alpha^F(x) - \alpha^F(x', h))\chi_{\{x_d>h\}} + (\beta^F(x) - \beta^F(x', h))\chi_{\{x_d \leq h\}} \quad  \text{for $ x = (x', x_d)  \in \omega_\delta$},
        \]
        and such that 
        \[
            \limsup_{\epsilon \to 0}E_\epsilon^F [v_\epsilon; \omega_h]\leq \int_{\omega \times \{h\}} \sigma_{\rm geo}^F(x) \, d \mathcal{H}^{d-1}(x).
        \]
    
    \end{theorem}

    We remark that we use the above theorem primarily as a building block for the general $\limsup$ inequality proven in Theorem \ref{thm:general-limsup} below. We point out that the construction also shows that $v_\eps$ can be chosen to match $v_0$ along the top and bottom boundaries of the cylinder $\omega_\delta$.

    \begin{proof}
    For simplicity, we first treat the case $h = 0$, $\delta = 1$, and 
    $$\omega = (-1,1)^{d-1},$$ 
    i.e., we work on the cube $Q:= (-1,1)^d$. The adaptation to the general situation are briefly described in Step 4 below.  
    Fix   $\lambda \in (0,1)$. We construct $\{v_\eps\}_{\eps > 0} \subset H^2(Q; \R^d)$ with $v_\eps \to v_0$ in $L^1(Q; \R^d)$ such that 
    \begin{equation}\label{eqn:limsup+lambda}
    \limsup_{\epsilon \to 0}E_\epsilon^F [v_\epsilon;Q]\leq \int_{\omega \times \{0\}}\sigma_{\rm geo}^F(x)\, d \mathcal{H}^{d-1} (x)  + C\lambda
    \end{equation} 
   for a constant $C > 0$ independent of $\lambda$. A diagonal argument as $\lambda\to 0$ then yields the result.

    \emph{Step 1 (Definition of recovery sequence on a phase-transition layer).}   
    For every $m\in \N$, we divide the $(d-1)$-dimensional cube $\omega$ into smaller open $(d-1)$-dimensional cubes $\omega_i$ with pairwise disjoint interiors, sides of length $2^{-m}$, and centers $x'_i$, i.e., $\bigcup_i \omega_i$ coincides with $\omega$ up to a negligible set.  Associated to the centers $x_i:=(x'_i,0)$,   we take the near optimal geodesics $\zeta_i^\eps \colon  (-1, 1) \to \R^d$ provided by Lemma \ref{lem:near_optimal_geodesic_diffeomorphed_domain}, with $\lambda>0$ as chosen above. We further define 
    \begin{align}\label{eq:choice_of_rhoeps}
        \rho_\eps := 2 K_\lambda \eps.
    \end{align}
    Each $\zeta_i^\eps$ satisfies $\zeta_i^\eps(-s) = - \kappa(0)$ and  $\zeta_i^\eps(s) = \kappa(0)$ for $\rho_\eps/2 \le s  <1$. We naturally extend $\zeta_{i}^\eps$ to be constant left and right of $(-1,1)$. Notice also that, by Lemma \ref{lemma:diffeomorphed_quantities_are_comparable}, \eqref{eq:near_optimality_condition_of_pushforward_of_geodesic}, and \eqref{eq:uniform_bound_path_length_pushforward_of_geodesic}, 
    \begin{align}\label{eq:uniform_boundedness_of_geodesics_and_path_length}
        \|\zeta_i^\eps\|_\infty + \int_{-1}^1|(\zeta_i^\eps)'| \, ds + \eps\int_{-1}^1 |(\zeta_i^\eps)'|^2   \, ds \le  C_0
    \end{align} 
    for a sufficiently  large constant $C_0 = C_0(\lambda) > 0$, where we used the fundamental theorem of calculus to control the first term together with the fact that that $\kappa$ is bounded. Now, we choose any sequence $\lbrace \eta_\eps \rbrace_\eps \subset  (0, 1/4)$ such that 
   \begin{align}\label{eta-choice}
        \frac{\eps}{\eta_\eps^2} \to 0 \qquad \text{ and } \qquad  \eta_\eps \to 0. 
    \end{align}
    Let $\{\phi_i^\eps\}_i$  be a partition of unity subordinate to $\{(1+\eta_\epsilon)\omega_i\}_i$ with  
    \begin{equation}
        \label{eq:rec1}
        |\nabla'\phi_i^\eps(x')| \leq \frac{C}{\eta_\epsilon} \quad \text{ and } \quad |(\nabla')^2\phi_i^\eps(x')| \leq \frac{C}{\eta_\epsilon^2} \quad \text{ for all }x'\in \omega_i,
    \end{equation}
    and 
    \begin{equation}
    \label{eq:rec2}
        \phi_i^\eps(x') = 1  \quad \text{ on }(1-\eta_\epsilon)\omega_i.
    \end{equation}
    Here, $\nabla'$ denotes the part of the gradient in the first $d-1$ directions. Denote the layer surrounding the boundary  $\partial \omega_i$  by 
    \[
       \tau_i^\eps  := (1+\eta_\epsilon)\omega_i\setminus (1-\eta_\epsilon)\omega_i.
    \] 
    Next, we define the function 
    \begin{equation}\nonumber
        v_\epsilon(x',x_d) : = \sum_i \left(\int_{0}^{x_d}\zeta_i^\eps(t)\, dt\right) \phi_i^\eps(x') \quad \text{ for }x \in T_\epsilon,
    \end{equation}
    where $T_\epsilon:=\omega\times (-\rho_\epsilon, \rho_\epsilon)$ is the \emph{phase-transition layer}.

   \begin{figure}
       \centering
       \includesvg[width=0.7\linewidth]{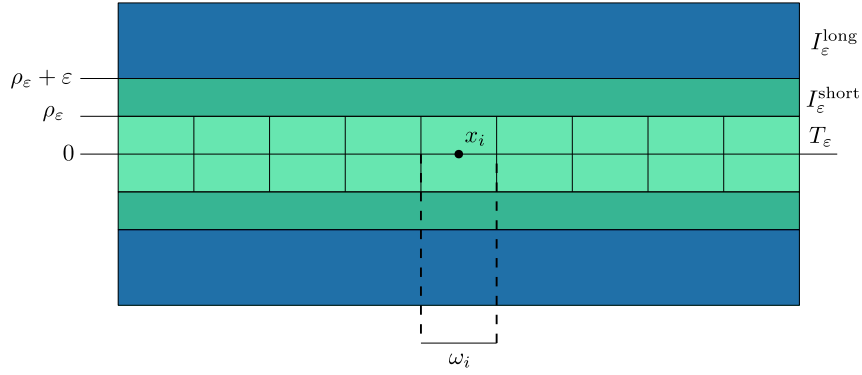}
       \caption{The discretization and the layers used in Theorem \ref{thm:first-limsup}.  }
       \label{fig:limsup_pic}
   \end{figure}

   In the next steps, we will interpolate the functions $v_\eps$ to the map $\alpha^F(x_d)-\alpha^F(0)$ (respectively $\beta^F(x_d)-\beta^F(0)$) above (respectively below) the phase transition layer in a low-energy way by completing a first and second matching as in \cite{contifonsecaleoni2002}. In particular, we will show that the extension of $v_\eps$ (without renaming), defined on the cube $Q$, fulfills  
    \begin{equation}\label{limsup_equation}
        \limsup_{\epsilon\to 0} E_\epsilon^F[v_\epsilon;T_\epsilon]\leq \sum_{i}\mathcal{H}^{d-1}(\omega_i)\left(\sigma_{\rm geo}^F(x_i) +  C_\lambda  \varsigma( 2^{-m} ) + C\lambda \right)
    \end{equation}
    for some constant $C_\lambda>0$ depending on $\lambda$, where $\varsigma \colon [0,\infty) \to [0,\infty)$ is a continuous nondecreasing function with $\varsigma(0)=0$,  and that 
    \begin{equation}\label{eqn:vanishEnergy}
        \limsup_{\epsilon\to 0} E^F_\epsilon [v_\epsilon;Q\setminus T_\epsilon] = 0 .
    \end{equation}
    %Since $\sigma_{\rm geo}^F$ is continuous, 
   Due to Lemma \ref{lemma:relation_of_pushforward_of_geodesic_distance}, Definition \ref{def:diffeomorphed_surface_density}, and the continuity of $\sigma_{\rm geo}$, which follows from Lemma \ref{geodesic-distance-gradient-estimates} proven in the next section, that $\sigma_{\rm geo}^F$ is continuous.  
    Consequently, by combining \eqref{limsup_equation} and \eqref{eqn:vanishEnergy}, and by sending $m\to \infty$,  a diagonalization argument concludes \eqref{eqn:limsup+lambda}.
    
    \textit{Step 2 (Estimate \eqref{limsup_equation} on $T_\epsilon$).}
    We directly compute the derivatives of $u_\epsilon$ as
    \begin{equation}\label{eqn:firstDerivatves}
        \nabla v_\epsilon(x', x_d) = \sum_i \left[  \left(\int_{0}^{x_d}\zeta_i^\eps(t)\, dt\right) \otimes \nabla'\phi_i^\eps(x') \ \Big|\ \zeta_i^\eps(x_d) \phi_i^\eps(x') \right]
    \end{equation}
    and 
    \begin{equation}\label{eqn:secondDerivatives}
        \nabla^2 v_\epsilon(x', x_d) = \sum_i \left[
        \begin{array}{cc}
            \left(\int_{0}^{x_d}\zeta_i^\eps(t)\, dt\right) \otimes (\nabla')^2\phi_i^\eps(x')      &   [\zeta_i^\eps(x_d)\otimes \nabla'\phi_i^\eps(x')]^T\\
            \zeta_i^\eps(x_d)\otimes \nabla'\phi_i^\eps(x') &  (\zeta_i^\eps)'(x_d) \phi_i^\eps(x')
        \end{array}
        \right]
    \end{equation}
    for every $x\in T_\eps$.

    \textit{Substep 2.1 (First-order estimates).}
    To estimate the elastic energy, we subdivide $T_{\eps}$ into the regions $(1-\eta_\eps)\omega_i$ and $ \tau_i^\eps $, and obtain 
    \begin{align}
         \int_{T_\eps} \frac{1}{\eps} W_F(x, \nabla v_\epsilon(x)) \, dx &= \frac 1 \epsilon \sum_i \int_{\omega_i \times (-\rho_\epsilon, \rho_\epsilon)} W_F(x, \nabla v_\epsilon(x)) \, dx \nonumber \\ 
        &\leq \frac 1 \epsilon \sum_i \Bigg( \int_{ \tau_i^\eps \times (- \rho_\epsilon, \rho_\epsilon)} W_F(x, \nabla v_\epsilon(x)) \, dx  + \int_{(1-\eta_\epsilon)\omega_i \times (- \rho_\epsilon, \rho_\epsilon)} W_F(x, \nabla v_\epsilon(x)) \, dx \Bigg) \nonumber \\
        &\label{eq:rec3}\leq \frac 1 \epsilon \sum_i \Bigg(  \int_{ \tau_i^\eps \times (- \rho_\epsilon, \rho_\epsilon)} W_F(x, \nabla v_\epsilon(x)) \, dx   + \int_{\omega_i \times (- \rho_\epsilon, \rho_\epsilon)} W_F \left(x, \zeta_i^\eps (x_d) \otimes e_d \right)\, dx \Bigg), %\label{eqn:breakFirstOrder}
    \end{align}
    where in the latter inequality we have used \eqref{eq:rec2} and \eqref{eqn:firstDerivatves}.
    To bound the first term on the right-hand side of \eqref{eq:rec3}, we use \ref{H2} and \eqref{eqn:firstDerivatves} to find
    \begin{align}
        \frac 1 \epsilon \sum_i &\int_{ \tau_i^\eps \times (- \rho_\epsilon , \rho_\epsilon)} W_F(x, \nabla v_\epsilon(x)) \, dx \leq  \frac 1 \epsilon \sum_i \int_{ \tau_i^\eps \times (-\rho_\epsilon, \rho_\epsilon)} C(1 + |\nabla v_\epsilon(x)|^2) \, dx \nonumber \\
        \leq & \, \frac{C}{\eps}  \left(\eta_\epsilon\rho_\epsilon +  \sum_i  \int_{ \tau_i^\eps \times (-\rho_\epsilon, \rho_\epsilon)} \left(\left| \left(\int_{0}^{x_d} \zeta_i^\eps(t)\, dt\right) \otimes \nabla'\phi_i^\eps(x')\right|^2 + |\zeta_i^\eps (x_d) \phi_i^\eps(x')|^2 \right) dx \right).\label{eqn:forSum}
    \end{align}
    By applying the Poincaré inequality together  with \eqref{eq:uniform_boundedness_of_geodesics_and_path_length} and \eqref{eq:rec1}, we infer for all $x \in T_\eps$ that
    \begin{align*}
        \left| \left(\int_{0}^{x_d}\zeta_i^\eps(t)\, dt\right) \otimes \nabla'\phi_i^\eps(x')\right| \leq \frac{C}{\eta_\eps}\int_{-\rho_\eps}^{\rho_\eps} |\zeta_i^\eps| \,  dt \leq \frac{C\rho_\eps}{\eta_\eps}\left( |\kappa(0)|+ \int_{-\rho_\eps}^{\rho_\eps} |(\zeta_i^\eps)'| \,  dt \right) \leq  \frac{C\rho_\eps}{\eta_\eps},
    \end{align*}
    where we recall that $\zeta_i^\eps( \pm \rho_\eps) = \pm \kappa(0)$.
    Using the last estimate, we obtain  
    \[
        \int_{ \tau_i^\eps \times (-\rho_\eps, \rho_\epsilon)} \left| \left(\int_{0}^{x_d}\zeta_i^\eps(t)\, dt\right) \otimes \nabla'\phi_i^\eps(x')\right|^2  dx \leq \frac{C\rho_\epsilon^2}{\eta_\epsilon^2} \mathcal{L}^d( \tau_i^\eps \times (-\rho_\epsilon, \rho_\epsilon)) \leq C\frac{\rho_\eps^3}{\eta_\eps 2^{m(d-1)}}.
    \]
    Applying again \eqref{eq:uniform_boundedness_of_geodesics_and_path_length}, we also infer
    \[
        \int_{ \tau_i^\eps \times (-\rho_\eps, \rho_\epsilon)} |\phi_i^\eps(x')\zeta_i(x_d)|^2 \, dx \leq C \mathcal{L}^d ( \tau_i^\eps \times (-\rho_\eps, \rho_\eps)) \leq  \frac{C\eta_\eps \rho_\eps}{ 2^{m(d-1)}}.
    \]
    Summing these inequalities over $i$ in \eqref{eqn:forSum} and using that
      $\rho_\epsilon =  2K_\lambda  \epsilon$ (see \eqref{eq:choice_of_rhoeps}), we conclude
    \begin{equation}
    \label{eq:rec4}
        \frac 1 \epsilon \sum_i \int_{ \tau_i^\eps \times (-\rho_\epsilon , \rho_\epsilon)} W_F(x, \nabla v_\epsilon(x)) \, dx \leq   C_\lambda  \left(\eta_\epsilon + \frac{\epsilon^2}{\eta_\epsilon}\right)
    \end{equation}
    for some constant $ C_\lambda  >0$ independent of $m$. Combining \eqref{eq:rec4} with \eqref{eq:rec3}, we find
    \begin{equation}\label{eqn:firstOrderContrie}
     \int_{T_\eps} \frac{1}{\eps} W_F(x, \nabla v_\epsilon(x)) \, dx \leq C_\lambda \left(\eta_\epsilon + \frac{\epsilon^2}{\eta_\epsilon}\right) + \sum_i \int_{\omega_i \times (-\rho_\epsilon, \rho_\epsilon)} \frac{1}{\eps} W_F(x, \zeta_i^\eps (x_d) \otimes e_d) \, dx.
    \end{equation}
    \textit{Substep 2.2 (Second-order estimates).}
    Similarly, we continue with the second-order term in $E_\eps^F [v_\eps;T_\eps]$. By \eqref{eqn:secondDerivatives} and Lemma \ref{lemma:diffeomorphed_quantities_are_comparable} we have  
    \begin{equation}\label{eqn:secondOrderExpandies}
         \int_{\omega\times (-\rho_\eps, \rho_\epsilon)} \eps |\nabla^2 v_\epsilon|^2_F \, dx \leq \eps \sum_i \left( C\int_{ \tau_i^\eps \times (-\rho_\eps, \rho_\epsilon)} |\nabla^2 v_\epsilon|^2 \, dx + \int_{\omega_i\times (-\rho_\eps, \rho_\epsilon)} |(\zeta_i^\eps)'(x_d) \otimes e_d \otimes e_d  |^2_{F,x} \, dx \right).
    \end{equation}
    Again employing \eqref{eqn:secondDerivatives}, we infer the bound 
    \begin{equation}\nonumber
    \begin{aligned}
        \sum_i \int_{ \tau_i^\eps\times (-\rho_\eps, \rho_\epsilon)} |\nabla^2 v_\epsilon|^2 \, dx & \leq  C \sum_i   \int_{ \tau_i^\eps\times (-\rho_\eps, \rho_\epsilon)} \left|\left(\int_{0}^{x_d}\zeta_i^\eps(t)\, dt\right)\otimes (\nabla')^2\phi_i^\eps(x') \right|^2  dx  \\
        & \hspace{3em} + C  \sum_i \int_{ \tau_i^\eps\times (-\rho_\eps, \rho_\epsilon)}\left(|\zeta_i^\eps(x_d) \otimes \nabla'\phi_i^\eps(x') |^2 + |(\zeta_i^\eps)'(x_d) \phi_i^\eps(x')|^2 \right) dx.
    \end{aligned}
    \end{equation}
    As before, we estimate each contribution separately, and obtain by  \eqref{eq:uniform_boundedness_of_geodesics_and_path_length} and \eqref{eq:rec1}
    \begin{gather*}
        \int_{ \tau_i^\eps\times (-\rho_\eps, \rho_\epsilon)} \left|\left(\int_{0}^{x_d}\zeta_i^\eps(t)\, dt\right)\otimes (\nabla')^2\phi_i^\eps(x') \right|^2  dx \leq \frac{C\rho_\epsilon^2}{\eta_\epsilon^4} \mathcal{L}^d ( \tau_i^\eps \times (-\rho_\eps, \rho_\epsilon))\leq C\frac{\rho_\eps^3}{\eta_\eps^3 2^{m(d-1)}    },\\
        \int_{ \tau_i^\eps\times (-\rho_\eps, \rho_\epsilon)} |\zeta_i^\eps(x_d) \otimes \nabla'\phi_i^\eps(x') |^2 \, dx \leq C\frac{1}{\eta_\epsilon^2} \mathcal{L}^d  ( \tau_i^\eps \times (-\rho_\eps, \rho_\epsilon)) \leq C \frac{\rho_\eps}{\eta_\eps  2^{m(d-1)}},
    \end{gather*}
    and lastly  \[
         \int_{ \tau_i^\eps\times (-\rho_\eps, \rho_\epsilon)} |(\zeta_i^\eps)'(x_d) \phi_i(x')|^2 \, dx   \le   C \frac{ \eta_\eps   }{  2^{m(d-1)} \eps}.
    \] 
     Gathering the above inequalities and using  \eqref{eq:choice_of_rhoeps}, we arrive at 
    \begin{equation}\label{eqn:secondOrderVanishers}
        \epsilon \sum_i \int_{ \tau_i^\eps \times (-\rho_\eps, \rho_\epsilon)} |\nabla^2 v_\epsilon|^2 \, dx \leq C_\lambda  \left(\frac{\epsilon^4}{\eta_\epsilon^3} + \frac{\epsilon^2}{\eta_\epsilon} + \eta_\epsilon \right),
    \end{equation}
    where $C_\lambda  >0$ is independent of $m$.  In particular,  we use \eqref{eta-choice}
     and
    we combine \eqref{eqn:firstOrderContrie}, \eqref{eqn:secondOrderExpandies}, and \eqref{eqn:secondOrderVanishers} to obtain  
    \begin{equation}\label{eqn:almostTepsest}
    \limsup_{\epsilon \to 0} E_\epsilon[v_\epsilon; T_\eps] \leq 
         \limsup_{\epsilon \to 0} \sum_i \left(\int_{\omega_i \times (-\rho_\eps, \rho_\epsilon)} \frac{1}{\epsilon}W_F (x, \zeta_i^\eps (x_d) \otimes e_d ) + \epsilon |(\zeta_i^\eps)'(x_d) \otimes e_d \otimes e_d |_{F, x}^2 \, dx  \right).
    \end{equation}
    \textit{Step 2.3 (Localisation in $x$).}
    By \eqref{eq:uniform_boundedness_of_geodesics_and_path_length}, we note that  for $x \in \omega_i \times (-\rho_\eps, \rho_\eps)$
    \begin{align}\label{eq:Lipschitz_cond_1}
        |W_F(x, \zeta_i^\eps (x_d) \otimes e_d) -  W (x_i, \zeta_i^\eps (x_d) \otimes e_d) | \leq \tilde \varsigma(|x - x_i|) \leq \tilde \varsigma(2^{-m}),
    \end{align}
    where $\tilde \varsigma$ is the modulus of continuity for $W \colon \overline {\Omega\times B(0,C_0)}\to \R$, and, recalling Definition \ref{def:diffeomorphed_quantities},    
    \begin{align}\label{eq:Lipschitz_cond_2}
        ||\mathfrak T|_{F,x} - |\mathfrak T|_{ F, x_i}| \leq  C|(\nabla F\circ F^{-1})(x) - (\nabla F \circ F^{-1}) (x_i)|  |\mathfrak T| \leq \frac{C}{2^m}|\mathfrak T|
    \end{align}
        
    for all $\mathfrak T \in \R^{d \times d \times d}$. Here, we used that $\nabla F$ is Lipschitz on $\overline \Omega$. Inserting these latter estimates into \eqref{eqn:almostTepsest}, using \eqref{eq:uniform_boundedness_of_geodesics_and_path_length} with Lemma \ref{lemma:diffeomorphed_quantities_are_comparable}, and \eqref{eq:choice_of_rhoeps}, we can derive, with $\varsigma := \tilde \varsigma + |\cdot|$,
    \begin{align}\label{eqn:localLimsup}
        &\left(\int_{\omega_i \times (-\rho_\eps, \rho_\epsilon)} \frac{1}{\epsilon}W_F(x, \zeta_i^\eps (x_d) \otimes e_d ) + \epsilon |(\zeta_i^\eps)'(x_d) \otimes e_d \otimes e_d|^2_{ F, x} \, dx  \right) \nonumber \\
        & \leq  \left(\int_{\omega_i \times (-\rho_\eps, \rho_\epsilon)} \frac{1}{\epsilon}W_F(x_i, \zeta_i^\eps (x_d) \otimes e_d ) + \epsilon |(\zeta_i^\eps)'(x_d) \otimes e_d \otimes e_d|^2_{ F, x_i} \, dx  \right) + C_\lambda \mathcal{H}^{d-1}(\omega_i)\varsigma(2^{-m})
    \end{align}    
    Plugging \eqref{eqn:localLimsup} into \eqref{eqn:almostTepsest}, we apply Fubini's theorem and use \eqref{eq:near_optimality_condition_of_pushforward_of_geodesic} to conclude \eqref{limsup_equation}.

    \textit{Step 3 (Definition of $v_\eps$ and energy estimate  above $T_\epsilon$).}
    In this step, we show that we can define a suitable extension above the surface $\omega \times \{\rho_\eps\}$ which asymptotically carries no energy. First, we prove
    \begin{equation}\label{eqn:traceEnergy}
        \int_{\omega \times \{\rho_\epsilon\}} \left(\frac{1}{\epsilon}W_F(x,\nabla v_\epsilon) + \epsilon |\nabla^2 v_\epsilon|^2_F + |\nabla v_\epsilon - A^F|^2 +\frac{1}{\epsilon}|v_\epsilon - (\alpha^F - \alpha^F(0))|^2\right) d \mathcal{H}^{d-1} \to 0.
    \end{equation}
    By definition, the function $v_\epsilon$ has a well-defined trace on the upper boundary of $T_\epsilon$, which, for convenience, we denote by 
    \[
        v^\rho_\epsilon(x') : = v_\epsilon(x',\rho_\epsilon).
    \]
    By \eqref{eq:uniform_boundedness_of_geodesics_and_path_length}, \eqref{eq:rec1}, and \eqref{eqn:firstDerivatves}--\eqref{eqn:secondDerivatives},  we observe
    \begin{align}\label{eq:uniform_estimate_on_trace}
        \|v^\rho_\eps\|_\infty + {\eta_\eps}\|\nabla' v^\rho_\eps \|_\infty + \eta_\eps^2\|(\nabla')^2 v^\rho_\eps \|_\infty \leq C\rho_\eps.
    \end{align}
    We note that, because $\zeta_i^\eps(\rho_\eps)= \kappa(0)$ for all $i$, \eqref{eqn:firstDerivatves} simplifies to 
    \begin{equation}\label{eqn:firstDeriveTrace}
        \nabla v_\eps (x)  = \sum_i\left[\left(\int_0^{\rho_\eps} \zeta_i^\eps(t) \, dt \right)\otimes \nabla'\phi_i^\eps(x') \,\Big|\, \zeta_i^\eps(\rho_\eps)\phi_i^\eps(x') \right] 
        =[\nabla' v_\eps^\rho(x') \,|\, \kappa (0)] \quad \text{ for }x\in \omega\times \{\rho_\eps\} .
    \end{equation}
    By \eqref{eq:uniform_estimate_on_trace}, recalling that $A^F  (x) = \kappa(x_d)  \otimes e_d$ (see \eqref{symmetriewell}), \eqref{eq:choice_of_rhoeps}, \eqref{eta-choice}, and \eqref{eqn:firstDeriveTrace}, we derive
    \begin{equation}\nonumber
    \begin{aligned}
        \int_{\omega \times \{\rho_\epsilon\}}  \frac{1}{\epsilon}|\nabla v_\epsilon - A^F|^2\, dx \leq &\, C \int_{\omega \times \{\rho_\epsilon\}}\frac{1}{\eps}\left(|\nabla' v_\eps^\rho|^2 + |\kappa(\rho_\eps) - \kappa(0)|^2\right)\, d\mathcal{H}^{d-1}(x) \\
    \leq & \, C \left(  \frac{1}{\eps}\frac{\rho_\eps^2}{\eta_\eps^2} \sum_i\mathcal{H}^{d-1}(\tau_i^\eps \times \{\rho_\epsilon\}) \right) + C\frac{\rho_\eps^2}{\eps} 
    \leq  C_\lambda \left( \frac{\eps}{\eta_\eps} + \eps\right) \to 0,
    \end{aligned}
    \end{equation}
    where we have used the fact that $|\kappa(\rho_\eps) - \kappa (0)|\leq C\rho_\eps$ by the regularity of $A$. By Lemma  \ref{lemma:diffeomorphed_quantities_are_comparable}, $W_F$ satisfies \ref{H2}, and thus we infer from the last estimate    
    \[
        \frac{1}{\eps}\int_{\omega \times \{\rho_\epsilon\}} W_F(x,\nabla v_\epsilon) \, d\mathcal{H}^{d-1}  \to  0.
    \]
    Recall that    $\zeta_i^\eps(s) = \kappa(0)$ for all $s \ge \rho_\eps/2$, see below  \eqref{eq:choice_of_rhoeps}. Therefore, by \eqref{eqn:secondDerivatives} we have for all $x \in \omega \times \{ \rho_\eps \}$
    \begin{equation}
        \nabla^2 v_\epsilon(x) = \sum_i \left[
        \begin{array}{cc}
            \left(\int_{0}^{\rho_\eps}\zeta_i^\eps(t)\, dt\right) \otimes (\nabla')^2\phi_i^\eps(x')      &   [\kappa(0) \otimes \nabla'\phi_i^\eps(x')]^T\\
            \kappa(0) \otimes \nabla'\phi_i^\eps(x') &  0
        \end{array}
        \right]. \nonumber
    \end{equation}
    In particular, by Lemma \ref{lemma:diffeomorphed_quantities_are_comparable}, \eqref{eq:choice_of_rhoeps},  \eqref{eq:uniform_boundedness_of_geodesics_and_path_length}, \eqref{eta-choice}, and \eqref{eq:rec1}, it follows that 
    \[
        \int_{\omega \times \{\rho_\epsilon\}} \epsilon |\nabla^2 v_\epsilon|^2_F \, d\mathcal{H}^{d-1} \leq  C_\lambda  \left(\frac{\eps^3}{\eta_\eps^4} + \frac{\eps}{\eta_\eps^2} \right)  \to  0.
    \]
    Since $|v_\eps|\leq C\rho_\eps$ and $|\alpha(\rho_\eps) - \alpha(0)|\leq C\rho_\eps$,  we also have 
    \[
        \int_{\omega \times \{\rho_\epsilon\}} \frac{1}{\epsilon}|v_\epsilon - (\alpha - \alpha(0))|^2 \, d\mathcal{H}^{d-1} \to 0
    \]
    holds. This concludes the proof of \eqref{eqn:traceEnergy}.

    \textit{Substep 3.1 (First matching).} On the \emph{short interpolation layer} 
    $$I_\epsilon^{\rm short}:= \omega \times (\rho_\epsilon, \rho_\epsilon + \epsilon)$$ 
    we define 
    \begin{equation}\label{def:Ishort}
        v_\epsilon(x) : = \psi_\epsilon(x_d)\big(v_\epsilon^{\rho}(x') + \kappa(0)(x_d-\rho_\epsilon)\big)+ (1-\psi_\epsilon(x_d))\left(\alpha^F(x) - \alpha^F(\rho_\epsilon e_d)  + v^\rho_\epsilon(x')\right),
    \end{equation}
    for $x\in I_{\epsilon}^{\rm short}$,
    where $\psi_\epsilon$ is a smooth function transitioning from $1$ to $0$ on the interval $(\rho_\epsilon, \rho_\epsilon + \epsilon)$ with
    \begin{align}\label{psiprime}
        \psi'_\eps = 0 \quad \text{ in } (\rho_\epsilon, \rho_\epsilon + \epsilon)^c, \quad \text{and}\quad \quad  \epsilon |\psi_\epsilon'| + \epsilon^2 |\psi_\epsilon''|\leq C <\infty.
    \end{align}
    By \eqref{eqn:firstDeriveTrace}, we have that the trace of $\nabla v_\eps$ on $\omega\times \{\rho_\eps\}$ is $[\nabla' v_\eps^\rho(x) \,|\, \kappa(0)]$, so that definition \eqref{def:Ishort} ensures $v_\eps\in H^2(T_\eps\cup I_\eps^{\rm short})$.

     To show that the energy on $I_\eps^{\rm short}$ vanishes as $\epsilon\to 0$, we first compute the gradient and Hessian of $v_\eps$ in this region. Computing the derivatives of $v_\eps$, we find
    \begin{equation}\label{eqn:uepsIGrad}
    \begin{aligned}
        \nabla v_\epsilon (x) &= \psi_\epsilon(x_d)A^F(0)+ (1-\psi_\epsilon(x_d))A^F(x) + \nabla v^\rho_\epsilon(x')  \\
        & \hspace{3em} +  \psi_\epsilon'(x_d) \Big( \kappa(0)(x_d-\rho_\epsilon) - (\alpha^F(x) - \alpha^F(\rho_\epsilon e_d))\Big)\otimes e_d,
    \end{aligned}
    \end{equation}
    and
    \begin{equation}\nonumber
    \begin{aligned}
        \nabla^2 v_\epsilon (x) &= (1-\psi_\epsilon(x_d)) [(A^F)'(x)\otimes e_d ] + \nabla^2 v^\rho_\epsilon(x')  \\
        &\hspace{3em} +  2\psi_\epsilon'(x_d) \left[( A^F(0) - A^F(x))\otimes e_d\right] \\
        &\hspace{3em} +  \psi_\epsilon''(x_d) \left[\Big( \kappa(0)(x_d-\rho_\epsilon) - (\alpha^F(x) - \alpha^F(\rho_\epsilon e_d))\Big)\otimes e_d\right]\otimes e_d
    \end{aligned}
    \end{equation}
    for every $x\in I^{\textrm{short}}_\eps$.
    With this, we now show that 
    \begin{equation}\label{eqn:IshortVanish}
    \begin{aligned}
        \lim_{\eps\to 0}\int_{I_\epsilon^{\rm short}} \left(\frac{1}{\epsilon}W_F(x,\nabla v_\eps) + \epsilon |\nabla^2 v_\eps|_F^2  + |\nabla v_\eps - A^F|^2 +|v_\eps - (\alpha^F - \alpha^F(0))|^2\right) dx =0.
    \end{aligned}
    \end{equation}
    Indeed, as $W_F$ satisfies \ref{H2} (see  Lemma \ref{lemma:diffeomorphed_quantities_are_comparable}),  by \eqref{psiprime},   \eqref{eqn:uepsIGrad},  Poincar\'e's inequality with trace $0$ on $\omega\times \{\rho_\eps\}$, the fact that $I_\epsilon^{\rm short}$ has width $\eps$,  and also recalling  $A^F(x)   = \kappa  (x_d)  \otimes e_d$ from \eqref{symmetriewell},  we have 
    \begin{equation}\nonumber
    \begin{aligned}
    & \int_{I_\epsilon^{\rm short}} \frac{1}{\epsilon}W_F(x,\nabla v_\eps) \, dx\\
    &\leq \frac{C}{\epsilon} \int_{I_\epsilon^{\rm short}}\left(| A^F(0) -A^F(x)|^2 + |\nabla v^\rho_\epsilon (x')|^2 + \frac{1}{\epsilon^2}|( \kappa (0)(x_d-\rho_\epsilon) - (\alpha^F(x) - \alpha^F(\rho_\epsilon e_d))|^2\right) dx \\
    &\leq \frac{C}{\epsilon} \int_{I_\epsilon^{\rm short}}\left(|\nabla v_\eps^{\rho}(x')|^2 + |A^F(0) -A^F(x)|^2\right) dx \\
    &\leq C \int_{\omega\times\{\rho_\epsilon\}} W_F(x,\nabla v_\eps )  \, d \mathcal{H}^{d-1}+C\epsilon^2, 
    \end{aligned}
    \end{equation}
    where in the last inequality we have used that $|M' |^2\leq \min\{|M-A^F_x|^2,|M - B^F_x|^2\} \leq  CW_F(x,M)$ for all $x\in Q$ and $M\in \R^{n\times n}$ (with $M'$ being the matrix arising from removing the last column of $M$), $\nabla v_\eps^{\rho} = [\nabla' v_\eps^{\rho},0]$, and that $|A^F(x) -A^F(0)|\leq C|x_d|$ by the regularity of $\alpha$.  
    By (\ref{eqn:traceEnergy}), we conclude that 
    $$\int_{I_\epsilon^{\rm short}} \frac{1}{\epsilon}W_F(x,\nabla v_\eps) \, dx \to 0. $$
    Analogously, using again Lemma \ref{lemma:diffeomorphed_quantities_are_comparable}, one may show that
    \begin{align*}
        &  \eps \int_{I_\eps^{\rm short}} |\nabla^2 v_\eps|_F^2 \, dx  \\
        &  \leq  C\eps \int_{I_\eps^{\rm short}} \left(1 + |\nabla^2 v_\eps^\rho(x')|^2 + \frac 1{\eps^2} |A^F(0) - A^F(x)|^2 + \frac{1}{\eps^4}|\kappa(0)(x_d - \rho_\eps) -  ( \alpha^F(x) - \alpha^F( \rho_\epsilon e_d ))|^2\right) dx  \\
        &  \leq C\left(\eps + \eps \int_{\omega\times\{\rho_\epsilon\}}| (\nabla')^2 v_\eps^\rho(x')|^2 \, d\mathcal{H}^{d-1} \right) \to 0,
    \end{align*}
    where we have also used \eqref{eq:uniform_estimate_on_trace}. The rest of (\ref{eqn:IshortVanish}) is verified in a similar fashion. 
    
    \textit{Substep 3.2 (Second matching).} On the \emph{long interpolation layer} $I_\epsilon^{\rm long}:= \omega \times ( \rho_\epsilon + \epsilon, 1  )$ we define 
    \begin{equation}\nonumber
        v_\eps(x) : =  \alpha^F (x) - \alpha^F (\rho_\epsilon e_d) + \psi(x_d)v_\eps^{\rho}(x') \quad \text{ for }x\in I_{\epsilon}^{\rm long},
    \end{equation}
    where $\psi$ is a smooth function transitioning from $1$ to $0$ on the interval $(\rho_\epsilon + \epsilon, 1 )$ with 
    $$ 
        |\psi'| + |\psi''|\leq C <\infty.
    $$ 
    Similar computations as in Substep 3.1 yield 
    \begin{equation}\label{eqn:IlongVanish}
    \begin{aligned}
    \lim_{\eps \to 0} \int_{I_\epsilon^{\rm long}} \left(\frac{1}{\epsilon}W_F(x,\nabla v_\eps) + \epsilon |\nabla^2 v_\eps|^2_{F}  + |\nabla v_\eps -  A^F |^2 +|v_\eps - (\alpha^F - \alpha^F(0))|^2\right)\, dx = 0.
    \end{aligned}
    \end{equation}
    To see this, as before, by applying \eqref{eta-choice} and  \eqref{eq:uniform_estimate_on_trace} we have
    \begin{equation}\nonumber
    \begin{aligned}
        & \int_{I_\epsilon^{\rm long}} \frac{1}{\epsilon}W_F(x,\nabla v_\eps)\, dx \leq C \int_{\omega\times\{\rho_\epsilon\}} \frac{1}{\epsilon}\left(|\nabla' v_\eps^\rho|^2 + |v^\rho_\epsilon|^2 \right) \, d \mathcal{H}^{d-1} \leq C \left( \frac{\eps}{\eta_\eps^2} + \eps \right)\to 0.
    \end{aligned}
    \end{equation}
    and
    \begin{align*}
        \int_{I_\epsilon^{\rm long}} \eps |\nabla^2 v_\eps|_F^2 \, dx \leq C\left( \eps + \frac{\eps^3}{\eta_\eps^4} \right) \to 0.
    \end{align*}
    The convergence of the remaining terms in (\ref{eqn:IlongVanish}) is again checked in a similar fashion. This concludes the construction of the recovery sequence on the upper part $(-1,1)^{d-1} \times (-\rho_\eps,1)$. Extending $v_\eps$ to $(-1,1)^{d-1} \times (-1,1)$, along with the corresponding energy estimates, is done in an analogous way. Eventually, combining   \eqref{eqn:IshortVanish} and \eqref{eqn:IlongVanish} shows \eqref{eqn:vanishEnergy}.

    \textit{Step 4 (Construction for arbitrary $\omega$).} We briefly indicate the adaptations in the case that $\omega \subset \R^{d-1}$ with $\H^{d-1}(\partial \omega) = 0$, $h =0$, and $\delta > 0$ are arbitrary. We cover $\omega$ with   $(d-1)$-dimensional cubes $\omega_i$ of size $1/2^m$ with pairwise disjoint interiors and $\omega_i \cap \omega \neq \emptyset$. Instead of choosing the midpoints of the cubes, we choose arbitrary $x_i \in \omega_i \cap \omega$ and repeat the steps from Step 1 to Step 3.2 to define $v_\eps$ such that 
    \begin{align}\label{6.36}
        \limsup_{\epsilon\to 0} E_\epsilon^F[v_\epsilon; \omega_\delta]\leq \sum_{i}\mathcal{H}^{d-1}(\omega_i)\left(\sigma_{\rm geo}^F(x_i) + \varsigma(2^{-m}) + C\lambda \right).
    \end{align}
    The only difference is that the long transition in Substep 3.2 in performed an interval of length $\delta$. By continuity, we then obtain 
    \[
        \sum_{i}\mathcal{H}^{d-1}(\omega_i)\sigma_{\rm geo}^F(x_i) \xrightarrow{m \to \infty} \int_{\overline{\omega}} \sigma_{\rm geo}^F(x) \, d\H^{d-1}(x) = \int_{\omega} \sigma_{\rm geo}^F(x) \, d\H^{d-1}(x),
    \]
    where in the last step we used $\mathcal{H}^{d-1}(\partial \omega) = 0$. If $h \neq 0$, we repeat the exact construction at height $h$.
    \end{proof}

    Eventually, we come to the main result of this section, namely, recovery sequences on general domains. Once this is achieved, the proof of Theorem \ref{thm:main} is concluded.
  \begin{theorem}\label{thm:general-limsup}
        Let $\Omega \subset \R^d$ be a simply connected Lipschitz domain, and $u \in X(\Omega)$. Let $W$ satisfy \ref{H1}--\ref{H4}.  Then, there exists  $\lbrace u_\eps\rbrace_\eps\subset H^2(\Omega; \R^d)$ such that 
        \[
            u_\eps \to u
        \]
        in $L^1(\Omega; \R^d)$, and 
        \[
            \limsup_{\epsilon \to 0}E_\epsilon (u_\epsilon) \leq \int_{ J_{\nabla u}}\sigma_{\rm geo}(x')\, d \mathcal{H}^{d-1} (x'). 
        \]
    \end{theorem}

    \begin{proof}
    In view of Lemma \ref{lem:equivalence_of_gamma_limits_under_diffeomorphism}, it suffices to consider $v = u \circ F^{-1} $ with $\int_{J_v}\sigma_{\rm geo}^F \, d\mathcal{H}^{d-1}< \infty,$ and construct a recovery sequence  $\lbrace v_\eps\rbrace_\eps$ in the diffeomorphed setting.  The proof proceeds in three steps. In Step 1, we reduce to a target $v$ with finitely many connected components in the jump set. In Step 2, we construct a recovery sequence in the case that the jump set is contained in a single hyperplane. Finally, in Step 3, we use Step 2 to construct a curl-free field near hyperplanes containing the jump set and naturally extend the field far away from the jump set to define the gradient of a recovery sequence. 
    
    \emph{Step 1 (Simple target function $v$).}
        By Theorem \ref{thm:characterisation-of-piecewise-BV-space} applied to $v\in X(F(\Omega))$, defined with $A^F$ and $B^F$ in place of $A$ and $B$, there exists a sequence $\lbrace t_k\rbrace_k  \subset \R$ such that 
        \[
            J_{\nabla v} \subset \bigcup_{k \in \N} \{x \in F(\Omega) : x_d = t_k \}.
        \]
        Consequently, $F(\Omega) \setminus J_{\nabla v}$  consists of a countable number of connected components. By an approximation argument detailed in \cite[Theorem 5.10, Step 4]{contifonsecaleoni2002} we can approximate $v$ by a function which is affine on a \emph{finite} number of connected components. In particular, we can restrict ourselves to finitely many $t_k$, $k = 1, \ldots, k_0 $. We let $\overline{J_v}$ and $\partial J_v$ denote the relative closure and boundary, respectively, of $J_v$ within the union of hyperplanes determined by $t_k.$ 

        \textit{Step 2 (Interface contained in single plane).} In the following, we suppose $k_0 = 1$, so that the jump set $J_v$ is contained in a single hyperplane.  We construct a recovery sequence such that $\nabla v_\eps(x) = \nabla v(x)$  for each $x \in F(\Omega)$ with  $|x_d -  t_{k_0} | \not \in [0,\delta)$  where $\delta \in (0, 1/m)$ is a small chosen parameter.
        
		\emph{Substep 2.1 (Covering $J_v$ by cubes).}        
        We begin by covering  $\{x \in F(\Omega)\colon \,  x_d = t_{k_0} \}$ by $(d-1)$-dimensional cubes $\lbrace \omega_i\rbrace_{i \in I}$  with sidelength $2^{-m}$ as in Theorem \ref{thm:first-limsup}. Let $\eta_\eps \to 0^+$ such that $0 < \eta_\eps < 1/4$. We categorize the cubes into five different types as follows: We split the index set of the cubes into
        \begin{align}\label{eq:splitting_indexing}
            I = I_{A,B} \cup I_{B,A} \cup I_A \cup I_B \cup I_\partial
        \end{align}
        in the following way  (see also Figure \ref{fig:indexing}):
        \begin{itemize}
            \item (Transitional cubes) If $(1+\eta_\eps)\omega_i \subset J_v $, $\nabla v = A^F$ in $(1+ \eta_\eps )\omega_i \times (0,\delta) \cap F(\Omega)$ and $\nabla v = B^F$ in $(1+ \eta_\eps )\omega_i \times (-\delta, 0) \cap F(\Omega)$,  then $i \in I_{A,B}$. Analogously, we say  $ i \in I_{B, A}$ if $A^F$ is replaced by $B^F$.  
            \item (Pure phase cubes) If $(1+\eta_\eps)\omega_i \subset (F(\Omega)\cap \{x_d = t_{k_0} \}) \setminus  J_{\nabla v} $ and $\nabla v = A^F$ in $((1+ \eta_\eps)\omega_i \times (-\delta,\delta)) \cap F(\Omega) $, then $i \in I_{A}$. Analogously, we say $ i \in I_B$ if $A^F$ is replaced by $B^F$. 
            \item (Relative boundary cubes) If  $(1+\eta_\eps)\omega_i \cap \partial J_v \neq \emptyset$, then $i \in I_\partial$.           
        \end{itemize}
        \begin{figure}
            \centering
            \includesvg[width=0.7\linewidth]{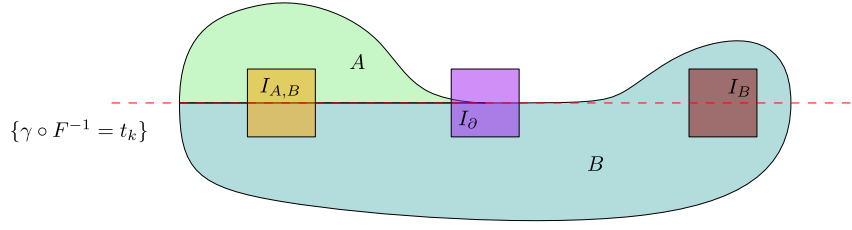}
            \caption{The three different types of cubes detailed in the construction of the recovery sequence in the proof of Theorem \ref{thm:general-limsup}.}
            \label{fig:indexing}
        \end{figure}
        Note that any index $i$   falls exactly into one of the five categories above, provided that $m$ is large enough and $\delta > 0$ is sufficiently small.   
        Moreover,  a neighbor  of a cube either matches its behavior or is a boundary cube, in the sense that, for two indices  $i_1 , i_2 \in I$  of two neighboring cubes,   we have 
        \[
            i_1 \in I_A \implies i_2 \in I_A \cup I_\partial, \qquad \text{ and } \qquad  i_1 \in I_{A,B} \implies i_2 \in I_{A,B} \cup I_\partial, 
        \]
        and analogous statements with the roles of $A$ and $B$ switched. With this in mind, we will now adapt the construction of Theorem \ref{thm:first-limsup} (using also the same notation). 
        
        The main idea is as follows: if a cube is transitional, we will define $v_\eps$ near the cube in terms of an almost optimal geodesic as in Theorem \ref{thm:first-limsup}. If a cube is a pure phase, we will (essentially) define $v_\eps = v$ near the cube. Finally, given that a boundary cube may   have neighboring  cubes in all other four categories, we will define the transition $v_\eps$ explicitly in terms of $v$, but smooth out its gradient discontinuity. This will lead to an error term proportional to $\mathcal{H}^{d-1}(\omega_i)$ for the boundary cube. Then, using \ref{H4}(iii) we   control the number of boundary cubes present.        
        
        \emph{Substep 2.2 (Constructing recovery sequence $v_\eps$).}  
       Fix $\lambda > 0$. Consider a cube $\omega_i$ with center $x_i$. If $i \in I_{A,B}$, let $\zeta_i^\eps\colon (-\rho_\eps, \rho_\eps) \to \R^d$ be the curve connecting $A^F$ to $B^F$ given by Lemma \ref{lem:near_optimal_geodesic_diffeomorphed_domain}. For $i \in I_A$,  we define $\zeta_i^\eps\colon (-\rho_\eps, \rho_\eps) \to \R^d$ with $\zeta_i^\eps = \kappa(t_{k_0})$, and for $i \in I_B$ we set  $\zeta_i^\eps = - \kappa(t_{k_0})$. For the approximation on the relative boundary cubes, we use any (possibly suboptimal) phase transition with bounded energy that almost matches $v$ on top and bottom of the cube. For definiteness, we use the following transition:  First, set
        \[
            \Wtrunc^{k_0}(x) := \kappa(t_{k_0})  (x_d - t_{k_0} )  \chi_{\{\nabla v = A^F\}}  -\kappa (t_{k_0})  (x_d - t_{k_0} )   \chi_{\{\nabla v = B^F\}} \in H^1(F(\Omega);\R^d). 
        \]
        Then, let $\psi_\eps \in C^\infty([-\delta,\delta]; [0,1])$ be such that 
        \[
            \psi_\eps(0) = 0 \text{ in  a neighborhood of $0$}, \qquad \psi_\eps(s) = 1 \text{ for }s \in [-\delta, \delta]\setminus[-\rho_\eps, \rho_\eps],
        \]
        and 
        \[  
            \eps|\psi_\eps'| + \eps^2|\psi_\eps''| \le   C_\lambda
        \]
        for some $C_\lambda>0$ depending on $\lambda$, where we recall the definition of $\rho_\eps$ in \eqref{eq:choice_of_rhoeps}. 
        We then define
        \[
            \tilde v_\eps^{k_0}(x) := \psi_\eps \Wtrunc^{k_0} \in C^\infty\big( T_\eps;  \R^d\big),
        \]
        where $T_\eps := F(\Omega) \cap (\R^{d-1} \times ( t_{k_0} - \rho_\eps, t_{k_0} + \rho_\eps))$. 
        Define a suitable partition of unity $\phi^\eps_i$ as in Theorem~\ref{thm:first-limsup}, see  \eqref{eq:rec1}, and set
        \[
            v_\epsilon(x',x_d) : = \sum_{i \in I \setminus I_\partial} \left(\int_{0}^{x_d}\zeta_i(t)\, dt\right) \phi_i^\eps(x') + \sum_{i \in I_\partial} \tilde v^{k_0}_\eps(x) \phi_i^\eps(x') \quad \text{ for }x \in T_\epsilon.
        \]
                 We notice  that by construction
        \[
            \nabla v_\eps( x', t_k \pm \rho_\eps) =  \nabla V^{k_0}(x', t_{k_0} \pm \rho_\eps) = \kappa(\pm t_{k_0}) \otimes e_d.  
        \]
        Repeating the estimates from Theorem \ref{thm:first-limsup},  see in particular \eqref{6.36},  and applying \eqref{eq:Lipschitz_cond_1} and \eqref{eq:Lipschitz_cond_2} in the pure phase cubes, one derives   
        \begin{align}
            \limsup_{\epsilon\to 0} E_\epsilon^F[v_\epsilon; T_\epsilon]\leq & \sum_{i \in I_{A,B} \cup I_{B,A}}\mathcal{H}^{d-1}(\omega_i)\left(\sigma_{\rm geo}^F(x_i) + \varsigma( 2^{-m} ) + C\lambda \right) \nonumber \\
            &\quad + C\left(\sum_{i \in I_\partial}\H^{d-1}(\omega_i) + \sum_{i \in I_A \cup I_B}\H^{d-1}(\omega_i)\varsigma( 2^{-m}  )\right). \label{eqn:correctMiddleLayer}
        \end{align} 
        If $\delta  = \delta(m)$ with $\delta\to 0$ as $m \to \infty$ is chosen small enough (independently of $\eps$), we have 
        \begin{align*}
            \partial S_\delta \cap F(\Omega) = \big(\R^{d -1} \times \{t_{k_0}  - \delta, t_{k_0}  + \delta\}\big)\cap F(\Omega),
        \end{align*}
        where
        \begin{align}\label{eq:transition_layer_definition_general_proof}
            S_\delta :=  \left( \bigcup_{i \in I} \omega_i \times [ t_{k_0}  - \delta, t_{k_0}  + \delta]\right)  \cap F(\Omega).
        \end{align}
        Thus, for $\eps \ll \delta$ we have defined $v_\eps$ in all of $S_\delta.$ 
        On $S_\delta \setminus  T_\eps$ we proceed exactly   as in Step 3 of Theorem~\ref{thm:first-limsup} to extend $v_\eps$ to the whole of $ S_\delta$ with
        \begin{equation}\label{eqn:vanishingTopLayer}
            \limsup_{\epsilon\to 0} E_\epsilon^F[v_\epsilon; S_\delta \setminus T_\epsilon] = 0
        \end{equation}
        and that        \begin{equation}\label{eqn:gradMatching}
            \nabla v_\eps(x) = \nabla v(x)
        \end{equation}
      whenever $x \in F(\Omega) \cap \partial S_\delta $. Defining now
        \begin{align*} 
            M_\eps(x) := \begin{cases}
                \nabla v_\eps(x) & \text{for }  x \in S_\delta \cap F(\Omega) \\
                \nabla v(x) & \text{for }  x \in F(\Omega) \setminus S_\delta
            \end{cases}
        \end{align*}
        yields a curl-free field. Since the domain is simply connected, there exists a vector field $v_\eps \in H^2(\Omega; \R^d)$ (without renaming) such that $\nabla v_\eps =  M_\eps$ in $F(\Omega)$ since $F(\Omega)$ is simply connected.
        
        At this point, one could apply a diagonalization argument in $m$, $\lambda,$ and $\eps$ to show that $v_\eps$ forms a recovery sequence for $v$ in $F(\Omega)$. 
        
        To make use of this step for more complicated geometries, we phrase it as follows: We have shown that in $S_\delta$ there is a curl-free field $\nabla v_\eps$ matching $\nabla v$ near the top and bottom boundaries of $S_\delta$ such that the energetic estimates \eqref{eqn:correctMiddleLayer} and \eqref{eqn:vanishingTopLayer} hold. 
        
        \textit{Step 3 (Finitely many interfaces).} We now discuss how the above argument can be extended to a jump set that is contained in multiple hyperplanes, i.e., the number of level sets $k_0$ is bigger than $1$. For this, our objective is to construct a curl-free field representing $\nabla v_\eps.$

         For every $k = 1,\ldots,k_0$,  we apply Step 2 to construct a function $ v_\eps^{k}$ defined on $S_\delta^k \cap F(\Omega)$ via a partition $\{\omega_i^k\}_i$ of cubes with centers $x_i^k$, where $S_\delta^k$ are taken as in \eqref{eq:transition_layer_definition_general_proof} and we use index sets $I^k, I_A^k,$ etc., as in \eqref{eq:splitting_indexing}. (Here, all  notations include also the index $k$.)
         Importantly $v_\eps^{k}$ satisfies \eqref{eqn:correctMiddleLayer}, \eqref{eqn:vanishingTopLayer}, and \eqref{eqn:gradMatching}. Further, we may choose $\delta$ smaller, depending only on $v$, such that $\lbrace S_\delta^k\rbrace_{k=1}^{k_{0}}$  are pairwise disjoint. 
          Next, we define 
        \begin{align}\label{eq:def_rec_gradient_step3}
            M_\eps(x) = \begin{cases}
                \nabla v_\eps^k(x) & \text{ if } x \in F(\Omega) \cap S_\delta^k, \quad k=1,\ldots,k_0, \\
                \nabla v(x)           & \text{ if } x \in F(\Omega) \setminus \bigcup_{k=1}^{k_0} S_\delta^k.
            \end{cases}
        \end{align}
        By construction, this is curl-free. Therefore, since $\Omega$ is a simply connected Lipschitz set, there exists $v_\eps \in H^2(\Omega; \R^d)$ with $\nabla v_\eps = M_\eps$. Now, we observe that  
        \begin{align}\label{eq:rec_bound}
            \limsup_{\epsilon\to 0} E_\epsilon^F[v_\epsilon; F(\Omega)]  \leq & \sum_{k = 1}^{k_0} \left[ \sum_{i \in I_{A,B}^k \cup I_{B,A}^k}\mathcal{H}^{d-1}( \omega_i^k  )\left(\sigma_{\rm geo}^F(x_i^k) + \varsigma(2^{-  m }) + C\lambda \right)\right. \nonumber \\
            & \quad + C\left.\left( \sum_{i \in I_\partial^k}\H^{d-1}(  \omega_i^k  ) + \sum_{i \in I_A^k \cup I_B^k}\H^{d-1}( \omega_i^k)\varsigma(2^{- m }) \right) \right].
        \end{align} 
        Notice that
        \[
              \sum_{k = 1}^{k_0} \sum_{i \in I_\partial^k} \H^{d-1}( \omega_i^k  ) \to 0
        \]
        as $ m  \to \infty$ by \ref{H4}(iii). 
       Therefore, by choosing a diagonal sequence (without renaming), i.e.,  passing to the limits $m \to \infty$ and $\lambda \to 0$,  we infer  by the continuity of $\sigma_{\rm geo}^F$ (see Lemma \ref{lemma:relation_of_pushforward_of_geodesic_distance} and Lemma \ref{geodesic-distance-gradient-estimates} below) that
        \[
            \limsup_{\epsilon\to 0} E_\epsilon^F[v_\epsilon; F(\Omega)]  \leq   \int_{J_{\nabla v}} \sigma_{\rm geo}^F(x) \, d\H^{d-1}(x).
        \]
        Finally, we note that by construction  $\nabla v_\eps \to \nabla v$ in $L^2(F(\Omega);\R^{d\times d})$. Indeed, by \eqref{eq:rec_bound} and Theorem~\ref{mainthm:compactness} every subsequence of $\nabla v_\eps$ has a subsequence that converges to a gradient $\nabla w$ with $w \in X(\Omega)$. By our definition of $v_\eps$ in \eqref{eq:def_rec_gradient_step3} and since $\delta = \delta(m) \to 0$, we infer $\nabla w = \nabla v$. By the Urysohn subsequence principle the convergence of $\nabla v_\eps$ follows.   which by the Poincar\'e inequality, up to shifting $v_\eps$ by a constant, implies that $v_\eps \to v$ in $L^2(F(\Omega);\R^d).$ 
        As previously mentioned, this concludes the proof by Lemma \ref{lem:equivalence_of_gamma_limits_under_diffeomorphism}. 
    \end{proof}
    
  \section{Geodesics with bounded path lengths and consequences}
\label{sec:geodesics-disc}

  The main result of this section (Theorem \ref{existence-of-geodesics-with-bounded-path-length}) establishes the existence of geodesics with bounded path length. This result allows us to deduce four consequences which have been used crucially in the proof of  the  $\Gamma$-convergence for $E_\eps$: 
  \begin{enumerate}
        \item gradient bounds on the geodesic energy in terms of the endpoint (Lemma \ref{geodesic-distance-gradient-estimates-new}),
        \item Lipschitz continuity of $d_{\WtruncLI}$ with respect to paths $\psi$ (see Proposition \ref{comparision-geodesic-distances}),
        \item the fact that $d_W$ remains unchanged under truncation of the potential $W$ (Lemma \ref{lemma:trunction_does_not_matter}), and
        \item uniform bounds on the path length of a Modica--Mortola type reparameterization of nearly-optimal paths in terms of the end points (Lemma \ref{lemma:rescaling_of_nearly_geodesics}).
  \end{enumerate}

We formulate the results of this section in a slightly more general setting. We define the potential $\Wtrunccc \colon \overline{\Omega} \times \R^{d\times d} \to [0,\infty)$ by 
\begin{align}\label{frakdef}
\Wtrunccc(x,M) = \min \lbrace W(x,M), \alpha_* (|M|^{2\beta_*} + 1) \rbrace \quad \text{for } (x,M) \in \overline{\Omega} \times \R^{d \times d}, 
\end{align}
where $\alpha_* \ge 1$ and $\beta_* \ge 0$.  
In particular, for $\alpha_*$ large enough and $\beta_* =1$, we get $\Wtrunccc = W$ by \ref{H2}. On the other hand, for $\alpha_* \equiv \frac{C_0}{2}$ and $\beta_*=0$, we recover the potential $\WtruncLI$ introduced in \eqref{Wtrunc-neu}. 

\subsection{Statement of the main result and four consequences} 
We now formulate the main result of this section on the bounded length of nearly-optimal paths. This result is in the spirit of similar statements found in \cite{CristoferiDeutschPignatelli26, CrisFonGa23, CristoferiGravina2021}, although the exact structure of the geodesic distance function differs here. 
For convenience, we define the set of paths between two matrices $M, N \in \R^{d \times d}$ by  
\begin{align}\label{paths}
    {\rm Paths}(M, N) := \big\{\Phi \in W^{1,1}((-1, 1); \R^{d\times d}): \Phi(-1) = M, \, \Phi(1) = N \big\}.
\end{align}
and the set of admissible curves by
\begin{align}\label{admissable-control}
    {\rm Curv} := \big\{\psi \in W^{1,1}((-1, 1); \Omega)\colon \,  {\rm osc}(\psi) < \delta_0 \}
\end{align}
where  ${\rm osc}(\psi)$ denotes the oscillation of $\psi$ and $\delta_0$ is defined as 
\begin{align}\label{eq:definition_delta0}
    \delta_0 = \frac{\min_{x \in \overline{\Omega} }|A(x) - B(x)|}{4\max\{{\rm Lip}(A) + {\rm Lip}(B), 1\}}
\end{align}
with ${\rm Lip}(A), {\rm Lip}(B)$ being the Lipschitz constants of the wells $A$ and $B$. 

\begin{theorem}\label{existence-of-geodesics-with-bounded-path-length} 
    Let $W$ satisfy \ref{H1}--\ref{H3}. Let $M,N \in \R^{d \times d}$ and   $\psi \in {\rm Curv}$.  Then, there exists a sequence $\{ \Phi_m \}_m \subset {\rm Paths}(M,N)$  with 
    \[
       \int_{-1}^1   \Wtruncc(\psi, \Phi_m)|\Phi_m'| \,dt \leq d_\Wtrunccc(\psi, M, N) + \frac{1}{2^m}
    \]
    such that
    \begin{equation}\label{eqn:geoLengthBound}
         \int_{-1}^1  |\Phi_m'| \,dt \leq  C \min\lbrace 1+ |M|^{1-\beta_*} + |N|^{1-\beta_*}, \alpha_*  \rbrace (1 + |M|^{\beta_* +1} + |N|^{\beta_* +1} )
    \end{equation}
    with a constant $C >0 $ depending only on $A$,  $B$, and $\|\psi'\|_{L^1(-1,1)}$.
\end{theorem}
 
The proof of Theorem \ref{existence-of-geodesics-with-bounded-path-length} is based on methods already used in \cite{CrisFonGa23, CristoferiGravina2021, ZunigaSternberg2016}. One covers the domain of a nearly-optimal curve with suitable intervals where on each interval its   length is controlled by the length (with respect to $W$) of a suitable competitor for the geodesic distance. We remark that the setup is of different nature compared to the aforementioned literature since the geodesic distance is \textit{not} invariant under reparameterization and it depends on the curve $\psi$.E

 Before we present the proof of Theorem \ref{existence-of-geodesics-with-bounded-path-length}, we discuss its applications, i.e., we give the proofs of 
Proposition \ref{comparision-geodesic-distances},  Lemma \ref{lemma:trunction_does_not_matter}, Lemma \ref{lemma:rescaling_of_nearly_geodesics},  and Lemma \ref{geodesic-distance-gradient-estimates-new}.

\begin{proof}[Proof of Proposition \ref{comparision-geodesic-distances}]
    We first note  that by \ref{H3} the  function $\sqrt{\WtruncLI}$  is globally Lipschitz in $x$ with Lipschitz constant $C_{\rm Lip}  = C_{\rm Lip}(C_0)>0$ . Consider   $\Phi \in W^{1,1}( (-1,1);  \R^{d \times d})$ and observe
    \[
        \int_{-1}^{1} \sqrt{\WtruncLI}(\psi_1, \Phi)|\Phi'| \,ds \leq \int_{-1}^{1} \sqrt{\WtruncLI}( \psi_2, \Phi)|\Phi'| \,ds +   C_{\rm Lip}  \norm{\psi_1 -  \psi_2}_{\infty}\int_{-1}^{1} |\Phi'| \,ds.
    \]
      Now by taking $\Phi$ as a minimizing sequence for $d_{\WtruncLI}( \psi_2, M, N)$ provided in Theorem \ref{existence-of-geodesics-with-bounded-path-length} (for $\alpha_*  = \frac{C_0}{2}$ and  $\beta_*=0$)   we conclude.  
 \end{proof}
\begin{proof}[Proof of Lemma \ref{lemma:trunction_does_not_matter}]
 First, by definition of $V_0$  it clearly  holds  
\[
    d_{W}(x, A(x), B(x)) \geq d_{\WtruncLI}(x, A(x), B(x))  
\]
for all $x \in \overline \Omega$.   For the other inequality, by choosing $\alpha_* = \frac{C_0}{2}$ and $\beta_* =0$, $\Wtrunccc$ defined in \eqref{frakdef} coincides with $\WtruncLI$. With these parameters chosen, we use Theorem \ref{existence-of-geodesics-with-bounded-path-length} for every $x\in \overline{\Omega}$ to approximate $ d_{\WtruncLI}  (x, A(x), B(x))$ with a sequence of paths $\lbrace \Phi_m  \rbrace_m  \subset {\rm Paths}(A(x), B(x))$  satisfying
    \[
        \int_{-1}^1 |\Phi_m'| \,dt \leq  C
    \]
    where $  C > 0$ depends on $A$ and $B$, but not on  $  C_0 =2\alpha_*$. In particular, there exists  $R>0$ independent of $C_0$ such that all paths $\Phi_m$ are contained in   $B_R(0)$. Now, choosing $C_0$ large enough such that  $  C_0 > \sup_{(x,M) \in \overline{\Omega} \times B_R(0)}W (x,M)$ yields $  \WtruncLI (x, \Phi_m) = W(x, \Phi_m)$ and, hence, proves the claim.
\end{proof}

\begin{proof}[Proof of Lemma \ref{lemma:rescaling_of_nearly_geodesics}]
    The fact that nearly-optimal paths $\Phi$ can be chosen to satisfy $\Phi(t) = A(y_0)$ and $\Phi(-t) = B(y_0)$ for all $t \in (\eps  K_\lambda,   1)$  follows by a  Modica--Mortola reparameterization (cf.\ \cite[Lemma 4.5]{CristoferiGravina2021}). The bound on the length \eqref{boundonthelength} is a consequence of $A,B \in L^\infty(\Omega;\R^{d \times d})$ and   Theorem \ref{existence-of-geodesics-with-bounded-path-length} above, applied for $\alpha_*$ large enough and $\beta_* =1$ in \eqref{frakdef}, which guarantees $\Wtrunccc = W$. 
\end{proof}

Concerning Lemma \ref{geodesic-distance-gradient-estimates-new}, we prove the following more general version. Clearly, 
Lemma \ref{geodesic-distance-gradient-estimates-new} is an immediate corollary of it  (by choosing $\beta_*=0$ and a prototypical $W$, see before \eqref{eqn:dampened osc}).  

\begin{lemma}\label{geodesic-distance-gradient-estimates}
    Under the same assumptions of Theorem \ref{existence-of-geodesics-with-bounded-path-length}, we have that
    $f(x,M) := d_\Wtrunccc(x, A(x), M)$ 
    is locally Lipschitz in $\overline{\Omega}\times \R^{d\times d}$ and satisfies the gradient estimates 
    \begin{equation}\label{eqn:gradxbound0}
        |\nabla_M f(x,M)|  \leq  \Wtruncc (x, M)
    \end{equation}
    for almost all $(x,M) \in \overline{\Omega}\times \R^{d \times d}$  and, given $R>0$,      \begin{equation}\label{eqn:gradxbound}
       |\nabla_x f(x,M)|  \leq  C^R_{\rm Lip}
    \end{equation}
    for almost all $(x,M) \in \overline{\Omega}\times B_R(0)$,
    where  $ C^R_{\rm Lip}$ depends on $\alpha_*$, $\beta_*$, $W$, $A$, and $R$. In case $\beta_* = 0$, we have $C^R_{\rm Lip} = \bar{C}(1+R)$ for a constant $\bar{C}>0$ depending only on $\alpha_*$, $W$, and $A$.      
\end{lemma}
  
\begin{proof}
We start with the Lipschitz property with respect to the  $x$ variable.    Let $M \in B_R(0)$ and $x, \tilde x \in {\Omega}$. Consider $\Phi \in {\rm Paths}(M, A(\tilde x))$ satisfying \eqref{eqn:geoLengthBound}. Further, define the straight line 
    \[
        \Theta(s) := sA(x) + (1-s)A(\tilde x) \quad \text{for } s \in [0,1],
    \]
  and 
    \[
        \Psi := \begin{cases}
            \Phi\big(2s+1   \big), & s \in (-1, 0], \\
            \Theta(s ), & s \in (0, 1).
        \end{cases}
    \]
    Notice that $\Psi \in {\rm Paths}(M, A(x))$. 
Defining
${\rho} :=  C\min\lbrace 1 +  R^{1-\beta_*} + \Vert A \Vert_\infty^{1-\beta_*}, \alpha_*  \rbrace (1 + R^{\beta_* +1} + \Vert A \Vert_\infty^{\beta_* +1} )$ we find  by \eqref{eqn:geoLengthBound} that $\Psi$ is contained in $B_\rho(0)$ for a sufficiently large universal constant $C$.
We observe that    by \ref{H3} the  function $\sqrt{\Wtrunccc}$  is  Lipschitz in $x$ on $\overline{\Omega} \times B_\rho(0)$ with Lipschitz constant $C_{\rm Lip}  = C_{\rm Lip}(\rho, \alpha_*, \beta_*)>0 $. Moreover, we denote by $C_{\rm max}$ the maximum of $\sqrt{\Wtrunccc}$  on $\overline{\Omega} \times \overline{B_\rho(0)}$. A simple substitution gives
    \begin{align*}
        &~ f(x, M) - \int^1_{-1}  \Wtruncc  (\tilde x, \Phi(s))|\Phi'(s)| \, ds  \\
        & \leq \int^1_{-1}   \Wtruncc  (x, \Psi(s))|\Psi'(s)| \, ds - \int^1_{-1}  \Wtruncc  (\tilde x, \Phi(s))|\Phi'(s)| \, ds \\
        & \leq \int^1_{-1} \big(  \Wtruncc  (x, \Phi(s)) -   \Wtruncc  (\tilde x, \Phi(s))\big)|\Phi'(s)| \, ds + \int^1_{0}   \Wtruncc  (x, \Theta(s))|\Theta'(s)| \, ds \\
        &\leq  C_{\rm Lip}  |x - \tilde x| \int^1_{-1} |\Phi'(s)| \, ds +  C_{\rm max}\,\text{Lip}(A)|x -\tilde x|.
    \end{align*}    
    As the above estimates apply to a minimizing sequence for $d_W(\tilde{x}, A(\tilde{x}), M)$ from Theorem \ref{existence-of-geodesics-with-bounded-path-length}, we find 
    \[
        f(x, M) - f(\tilde x, M) \leq   \big( C_{\rm Lip} \, \rho  + C_{\rm max} {\rm Lip}(A)  \big) |x - \tilde x|.
    \]
        We derive the Lipschitz continuity in the $x$-variable by exchanging the roles of $f(x, M)$ and $f(\tilde x, M)$. This also implies the gradient bound \eqref{eqn:gradxbound}. In the case $\beta_*  = 0$, $C_{\rm Lip}$ and $C_{\rm max}$ only depend on $\alpha_*$ and $W$, and ${\rho} $ depends on $\alpha_*$, $\Vert A\Vert_\infty$, and  $R$, with a linear dependence on $R$. 
    
    The Lipschitz continuity with respect to the variable $M$ follows by analogous arguments. 
    In particular, given $M, \tilde{M} \in \R^{d \times d}$, repeating the above computations   (though it is now just the reverse triangle inequality)  with 
    \[
        \Theta(s) := sM + (1-s)\tilde M \quad \text{for } s \in [0,1],
    \]
    we obtain
    \[
        \frac{f(x,M) - f(x, \tilde M)}{|M - \tilde M|} \leq \int^1_0  \Wtruncc  (x, \Theta(s)) \, ds,
    \]
    and thus the Lipschitz continuity.     At a point of differentiability $(x,M)$ of $f$,  we let $\tilde M \to M$ in the above equation, and thereby $\Theta \to M$ uniformly, to recover the  gradient estimate \eqref{eqn:gradxbound0}.   
\end{proof}
 
\subsection{Proof of Theorem \ref{existence-of-geodesics-with-bounded-path-length}}  
We now proceed towards the proof of Theorem \ref{existence-of-geodesics-with-bounded-path-length}. In order to obtain nearly-optimal paths with bounded path length, we potentially need to modify the path to a simpler path with controlled length. This procedure consists in replacing parts of the original paths.  As for this reason, we need to deal with time subintervals of $(-1,1)$, it is helpful to introduce the following notation: 
For $M, N \in \R^{d \times d}$ and   $-1 \leq t_1 \leq t_2 \leq 1$ we set 
    \[
   {\rm Paths}(M, N, t_1, t_2) := \big\{\Phi \in W^{1,1}((t_1, t_2); \R^{d\times d})\colon \,  \Phi(t_1) = M, \, \Phi(t_2) = N\big\}.
    \]
    For a curve $\psi \in W^{1,1}((-1,1); \overline{\Omega})$, we further define
    \[
        d_\Wtrunccc(\psi, M, N, t_1, t_2) := \inf\left\{ \int_{t_1}^{t_2} \Wtruncc (\psi(s), \Phi(s))|\Phi'(s)| \, dt\colon \,  \varphi \in {\rm Paths}(M, N, t_1, t_2) \right\}.
    \]
We note that for $t_1 = -1$ and $t_2 =1$ this is consistent with the notation introduced in  \eqref{eqn:geoGeneral} and \eqref{paths}. 
The next lemma shows that geodesics for $d_\Wtrunccc(\psi, M, N)$ are also locally optimal in every subinterval $[t_1,t_2]\subset [-1 ,1]$. We omit its proof as it follows from classical arguments by contradiction.
\begin{lemma}\label{geodesics-are-local-geodesics}
  Let   $\eps>0$, and $M, N \in \R^{d \times d}$ be fixed. Let $\psi \in W^{1,1}((-1,1); \overline{\Omega})$.   Suppose that $\Phi \in {\rm Paths}(M,N)$ is such that  \[
       \int_{-1}^1  \Wtruncc(\psi, \Phi)|\Phi'|\,dt \leq d_\Wtrunccc(\psi, M, N) + \epsilon.
    \]
    Then,  
    \[
        \int_{t_1}^{t_2} \Wtruncc(\psi, \Phi)|\Phi'|\,dt \leq d_\Wtrunccc(\psi, \Phi(t_1), \Phi(t_2), t_1, t_2) + \epsilon
    \]
    for all  $-1\leq t_1 < t_2 \leq 1$.
\end{lemma}

The next lemma tells us that, if we have a nearly-optimal paths, whose energy has low spatial oscillations, then the geodesic cannot jump between the wells $A$ and $B.$ Recall the definition of $\delta_0$ in \eqref{eq:definition_delta0}.  

\begin{lemma}\label{lemma:admissibility_lemma} 
    Let    $M, N \in \R^{d \times d}$ be fixed, and let $\psi \in {\rm Curv}$. Then, there exist $\eps_1 >0$ and $\delta_1 \in (0,\delta_0)$  such that for all $\eps \in (0,\eps_1)$, all $\delta \in (0,\delta_1)$, and all $\Phi \in {\rm Paths}(M,N)$ with 
    \begin{align}\label{neqeq} 
        \int_{-1}^{1} \Wtruncc(\psi, \Phi)|\Phi'|\,dt \le  d_\Wtrunccc(\psi, \Phi(-1), \Phi(1)) + \eps
    \end{align}
   we have  that  either
    \[
        \sup \big\{ t \in (-1,1) \colon \, |\Phi(t) - A(\psi(t))| < \delta \big\} < \inf\big\{ t \in (-1,1) \colon \,  |\Phi(t) - B(\psi(t))| < \delta \big\},
    \]
    or 
    \[
        \sup \big\{ t \in (-1,1) \colon \,  |\Phi(t) - B(\psi(t))| < \delta \big\} < \inf\big\{ t \in (-1,1) \colon \,  |\Phi(t) - A(\psi(t))| < \delta \big\}.
    \]
  \end{lemma}

\begin{proof}
     We argue by contradiction. Let $\eps_1 > 0$ and $\delta_1 \in (0, \delta_0)$ be fixed, chosen sufficiently small at the end of the proof. Let $\Phi \in {\rm Paths}(M,N)$ be such that  \eqref{neqeq} holds  for some $\eps \in (0, \eps_1)$.
     Suppose without restriction that there are $t_1 < t_2 < t_3$ such that for $\delta \in (0, \delta_1)$ 
    \begin{equation}\label{eqn:contraHyp}   
        \text{  $\Phi(t_1) \in B_\delta(A( \psi(t_1)))$, \quad  $\Phi(t_2) \in B_\delta(B(\psi(t_2)))$, \quad  and $\Phi(t_3)  \in B_\delta(A(\psi (t_3))$.}
    \end{equation}
    Then, by  Lemma \ref{geodesics-are-local-geodesics}
    \begin{align}\label{combini}
        \int_{t_1}^{t_3} \Wtruncc(\psi, \Phi)|\Phi'|\,dt \le  d_\Wtrunccc(\psi, \Phi(t_1), \Phi(t_3),t_1,t_3) + \eps.
    \end{align}
    We use a straight line as a competitor to estimate the geodesic distance. To this end, we define $ P_i :=   \Phi(t_i)  -  A(\psi(t_i))$ for $i=1,3$, and note that $|P_i| < \delta$. We   let  $\Theta$ be the straight line from $P_1$ to $P_3$ corrected by $A(\psi)$, i.e.,
    \[
        \Theta(t) := \frac{t_3 - t}{t_3 - t_1}P_1 + \frac{t - t_1}{t_3 - t_1}P_3 + A(\psi(t)) \quad \text{for } t\in [t_1,t_3].
    \]
     Notice that, due to the quadratic growth in \ref{H2},   we derive
    \begin{equation}\nonumber
            \int_{t_1}^{t_3} { \Wtruncc(\psi, \Theta)}|\Theta'| \,dt \leq  C\delta \left( (|P_3 - P_1| + \int_{t_1}^{t_3}|(A \circ \psi)'| \,dt\right)  \leq C\delta_1,
    \end{equation}
where  we simply use $\delta<\delta_1$ and  the constant depends on ${\rm Lip}(A)$ and $\psi$. Combining with \eqref{combini}
   yields  
   \begin{equation}\label{eqn:competit1}
        \int_{t_1}^{t_3} \Wtruncc(\psi, \Phi)|\Phi'|\,dt \le  C\delta_1   + \eps_1.
   \end{equation}  
   On the other hand, since $\psi \in {\rm Curv}$, see \eqref{admissable-control},  and $\delta< \delta_0$ (cf.\ \eqref{eq:definition_delta0})  we have
   \[
        |A(\psi(t)) - B(\psi(s))| > \frac{3}{4}\min_{x \in \overline \Omega}|A(x) - B(x)|
   \]
   for all $s,t \in [-1,1]$. Consequently, defining 
   \[
        c_0 := \min \big\{ W(x,M)\colon x \in \Omega, M \in  \big(B_{\delta_0}(A(x)) \cup B_{\delta_0}(B(x))\big)^c \big\} > 0,  
    \]
    we use  $\delta< \delta_0$  to find 
    \begin{align*}
        \int_{t_1}^{t_3} \Wtruncc(\psi, \Phi)|\Phi'|\,dt &= \int_{t_1}^{t_2} \Wtruncc(\psi, \Phi)|\Phi'|\,dt + \int_{t_2}^{t_3} \Wtruncc(\psi, \Phi)|\Phi'|\,dt \\
        &\geq 2c_0\dist\Big(B_{\delta_0}\big(A( \psi([-1,1]))\big), B_{\delta_0}\big(B(\psi ([-1,1])) \big) \Big) \\
        &\geq c_0\left( \frac{3}{4}\min_{x \in \overline\Omega}|A(x) - B(x)| - 2\delta_0 \right) \\
        &\geq \frac{c_0}{4}\min_{x \in \overline \Omega}|A(x) - B(x)|,
    \end{align*}
    where in the first inequality we used that \eqref{eqn:contraHyp} implies that on $[t_1,t_2]$ the curve $\Phi$ travels (at least) from $\partial B_{\delta_0}(A(\psi(t_1)))$ to $\partial B_{\delta_0}(B(\psi(t_2)))$ (likewise on $[t_2,t_3]$).
    Comparing this with \eqref{eqn:competit1}, we find that $\frac{c_0}{4}\min_{x \in  \overline\Omega}|A(x) - B(x)| \leq C\delta_1   + \eps_1.$ 
    Thus, if $\delta_1 > 0$ and $\eps_1 > 0$ are chosen sufficiently small, we obtain a contradiction, with which we conclude.
\end{proof}
As a final preparation for the proof of Theorem \ref{existence-of-geodesics-with-bounded-path-length}, let us choose  a  parameter $\delta \in (0,  \delta_1 )$, with $\delta_1$ defined in Lemma \ref{lemma:admissibility_lemma}, small enough such that we have
\begin{align}\label{Ma1}
\delta \le \frac{1}{4}\min_{x\in \overline{\Omega}}\textrm{dist}\,(A(x),B(x)),
\end{align}
as well as
\begin{align}\label{Ma2}
\Wtrunccc(x,M) = W(x,M) \quad  \text{for all $x\in \Omega$ and $M \in B_{\delta}(A(x))\cup B_{\delta}(B(x))$   }
\end{align}
see \ref{H2} and the definition of $\Wtrunccc$ in \eqref{frakdef}.

\begin{proof}[Proof of Theorem \ref{existence-of-geodesics-with-bounded-path-length}]
    Let $\{ \tilde \Phi_m \}_m$ be a sequence such that
  \begin{align}\label{tildephi}
       \int_{-1}^1 \Wtruncc(\psi, \tilde \Phi_m)|\tilde \Phi_m'| \,dt \leq d_\Wtrunccc(x, M, N) + \frac{1}{2^{2m}} .
    \end{align}
    Throughout  the proof, $m$ will be fixed. We will first modify  $\tilde \Phi_m$, and then we estimate the geodesic and Euclidian length of the resulting path. As a preparation, for $L \in \lbrace A,B\rbrace$ and  $k \in \N_{\rm 0}$, we define the \emph{$k$-th annulus centered at $L$} as  
 
\[
        \mathcal{A}^{L,k}  := \left\{ (x, M) \in \Omega \times \R^{d \times d}: \frac{\delta}{2^{k+1}}   \leq |M - L(x)| <  \frac{\delta}{2^{k}}   \right\},
    \]
    where $\delta$ is fixed in \eqref{Ma1}--\eqref{Ma2}.  This ensures $\mathcal{A}^{A,k} \cap \mathcal{A}^{B,k'} = \emptyset$ for all $k,k'\in \N_0 $  by \eqref{Ma1}.   Now, for   $L \in \{A,B\}$   we set 
    \[
        t_0^L := \inf \, (\psi, \tilde \Phi_m)^{-1} \left(\bigcup\nolimits_{k \geq m}   \mathcal{A}^{L,k}   \right) \quad \text{ and } \quad
        t_1^L := \sup \, (\psi, \tilde \Phi_m)^{-1} \left(\bigcup\nolimits_{k \geq m}   \mathcal{A}^{L,k}  \right).
    \]
    Here, $t_0^L$ and $t_1^L$ are the first and the last time, respectively, such that  the pair $(\psi, \tilde \Phi_m)$ satisfies $|\tilde \Phi_m(t) - L(\psi(t))|\leq 2^{-m} \delta $. 
    For convenience, we can assume that these values are well-defined. (If $(\psi, \tilde \Phi_m)$ never takes values in $  \mathcal{A}^{L,k}$, $k \ge m$, the following analysis can be simplified.)  By Lemma \ref{lemma:admissibility_lemma} we get that  $[t_0^A,t_1^A]$ and $[t_0^B,t_1^B]$ are disjoint intervals. Without loss of generality, we may assume $ t_1^A  \leq t_0^B$. 
    
    The idea is now to first substitute $\tilde \Phi_m$ with a (perturbed)  `straight path' in  $\bigcup_{k \geq m}   \mathcal{A}^{L,k}$, $L \in \lbrace A,B \rbrace$, whereby increasing the Euclidean path length and  the $\Wtrunccc$-geodesic-length only by a small amount (Step~1). Then, on each $   \mathcal{A}^{L,k}  $, $k<m,$ we control the length of $ \tilde \Phi_m$ in this annulus by considering a suitable competitor  (Step~2). Lastly, we estimate the path length of the remaining path outside of all annuli (Step~3).  
     
    \emph{Step 1 (Straight paths modifications close the wells).} Consider the curve 
\[
        \Phi_m := \begin{cases}
            \Theta_{A,m} & \text{ in } [t_0^A,t_1^A], \\
            \Theta_{B,m} & \text{ in } [t_0^B,t_1^B], \\
            \tilde \Phi_m & \text{ otherwise, }
        \end{cases}
    \]
    where $\Theta_{A,m}$ and $\Theta_{B,m}$ are defined as follows:
    Let   $ P_i := \tilde \Phi_m(t_i^A)  -  A(\psi(t^A_i))$ for $i=0,1$, and note that $|P_i|\leq   2^{-m} \delta$ by  the definition of $t_i^A$. Similarly as in the proof of Lemma \ref{lemma:admissibility_lemma}, we define $\Theta_{A,m}$ as the straight line from $P_0$ to $P_1$ corrected by $A(\psi)$, i.e.,
    \begin{align}\label{similarcon}
        \Theta_{A,m}(t) := \frac{t_1^A - t}{t_1^A - t_0^A}P_0 + \frac{t - t_0^A}{t_1^A - t_0^A}P_1 + A(\psi(t)).
    \end{align}
    We define    $\Theta_{B,m}$ analogously. 
    Notice that, due to the quadratic growth in \ref{H2}, we derive
    \begin{align} \label{similarargument}
        \int_{t_0^A}^{t_1^A} { \Wtruncc(\psi, \Theta_{A,m})}|\Theta_{A,m}'| \,dt &\leq \frac{C\delta}{ 2^{m}}\left( (|P_0 - P_1| + \int_{t_0^A}^{t_1^A}|(A \circ \psi)'| \,dt\right) \notag \\
        &\leq \frac{C}{2^{m}} \Big(1 + \max\{\text{Lip}(A) \} \int_{t_0^A}^{t_1^A}  |\psi'| \, dt\Big),
    \end{align}
    where we absorbed $\delta$ in the constant $C$.
  An analogous estimate holds for $\Theta_{B,m}$. This shows that the above substitution only increases the geodesic length by a small amount. 
    In a similar fashion, we get a  bound on the Euclidean path length of $\Theta_{A,m}$ and $\Theta_{B,m}$, namely  
    \begin{equation}\label{eqn:thetaBound}
        \int_{t_0^A}^{t_1^A}|\Theta_{A,m}'| \,dt  +  \int_{t_0^B}^{t_1^B}|\Theta_{B,m}'| \,dt \leq C\Big(1 + \text{Lip}(A),  \int_{t_0^A}^{t_1^A}  |\psi'| \, dt + \text{Lip}(B) \int_{t_0^B}^{t_1^B}  |\psi'| \, dt   \Big) .
    \end{equation}

    \emph{Step 2 (Path length in annuli $\mathcal{A}^{A,k}$ for $k<m$).} 
    We now want to estimate the path length of $\Phi_m$ in the interval $[-1, t^A_0]$. Note that on this set it holds  $\Phi_m = \tilde \Phi_m$. Without loss of generality we can assume $t^A_0 > -1$ as otherwise the path length in $[-1, t^A_0]$ is clearly zero.
    We proceed by inductively defining a decreasing sequence $s_1> s_2 > \cdots > s_{k_0}$ in the following way: Set
    \[
        s_0 := t^A_0,
    \]
   and
    \[
        s_{k} := 
            \inf \, \Big[ [-1, s_{k - 1}] \cap (\psi, \Phi_m)^{-1}(\mathcal{A}^{A, m - k}) \Big]
    \]
   for every $k\in \mathbb{N}$ such that the above number is  not $+\infty$. In   particular, among the indices $k\leq m$ we denote the largest of these indices by $k_0$ (which makes $s_{k_0}$ as small as possible). We stress that  $k_0 < m$ is possible. In this case, we  have $s_{k_0} = -1$. Indeed, if   $s_{k_0} > -1$,  we would have $[-1, s_{k_0}] \cap (\psi, \Phi_m)^{-1}(\mathcal{A}^{A, m - (k_0+1)}) \neq \emptyset$
   by the continuity of $(\psi, \Phi_m)$.

    Similarly as before, we define  $  P^k_0 =   \Phi_m(s_{k+1}) - A(\psi(s_{k+1}))$ and $ P^k_1 =  \Phi_m(s_{k}) -  A(\psi(s_{k}))$. Note that $|P^k_i| \le  2^{-m+k+1}$ for $i=0,1$. Let $ \Gamma_k$ be the straight line connecting $P^k_0$ and $P^k_1$ as in \eqref{similarargument}.  Define 
    \begin{equation}\label{eqn:annPaths}
         \Psi_m := \Gamma_k + A \circ \psi \quad  \text{ on } \quad [s_{k+1}, s_k] \ \ \ \text{for } 0 \le k \le k_0 -1,
    \end{equation}
    see \eqref{similarcon} for a similar construction. 
    We use $\Psi_m$ as a competitor to estimate the geodesic length and Euclidean length of $\Phi_m$. Notice   that by the definition of $s_k$ and $s_{k+1}$ we have 
    \[
        |\Phi_m  - A(\psi(\cdot))| \ge \frac{\delta}{ 2^{m-k-1}}  \ \ \text{ on } [s_{k + 1}, s_{k}] \ \ \ \text{for } 0 \le k \le k_0 -1. 
    \]
    By \ref{H2} and \eqref{Ma2} this yields 
    \begin{equation}\label{eqn:annuliCoercive}
        \Wtruncc(\psi,  \Phi_m )|_{[s_{k + 1}, s_{k}]} \geq \frac{\delta}{\sqrt{C_*} 2^{m - k - 1}}.
    \end{equation}
    In view of Lemma \ref{geodesics-are-local-geodesics} and \eqref{tildephi},  for every $0 \leq k \le k_0 - 1$ we have that
    \[
        \int_{s_{k + 1}}^{s_{k}} \Wtruncc(\psi,  \Phi_m  )| \Phi_m ' | \,dt \leq d_\Wtrunccc\left(\psi,  \Phi_m  (s_{k + 1}),  \Phi_m  (s_{k}), s_{k + 1}, s_{k }\right) + \frac{1}{2^{2m}}.
    \]
    Applying \eqref{eqn:annuliCoercive}, we estimate
    \begin{align*}
        \frac{\delta}{ \sqrt{C_*} 2^{m - k - 1}} \int_{s_{k + 1}}^{s_{k}}| \Phi_m ' |\, dt 
        &\leq \int_{s_{k + 1}}^{s_{k}} \Wtruncc(\psi,  \Phi_m  )| \Phi_m ' |\, dt \\
        &\leq d_\Wtrunccc \left(\psi,  \Phi_m  (s_{k + 1}),  \Phi_m  (s_{k}), s_{k + 1}, s_{k }\right) + \frac{1}{2^{2m}} \\
        &\leq \int_{s_{k + 1}}^{s_{k}}\Wtruncc(\psi,   \Psi_m  )|  \Psi_m' |\, dt + \frac{1}{2^{2(m - k)}},
    \end{align*}
    where in the last line we used $\Psi_m$ as a competitor in the definition of geodesic distance and $2^{-2m} \leq 2^{-2(m - k)}$ for later convenience.
    In view of \ref{H2}, \eqref{eqn:annPaths}, and the fact that   $|P^k_i| \le  2^{-m+k+1}$ for $i=0,1$  we also find 
    \begin{align*}
        \int_{s_{k + 1}}^{s_{k}} \Wtruncc(\psi,   \Psi_m  )|  \Psi_m ' |\, dt & \leq  \sqrt{C_*}  \int_{s_{k + 1}}^{s_{k}}| \Psi_m  -A(\psi)|| \Psi_m ' |\, dt   \leq   \sqrt{C_*}   \frac{\delta}{2^{m - k + 1}}\int_{s_{k + 1}}^{s_{k}}|  \Psi_m ' |\, dt \\
        &\leq \frac{C}{2^{m-k}}\left(\frac{1}{2^{m - k}} +    \int_{s_{k + 1}}^{s_{k}}   |\psi'| \, dt \right),
    \end{align*}
    where $C$  denotes a generic constant, also depending on $\text{Lip}(A)$ and $\delta$, see also \eqref{similarargument} for a similar calculation.  
    By combining the above inequalities we infer 
    \[
        \int_{s_{k+1}}^{s_k}| \Phi_m' |\, dt \leq C\left(\frac{1}{2^{m - k}} +  \int_{s_{k + 1}}^{s_{k}}  |\psi'| \, dt \right).
    \]
    Eventually, summing over the intervals $[s_{k+1},s_k]$ for $0 \le k \le k_0-1$, recalling $k_0 \le m$,  we obtain the estimate
    \begin{align}\label{gather1}
        \int_{s_{k_0}}^{s_{0}} | \Phi_m' | \, dt &\leq C\Big(1 + \int_{s_{k_0}}^{s_{0}} |\psi'| \, dt  \Big).
    \end{align}

   \emph{Step 3 (Estimating the path length of $\Phi_m$ on the remaining intervals).} Let us now estimate the path length of $\Phi_m$ on  ${[-1, s_{k_0}]}$.    As already observed, if $k_0 < m$, then $s_{k_0} = -1$ and there is nothing to prove. We thus consider   $k_0  =  m$ and $s_{m}>0$.   By repeating the argument in \eqref   {eqn:annuliCoercive}
    we get $\Wtruncc(\psi,  \Phi_m )|_{[0, s_{m}]} \geq c_0$ for some constant $c_0$ depending on $\delta$ and $C_*$.  Thus, we get  by Lemma \ref{geodesics-are-local-geodesics} and \eqref{tildephi} 
    \begin{align*}
        c_0\int_{-1}^{s_{m}} | \Phi_m ' | \, dt &\leq \int_{-1}^{s_m} \Wtruncc(\psi,  \Phi_m )| \Phi_m ' |\, dt  \leq d_\Wtrunccc \left(\psi,  \Phi_m  (-1),  \Phi_m  (s_{m}), -1, s_{m}\right) + \frac{1}{2^{2m}}.  
    \end{align*}
We observe that $\Phi_m  (-1) = M$ and that by construction $|\Phi_m  (s_m)| \le \delta + \Vert A \Vert_\infty$. Using the straight line between $\Phi_m  (-1)$ and $\Phi_m   (s_m)$ as competitor we get, by \eqref{frakdef} and \ref{H2},
\begin{align*}
        c_0 \int_{-1}^{s_{m}}  | \Phi_m ' | \, dt  
        &\leq d_\Wtrunccc \left(\psi,  \Phi_m  (-1),  \Phi_m (s_{m}), -1, s_{m}\right) + \frac{1}{2^{2m}} \\ & \le C\big( 1 + \min\lbrace |M|, \sqrt{\alpha_*} |M|^{\beta_*}\rbrace \big) |M -  \Phi_m  (s_m)| + \frac{1}{2^{2m}} \\
        & \le  C\big( 1 + \min\lbrace |M|,  {\alpha_*} |M|^{\beta_*}\rbrace \big) (1+ |M|)   \le C \min\lbrace  1+|M|^{1-\beta_*}, \alpha_*  \rbrace (1 + |M|^{\beta_* +1}),
    \end{align*}
where $C>0$  also depends on $\delta$ and     $\Vert A \Vert_\infty$.   
 Combining this with \eqref{eqn:thetaBound} and \eqref{gather1}, we   derive 
    \[
        \int_{-1}^{t_1^A} | \Phi_m'| \, dt \leq C\Big(1 + \int_{-1}^{1}|\psi'| \, dt  \Big) + C \min\lbrace 1 +  |M|^{1-\beta_*}, \alpha_*  \rbrace (1 + |M|^{\beta_* +1}).
    \]
    In a completely analogous way, one can estimate the remaining path length on $[t_1^B,1]$ and $[t_1^A,t_0^B]$.  With this, we conclude \eqref {eqn:geoLengthBound}, where the constant $C$ also depends on $\|\psi'\|_{L^1(-1,1)}$.
\end{proof}

\section*{Acknowledgements}
The research of E.D.\ was supported by the Austrian Science Fund (FWF) through the projects  \href{https://www.doi.org/10.55776/Y1292}{10.55776/Y1292}, \href{https://www.doi.org/10.55776/P35359}{10.55776/P35359}, and \href{https://www.doi.org/10.55776/F100800}{10.55776/F100800}.  
K.S. was supported by funding from the NSF (USA) RTG grant DMS-2136198.
For open-access purposes, the authors have applied a CC BY public copyright
license to any author-accepted manuscript version arising from this
submission.

\bibliographystyle{abbrv} 
\bibliography{references.bib}

\end{document}